\documentclass[a4paper, 12pt, twoside, notitlepage,leqno]{amsart}
\usepackage{enumerate}
\usepackage{fullpage}
\usepackage{csquotes}
\usepackage{mathrsfs}  
\usepackage{enumitem}
\usepackage{orcidlink}

\usepackage[margin=2cm]{geometry}

\newtheorem{theorem}{Theorem}[section]
\newtheorem{lemma}[theorem]{Lemma}

\theoremstyle{definition}
\newtheorem{definition}[theorem]{Definition}

\theoremstyle{remark}
\newtheorem{remark}[theorem]{Remark}

\numberwithin{equation}{section}

\newcommand{\abs}[1]{\lvert#1\rvert}

\newcommand\rwh[1]{\mathscr{F}{#1}}
\newcommand\wh[1]{\mathscr{F}{#1}}

\newcommand{\R}{\mathbb{R}}
\newcommand{\D}{\mathrm{d}}

\newcommand{\Rb}{\mathbb{R}}

\newcommand{\Cb}{\mathbb{C}}

\newcommand{\RT}{\mathbf{R}}

\newcommand{\I}{\mathrm{i}}

\newcommand{\lb}{\left(}

\newcommand{\rb}{\right)}
\newcommand{\PD}{\partial}

\newcommand{\Beq}{\begin{equation}}
\newcommand{\Eeq}{\end{equation}}
\newcommand{\beq}{\begin{equation*}}
	\newcommand{\eeq}{\end{equation*}}
\newcommand{\bal}{\begin{align}}
	\newcommand{\eal}{\end{align}}

\newcommand{\n}{\nabla}

\usepackage{mathtools}

\newcommand{\bp}{\begin{prob}}
	\newcommand{\ep}{\end{prob}}
\newcommand{\bpr}{\begin{proof}}
	\newcommand{\epr}{\end{proof}}

\newcommand{\bel}[1]{\begin{equation}\label{#1}}
\newcommand{\ee}{\end{equation}}

\newcommand{\Ac}{\mathcal{A}}

\newcommand{\Sc}{\mathcal{S}}

\newcommand{\Sb}{\mathbb{S}}

\newcommand{\Z}{{\mathbb Z}}

\newcommand{\vev}[1]{\left\langle#1\right\rangle}

\newcommand{\norm}[1]{\lVert #1 \rVert}

\DeclareMathOperator{\real}{Re}
\usepackage{scalerel,stackengine}
\stackMath
\newcommand\reallywidehat[1]{%
\savestack{\tmpbox}{\stretchto{%
  \scaleto{%
    \scalerel*[\widthof{\ensuremath{#1}}]{\kern-.6pt\bigwedge\kern-.6pt}%
    {\rule[-\textheight/2]{1ex}{\textheight}}
  }{\textheight}%
}{0.5ex}}%
\stackon[1pt]{#1}{\tmpbox}%
}
\begin{document}

\title[Weighted Divergent Beam Ray Transform]{Weighted Divergent Beam Ray Transform: Reconstruction, Unique continuation and Stability}

\author[S. R. Jathar]{Shubham R. Jathar\,\orcidlink{0000-0002-9109-7543}}

\address{Computational Engineering, School of Engineering Sciences,
Lappeenranta-Lahti University of Technology LUT, Lappeenranta, Finland}

\email{shubham.jathar@lut.fi}

\author[M. Kar]{Manas Kar\,\orcidlink{0000-0001-6036-1535}}
\address{Department of Mathematics, Indian Institute of Science Education and Research (IISER) Bhopal, India}
\email{manas@iiserb.ac.in}
\author[V. P. Krishnan]{Venkateswaran P. Krishnan\,\orcidlink{0000-0002-3430-0920}}
\address{Centre for Applicable Mathematics, Tata Institute of Fundamental Research, India}
\email{vkrishnan@tifrbng.res.in}
\author[R. R. Pattar]{Rahul Raju Pattar\,\orcidlink{0000-0003-4601-9842}}
\address{Centre for Applicable Mathematics, Tata Institute of Fundamental Research, India}
\email{rahulrajupattar@gmail.com}




\begin{abstract}
In this article, we establish that any symmetric $m$-tensor field can be recovered pointwise from partial data of the $k$-th weighted divergent ray transform for any $k \in \mathbb{Z}^{+} \cup\{0\}$. Using the unique continuation property of the fractional Laplacian, we further prove the unique continuation of the fractional divergent beam ray transform for both vector fields and symmetric 2-tensor fields. Additionally, we derive explicit reconstruction formulas and stability results for vector fields and symmetric 2-tensor fields in terms of fractional divergent beam ray transform data. Finally, we conclude by proving a unique continuation result for the divergent beam ray transform for functions.
\end{abstract}

\maketitle
\section{Introduction}\label{sec:Intro}
Let $n\ge 2$ be a positive integer. For a given point \( x \in \mathbb{R}^n \) (referred to as the source point) and \( \xi \in\mathbb S^{n-1} \) (referred to as the direction), the \textbf{divergent beam ray transform} is defined as the integral of a function \( f \) along a ray that originates at \( x \) and emanates in the direction \( \xi \). This transform is mathematically expressed as:
\[
Df(x,\xi) = \int_0^{\infty} f(x + t\xi) \, dt.
\]
The uniqueness properties of the divergent beam ray transform have been investigated in previous works, most notably by Hamaker et. al.~\cite{Hamaker:Smith:Solomon:Wagner:1980}, who established that if \( Df = Dg \) for every source point \( x \) and direction \( \xi \), then it necessarily follows that \( f = g \). Furthermore, partial uniqueness results under additional conditions were provided in \cite[Theorem 5.6]{Hamaker:Smith:Solomon:Wagner:1980}.
The relationship between the divergent beam ray transform and the Radon transform was explored in \cite[Section 4]{Hamaker:Smith:Solomon:Wagner:1980}. Additionally, Finch and Solmon characterized the range of the divergent ray transform, particularly when the source points lie on a sphere, in \cite[Theorem 2.4]{Finch:Solmon:1983}, analogous to the range characterization of the Radon transform in the case \( n = 2 \) by Helgason \cite{Helgason:1965} and Ludwig \cite{Ludwig:1966}. 

\noindent Extending this concept, Kuchment and Terzioglu introduced the \textbf{\( k \)-th weighted divergent beam ray transform} for \( k > -1 \) \cite[Def. 2.1]{Kuchment:Terzioglu:2017}. For a function \( f \in \mathcal{S}(\mathbb{R}^n) \) and a source-direction pair \( (x, \xi) \), this transform is defined as:
\[
D^k f(x,\xi) = \int_0^{\infty} t^k f(x + t\xi) \, dt.
\]
The inversion formula for recovering the function \( f \) from \( D^k f \) has been proven by Kuchment and Terzioglu \cite{Kuchment:Terzioglu:2017}.

\noindent This notion of the weighted divergent ray transform can be generalized to accommodate Schwartz class symmetric \( m \)-tensor fields. This is formalized in the following definition:
\begin{definition}\label{def:k_weight_DBT}
    For a non-negative integer \( k \), the \textbf{\( k \)-th weighted divergent beam ray transform} of a symmetric \( m \)-tensor \( f \in \mathcal{S}\left(\mathbb{R}^n; S^m(\mathbb{R}^n)\right) \) is defined by
    \[
    D^{k,m} f(x, \xi) = D_x^{k,m} f(\xi) := \int_0^{\infty} t^k f_{i_1 \cdots i_m}(x + t\xi) \xi^{i_1} \cdots \xi^{i_m} \, dt,
    \]
    where \( x \in \mathbb{R}^n \) is the source point and \( \xi \in \mathbb{S}^{n-1} \) is the unit vector indicating the direction of the beam.
\end{definition}
One of the primary objectives of this paper is to present a reconstruction formula for the tensor field \( f \) using the information derived from \( D^{k,m} f \). It is important to emphasize that for the recovery of \( f \), full data is not necessary; rather, only restricted data is required.
\begin{theorem}\label{tm:recon:dbrt}
    Let \( f \in \mathcal{S}\left(\mathbb{R}^n ; S^m(\mathbb{R}^n)\right) \) be a symmetric \( m \)-tensor field in \( \mathbb{R}^n \) and \( D^{k,m} f \) be its \( k \)-weighted divergent beam ray transform for $k\in \mathbb Z^+\cup\{0\}$. Then \( f(x) \) can be recovered pointwise from the knowledge of \( D^{k,m}_x f(x,\xi) \) for a finite set of directions $\xi$.
\end{theorem}
We refer to \cite{Ambartsoumian:Mishra:Zamindar:2024}, where special cases of the above theorem are studied in the plane.

The second objective of this paper is to present a reconstruction formula, unique continuation and stability results related to the fractional divergent beam ray transform, which is defined as follows.
\begin{definition}\label{def:frac_DBT}
    Let \( n \geq 2 \) be an integer. For each real number \( s \in \mathbb{R}^+ \), we define the \textbf{fractional divergent beam ray transform} of a tensor field \( f \in \mathcal{S}(\mathbb{R}^n; S^m(\mathbb{R}^n)) \) as follows:
    \[
    (\chi_{s,m} f)(x,\xi) = \int_0^{\infty} t^{2s-1} f_{i_1 \cdots i_m}(x + t\xi) \xi^{i_1} \cdots \xi^{i_m} \, dt = \int_0^{\infty} t^{2s-1} \langle f(x + t\xi), \xi^m \rangle \, dt.
    \]
\end{definition}
\noindent Note that for any \( -1 < k < n-1 \), such that \( k = 2s-1 \), one can generalize Definition \ref{def:k_weight_DBT} to Definition \ref{def:frac_DBT}, such that \( D^{k,m} f = \chi_{s,m} f \). 

The mathematical analysis of inverse problems related to integral transforms addresses key issues such as uniqueness, stability, reconstruction, unique continuation, range characterization, and partial data. In the second part of this article, we focus on the reconstruction, stability, and unique continuation properties of the fractional divergent beam ray transform.

The ray transform of symmetric $m$-tensor fields in $\mathbb{R}^n$ is defined as the integral of the tensor field along lines in $\mathbb{R}^n$. The reconstruction problem for this integral operator is overdetermined when $n \geq 3$. However, when the normal operator of the ray transform is considered, the problem becomes fully determined. As shown in \cite{sharafutdinov2012integral}, the solenoidal part of the tensor field can be recovered from the normal operator, while the potential part lies in its kernel. Recently, \cite{JKKS1} demonstrated that complete recovery of a symmetric tensor field is possible from the normal operators of the first $m+1$ momentum ray transforms, which are integrals of the symmetric tensor field with the weight $t^k$ for $0 \leq k \leq m$, along with an explicit reconstruction formula. The algorithm for recovering the tensor field from the momentum ray transform was established in \cite{Krishnan:Manna:Sahoo:Sharafutdinov:2019}. We also refer to related works on the reconstruction of various integral transforms, including \cite{ Ambartsoumian:Mishra:Zamindar:2024, JKKSII, Mishra:Thakkar:2024, PU2004}. In this paper, we study the averaging operator derived from the fractional divergent beam ray transform instead of the normal operator. We provide an explicit reconstruction of vector fields and symmetric 2-tensor fields from their averaging operator. It is worth noting that higher-order tensors can be recovered using similar techniques, although the computations become increasingly cumbersome.

The stability of the ray transform is a well-studied problem. In \cite[Theorem 2.2]{Pestov:Sharafutdinov:1988}, a conditional, non-sharp stability estimate was established for compact non-trapping, non-positive curvature manifold with strictly convex boundary, using energy-type estimates. Sharp stability estimates were obtained via microlocal analysis in \cite{Stefanov:Uhlmann:2004}. Local stability results using Melrose's scattering calculus are also available, as shown in \cite{Stefanov:Uhlmann:Vasy:2018, Uhlmann:Vasy:2016}. Sharp stability estimates for non-positive curvature manifolds were proven in \cite{Paternain:Salo:2021} using energy estimates. We also refer to \cite{Boman:Sharafutdinov:2018, Krishnan:Sharafutdinov:2022} for further stability results. In this paper, we prove stability results for the fractional divergent beam ray transform for vector fields and 2-tensor fields.

 Recently, the study of unique continuation for integral transforms has gained attention. For instance, the unique continuation property of the ray transform of functions, the d-plane transform, and the Radon transforms was examined in \cite{Ilmavirta:Monkkonen:2020}. This result was extended to the ray transform of one-forms in \cite{Ilmavirta:Monkkonen:2021}. The unique continuation for the ray transform of symmetric tensor fields and the $k$-th momentum ray transform was established in \cite{Agrawal:Krishnan:Sahoo:2022}. Additionally, a unique continuation result for the fractional divergent beam ray transform on Schwartz class functions was proven in \cite{ilmavirta2023unique}. In this article, we extend this result to vector fields and 2-tensor fields. We also conclude by demonstrating the unique continuation of the divergent beam ray transform for Schwartz functions, with a remark that this result does not extend to the case of vector fields.

\noindent In summary, this article presents the following results:

\begin{itemize}
\item Reconstruction of $m$-tensor field, $m \geq 0$, from corresponding $k$-weighted divergent beam ray transform data.
\item Reconstruction of vector fields and 2-tensor fields from corresponding fractional divergent beam ray transform data.
\item Unique continuation results for fractional divergent beam ray transform for vector fields and 2-tensor fields.
\item Stability results for fractional divergent beam ray transform for vector fields and 2-tensor fields.
\item Unique continuation results for divergent beam ray transform for Schwartz functions, along with a counterexample showing this result does not hold for vector fields.
\end{itemize}

\section{Preliminaries}\label{sec:Pre}

\subsection{Tensor algebra over \texorpdfstring{$\mathbb{R}^n$}{Rn}}
Consider \( T^m \mathbb{R}^n = T^m \), the \( n^m \)-dimensional complex vector space of \( m \)-tensors on \( \mathbb{R}^n \). Let \( e_1, \dots, e_n \) represent the standard basis for \( \mathbb{R}^n \). For a given tensor \( u \in T^m \), the components (or coordinates) of the tensor are denoted as \( u_{i_1 \cdots i_m} = u\left(e_{i_1}, \dots, e_{i_m}\right) \). In this framework, if \( u \in T^m \) and \( v \in T^k \), the tensor product \( u \otimes v \), which belongs to \( T^{m+k} \), is defined by:
\[
(u \otimes v)\left(x_1, \dots, x_m, x_{m+1}, \dots, x_{m+k}\right) = u\left(x_1, \dots, x_m\right) v\left(x_{m+1}, \dots, x_{m+k}\right).
\]
Now, let \( S^m = S^m (\mathbb{R}^n) \) denote the subspace of \( T^m \) that consists of symmetric tensors, having a dimension of \(\binom{n+m-1}{m}\). The symmetrization operator, \( \sigma: T^m \rightarrow S^m \), is defined as follows:
\[
\sigma u\left(e_1, \dots, e_m\right) = \frac{1}{m!} \sum_{\pi \in \Pi_m} u\left(e_{\pi(1)}, \dots, e_{\pi(m)}\right),
\]
where \( \Pi_m \) represents the group of permutations of the set \( \{1, \dots, m\} \). For tensors \( u \in S^m \) and \( v \in S^k \), the symmetric product \( u \odot v \), which resides in \( S^{m+k} \), is defined as \( u \odot v = \sigma(u \otimes v) \).

\subsection{Tensor fields}

Recall that the Schwartz space \(\mathcal{S}(\mathbb{R}^n)\) is a topological vector space comprising \(C^{\infty}\)-smooth, complex-valued functions defined on \(\mathbb{R}^n\) that exhibit rapid decay at infinity, along with all their derivatives. This space is equipped with the standard topology. The dual space of \(\mathcal{S}(\mathbb{R}^n)\) is known as the space of tempered distributions and is denoted by \(\mathcal{S}^{\prime}(\mathbb{R}^n)\).

Now, let \(\mathcal{S}(\mathbb{R}^n ; S^m) = \mathcal{S}(\mathbb{R}^n) \otimes S^m\) represent the topological vector space of smooth, rapidly decaying symmetric \(m\)-tensor fields defined on \(\mathbb{R}^n\). The components of these tensor fields belong to the Schwartz space. In Cartesian coordinates, such a tensor field is expressed as \(f = (f_{i_1 \ldots i_m})\), where the components \(f_{i_1 \ldots i_m} = f^{i_1 \ldots i_m} \in \mathcal{S}(\mathbb{R}^n)\) are symmetric with respect to all indices.

It is important to note that, in this context, there is no distinction between covariant and contravariant components, as we are working exclusively with Cartesian coordinates.

\subsection{Fourier transform}
For \( f \in \mathcal{S}(\mathbb{R}^n) \), the Fourier transform is defined by
\[
    \wh{f}(y) = \int_{\mathbb{R}^n} e^{-\mathrm{i} \langle x, y \rangle} f(x) \, dx, \quad y \in \mathbb{R}^n.
\]
If \( T \in \mathcal{S}^{\prime}(\mathbb{R}^n) \), its Fourier transform \( \wh{T} \) is the linear form on \( \mathcal{S}(\mathbb{R}^n) \) defined by
\[
\wh{T}(\varphi) = T(\wh{\varphi}).
\]
We recall the following properties of the Fourier transform from \cite[Theorem 4.2, 4.6, and Lemma 6.2, Chapter VII]{Helgason:2010:IntegralGeom:book}:
\begin{enumerate}
    \item \( \mathrm{i}^{|\alpha| + |\beta|} y^{\beta} D^{\alpha}(\wh{f})(y) = \wh{D^{\beta}(x^{\alpha}f)}(y) \)
    
    \item \( \rwh{\left( f_1 * f_2 \right)} = (\wh{f_1}) (\wh{f_2}) \) for \( f_1, f_2 \in \mathcal{S}(\mathbb{R}^n) \)
    
    \item \( \rwh{\left( T * f \right)} = (\wh{T}) (\wh{f}) \) for \( f \in \mathcal{S}(\mathbb{R}^n) \), \( T \in \mathcal{S}'(\mathbb{R}^n) \)
     
    \item \( \wh{ \left( |x|^\alpha\right) } (y) = 2^{n+\alpha} \pi^{\frac{n}{2}} \frac{\Gamma\left(\frac{n+\alpha}{2}\right)}{\Gamma\left(-\frac{\alpha}{2}\right)} |y|^{-\alpha-n}, \quad -\alpha-n \notin 2\mathbb{Z}^{+}. \)
\end{enumerate}
In the expressions, $*$ denotes the convolution. The Fourier transform \(\mathscr{F}: \mathcal{S}\left(\mathbb{R}^n ; S^m\right) \rightarrow \mathcal{S}\left(\mathbb{R}^n ; S^m\right), f \mapsto \wh{f}\) of symmetric tensor fields is defined component-wise: \( (\wh{f})_{i_1 \ldots i_m} = \wh{f_{i_1 \ldots i_m}}\).
Note that the Fourier transform $\wh{}$ acts as an isomorphism on the spaces $\mathcal{S}(\mathbb{R}^n)$, $\mathcal{S}^{\prime}(\mathbb{R}^n)$ and we denote its inverse by ${\wh{}}^{-1} u$. The Bessel potential of order $s \in \R$ is the Fourier multiplier $\vev{D}^s\colon  \mathcal{S}^{\prime}(\mathbb{R}^n) \to  \mathcal{S}^{\prime}(\mathbb{R}^n)$, that is
\begin{equation}\label{eq: Bessel pot}
    \vev{D}^s u \vcentcolon = {\wh{}}^{-1}(\vev{\xi}^s\wh{u}),
\end{equation} 
where $\vev{\xi}\vcentcolon = (1+|\xi|^2)^{1/2}$. If $s \in \R$ and $1 \leq p < \infty$, the Bessel potential space $H^{s,p}(\R^n)$ is given by
\begin{equation}
\label{eq: Bessel pot spaces}
    H^{s,p}(\R^n) \vcentcolon = \{ u \in  \mathcal{S}^{\prime}(\mathbb{R}^n)\,;\, \vev{D}^su \in L^p(\R^n)\},
\end{equation}
 endowed with the norm 
 \[
    \norm{u}_{H^{s,p}(\R^n)} \vcentcolon = \norm{\vev{D}^su}_{L^p(\R^n)}.
\]
For our convenience, we use the notation $J_s u$ for the Bessel potential such that
$J_s u = {\wh{}}^{-1}[(1+|\xi|^2)^{-\frac{s}{2}}\wh{u}]$, for all $s\in\mathbb{R}$, see \cite[Chapter 12]{Wong}.

Let $H^{t,p}(\mathbb{R}^n ; S^m(\mathbb{R}^n)) = H^{t,p}(\mathbb{R}^n) \otimes S^m(\mathbb{R}^n)$ represent the topological vector spaces of symmetric $m$-tensor fields defined on $\mathbb{R}^n$ which belongs to $H^{t,p}(\mathbb{R}^n)$ Sobolev class for all $t\in\mathbb{R}$ and $1<p<\infty.$ In Cartesian co-ordinates such a tensor field is expressed as \(f = (f_{i_1 \ldots i_m})\), where the components \(f_{i_1 \ldots i_m} = f^{i_1 \ldots i_m} \in H^{t,p}(\mathbb{R}^n)\) are symmetric with respect to all indices. Since Schwartz space $\mathcal{S}(\mathbb{R}^n)$ is dense in $H^{t,p}(\mathbb{R}^n)$, it is immediate to see that $ \mathcal{S}(\mathbb{R}^n ; S^m)$ is dense in $H^{t,p}(\mathbb{R}^n ; S^m(\mathbb{R}^n))$. The norm of $H^{t,p}(\mathbb{R}^n ; S^m(\mathbb{R}^n))$ is defined as
\[
\|f\|_{H^{t,p}(\mathbb{R}^n ; S^m(\mathbb{R}^n))}
=\sum_{i_1\cdots i_m}\|f_{i_1\cdots i_m}\|_{H^{t,p}(\mathbb{R}^n)}.
\]
We will now recall the following complex interpolation result from \cite{Berg}.
\begin{theorem}[Complex interpolation, Theorem 6.4.5, \cite{Berg}]\label{interpolation}
Let $\theta$ be given such that $0<\theta<1$. Also let $s^*, p^*$ be such that $s^* = (1-\theta)s_0 + \theta s_1$ and $\frac{1}{p^*} = \frac{1-\theta}{p_0} + \frac{\theta}{p_1}$. Then we have
\[
\left(   H^{s_0,p_0}(\R^n),   H^{s_1,p_1}(\R^n)\right)_{[\theta]} =   H^{s^{*},p^{*}}(\R^n),
\]
where $s_0\neq s_1$ and $1<p_0, p_1<\infty$.
\end{theorem}
In the above theorem, the notation $\left(   H^{s_0,p_0}(\R^n),   H^{s_1,p_1}(\R^n)\right)_{[\theta]} $ means that the Bessel potential type Sobolev space $ H^{s^*,p^*}(\R^n)$ can be obtained by the complex interpolation between $ H^{s_0,p_0}(\R^n)$ and $H^{s_1,p_1}(\R^n)$. Given two compatible couple of Banach spaces $\left( H^{s_0,p_0}(\R^n),   H^{s_1,p_1}(\R^n)\right)$ and $\left(   H^{\tilde{s},\tilde{p}}(\R^n),   H^{\tilde{s}_1,\tilde{p}_1}(\R^n)\right)$, the pair $\left(\left(   H^{s_0,p_0}(\R^n),   H^{s_1,p_1}(\R^n)\right)_{[\theta]}, \left(   H^{\tilde{s},\tilde{p}}(\R^n),   H^{\tilde{s}_1,\tilde{p}_1}(\R^n)\right)_{[\theta]}  \right)$ is an exact interpolation pair of exponent $\theta$. That is, if $T : H^{s_0,p_0}(\R^n) + H^{s_1,p_1}(\R^n) \rightarrow H^{\tilde{s},\tilde{p}}(\R^n) + H^{\tilde{s}_1,\tilde{p}_1}(\R^n)$ is a linear operator such that
$T : H^{s_0,p_0}(\R^n) \rightarrow H^{\tilde{s},\tilde{p}}(\R^n)$ and $T : H^{s_1,p_1}(\R^n) \rightarrow H^{\tilde{s}_1,\tilde{p}_1}(\R^n)$
are bounded, then $T : H^{s^*,p^*}(\R^n) \rightarrow H^{\tilde{s}^*,\tilde{p}^*}(\R^n)$ is bounded, where the interpolation spaces are $\left(   H^{s_0,p_0}(\R^n),   H^{s_1,p_1}(\R^n)\right)_{[\theta]} = H^{s^*,p^*}(\R^n)$ and $\left(   H^{\tilde{s},\tilde{p}}(\R^n),   H^{\tilde{s}_1,\tilde{p}_1}(\R^n)\right)_{[\theta]} = H^{\tilde{s}^*,\tilde{p}^*}(\R^n)$. The corresponding operator norm can be bounded by
\[
\|T\|_{H^{s^*,p^*} \rightarrow H^{\tilde{s}^*,\tilde{p}^*}} \leq \|T\|_{H^{s_0,p_0} \rightarrow H^{\tilde{s},\tilde{p}}}^{1-\theta} \|T\|_{H^{s_1,p_1}\rightarrow H^{\tilde{s}_1,\tilde{p}_1}}^{\theta}
\]
where $s^*, p^*$ be such that $s^* = (1-\theta)s_0 + \theta s_1$ and $\frac{1}{p^*} = \frac{1-\theta}{p_0} + \frac{\theta}{p_1}$ with $s_0\neq s_1$ and $1<p_0, p_1<\infty$, and $\tilde{s}^*, \tilde{p}^*$ be such that $\tilde{s}^* = (1-\theta)\tilde{s} + \theta \tilde{s}_1$ and $\frac{1}{\tilde{p}^*} = \frac{1-\theta}{\tilde{p}} + \frac{\theta}{\tilde{p}_1}$ with $\tilde{s}\neq \tilde{s}_1$ and $1<\tilde{p}, \tilde{p}_1<\infty$. In particular, if $f\in H^{s^*,p^*}(\R^n)$, then 
\begin{align*}
    \|Tf\|_{H^{\tilde{s}^*,\tilde{p}^*}(\R^n)}
    &\leq \|T\|_{H^{s^*,p^*} \rightarrow H^{\tilde{s}^*,\tilde{p}^*}} \|f\|_{H^{s^*,p^*}(\R^n)} \\
& \leq T\|_{H^{s_0,p_0} \rightarrow H^{\tilde{s},\tilde{p}}}^{1-\theta} \|T\|_{H^{s_1,p_1}\rightarrow H^{\tilde{s}_1,\tilde{p}_1}}^{\theta} \|f\|_{H^{s^*,p^*}(\R^n)}.
\end{align*}
We will use this interpolation result to prove certain boundedness estimate for averaging operators in the fractional Sobolev spaces.

\subsection{Riesz transform}
The Riesz transforms of a complex valued Schwartz class function $f$ on $\mathbb{R}^n$ are defined by
 \[
 \mathbf{R}_jf(x) = c_n \lim_{\epsilon\rightarrow 0}\int_{\mathbb{R}^n\setminus B_{\epsilon}(x)} \frac{(x_j - z_j)}{|x-z|^{n+1}} f(z) dz
 \]
for $j=1,2,\cdots, n$ where the constant $c_n$ is given by
\[
c_n = \frac{1}{\pi \omega_{n-1}} = \frac{\Gamma(\frac{n+1}{2})}{\pi^{\frac{n+1}{2}}}
\]
and $\omega_{n-1}$ being the $n-1$ dimensional Lebesgue measure of the unit sphere $\mathbb{S}^{n-1}$.
For $f\in L^p(\mathbb{R}^n), 1<p<\infty,$ the Riesz transform operators 
\( \mathbf{R}_j : L^p(\mathbb{R}^n) \rightarrow L^p(\mathbb{R}^n) \)
are bounded,
see for instance \cite[Theorem 2.6, Chapter VI]{Stein_Weiss}. 
For every $f\in L^2(\mathbb{R}^n)$, we have from \cite[Theorem 2.6, Chapter VI]{Stein_Weiss} that
\begin{equation}\label{Fourier_Riesz_transform}
\mathscr{F}{(\mathbf{R}_jf)}(\zeta) = -\mathrm{i} \frac{\zeta_j}{|\zeta|}(\mathscr{F} f)(\zeta).
\end{equation}
Since $f\in L^2(\mathbb{R}^n)$, the Riesz transform $\mathbf{R}_j f\in L^2(\mathbb{R}^n)$. Moreover, we have
\begin{equation}\label{Fourier_Riesz_transform_star}
\mathscr{F}({\mathbf{R}_j \circ \mathbf{R}_k f})(\zeta) = (-\mathrm{i})^2 \frac{\zeta_j\zeta_k}{|\zeta|^2}(\mathscr{F} f)(\zeta).
\end{equation}
Similarly, we have
\[
\mathscr{F}\left(\mathbf{R}_{i_1}\circ\cdots \circ \mathbf{R}_{i_k}f\right)(\zeta)
= (-\mathrm{i})^{k} \frac{\zeta_{i_1}\cdots \zeta_{i_k}}{|\zeta|^{k}} (\mathscr{F} f)(\zeta)
\]
holds for every $f\in L^2(\mathbb{R}^n)$.
In the case of $m$-tensor field $f$, we define the Riesz transform via each component function $f_{j_1\cdots j_m}$. In particular, we have
\begin{equation}
    \mathbf{R}_{\ell} f_{j_1\cdots j_m}= -\mathrm i \mathscr{F}^{-1}\left(\frac{\zeta_{\ell}}{|\zeta|} \mathscr{F} (f_{j_1\cdots j_m})\right),
\end{equation}
for all $\ell = 1, \cdots, n$. 
Next, we will show the boundedness properties of the Riesz transforms in the corresponding Sobolev spaces.
\begin{theorem}[Boundedness properties of Riesz transform]\label{bdd_riesz}
Let $1<p<\infty$ and $t\geq 0$ be a real number. Then the Riesz transforms
\[
\mathbf{R}_{j} : H^{t,p}(\R^n) \rightarrow H^{t,p}(\R^n)
\]
are bounded linear operator for all $j=1,2,\cdots, n$. 
\end{theorem}
\begin{proof}
    We have already noticed that, Riesz transforms
  \( \mathbf{R}_j : L^p(\mathbb{R}^n) \rightarrow L^p(\mathbb{R}^n) \)
are bounded, for $f\in L^p(\mathbb{R}^n), 1<p<\infty.$ From the definition of Riesz transform, we have
$
\mathbf{R}_j f = K * f
$
where the kernel is of the form $K(x,z) = c_n\frac{x_j -z_j}{|x-z|^{n+1}}$. Therefore, by the convolution theorem, we have
$
\partial^{\alpha}(\mathbf{R}_j f) = \partial^{\alpha}( K * f) = K * \partial^{\alpha} f = \mathbf{R}_j (\partial^{\alpha} f)
$
holds for all multi-indices $\alpha = (\alpha_1, \cdots, \alpha_n)$ with $\alpha_i\in \mathbb{N}$. Thus we have
\[
\|\mathbf{R}_j f\|_{W^{l,p}(\mathbb{R}^n)} \leq C \|f\|_{W^{l,p}(\mathbb{R}^n)} 
\]
where $C>0$ be a constant and $l=|\alpha|\in\mathbb{N}.$
In other words,
$
\mathbf{R}_j : W^{l,p}(\mathbb{R}^n) \rightarrow W^{l,p}(\mathbb{R}^n)
$
is a bounded linear operator. By the interpolation theorem, see Theorem \ref{interpolation}, we conclude that
$
\mathbf{R}_j : H^{t,p}(\R^n) \rightarrow H^{t,p}(\R^n)
$
is a bounded linear operator, where $1<p<\infty$ and $t\geq 0$ be a real number.
\end{proof}

\subsection{Riesz potential}\label{Sec:RP}

Consider a function \( f \) in the Schwartz space \(\mathcal{S}\left(\mathbb{R}^n\right)\) and a complex number \(\gamma\). The Riesz potential of \( f \) is defined as follows:
\begin{equation}\label{Sec:RP:Eq1}
(I^\gamma f)(x) = h_n(\gamma) \int_{\mathbb{R}^n} |y|^{\gamma - n} f(x-y) \, dy, \quad h_n(\gamma) =\frac{\Gamma\left(\frac{n-\gamma}{2}\right)}{2^\gamma \pi^{\frac{n}{2}} \Gamma\left( \frac{\gamma}{2}\right) }.
\end{equation}
In cases where \( -\gamma \) belongs to the set of positive even integers, the poles of \( \Gamma(\gamma / 2) \) are negated by those of \( |y|^{\gamma-n} \), as detailed in \cite[Chapter VII, Section 6]{Helgason:2010:IntegralGeom:book}. Therefore, if \(\gamma-n\) does not belong to the set of even integers, the Riesz potential can be represented as a convolution:
\begin{equation}
(I^\gamma f)(x) = \left( f * h_n(\gamma) |x|^{\gamma-n} \right) (x), \quad f \in \mathcal{S}(\mathbb{R}^n).
\end{equation}
Through the application of the Fourier transform, we obtain:
\begin{equation}\label{Sec:RP:Eq2}
\wh{(I^\gamma f)}(\eta) = |\eta|^{-\gamma} \wh{f} (\eta), \quad \gamma-n \notin 2\mathbb{Z}^+
\end{equation}
in the context of tempered distributions.

\begin{lemma}\cite[Chapter VII, Propositions 6.5 and 6.8]{Helgason:2010:IntegralGeom:book}
For \( f \in \mathcal{S}\left(\mathbb{R}^n\right) \), the following composition formulas are valid:
\begin{enumerate}
    \item \( I^\alpha (I^\beta f) = I^{\alpha+\beta}f, \quad \real(\alpha), \real(\beta)>0, \quad \real(\alpha+\beta) < n \).
    \item \( I^{-k} (I^k f) = f, \quad 0 < k < n \) and \( f(x)=O(|x|^{-N}) \) for some \( N>n \).
\end{enumerate}
\end{lemma}

\noindent Following \cite[Theorem 1, Chapter V]{Stein:1970},
we have that the Riesz potential operator
\begin{equation}\label{bdd_Riesz}
    I^{\alpha} : L^p(\mathbb{R}^n)\rightarrow L^q(\mathbb{R}^n)
\end{equation}
is a bounded linear operator where $0<\alpha<n, 1<p<q<\infty, \frac{1}{q} = \frac{1}{p} - \frac{\alpha}{n}.$ In particular, we have
\begin{equation}\label{estimate_Riesz}
    \|I^{\alpha}f\|_{L^q(\mathbb{R}^n)} \leq A_{p,q}\|f\|_{L^p(\mathbb{R}^n)}
\end{equation}
 where $A_{p,q}>0$ is a constant, sometimes it is called as $A_{p,q}$ weight.


\subsection{Inverse fractional Laplacian}

For a function \( f \in \mathcal{S}\left(\mathbb{R}^n\right) \) and a positive number \( s \), the Fourier transform of \( (-\Delta)^s f \) is given by \( |\eta|^{2s} \wh{f}(\eta) \). The negative power of the Laplacian, \( (-\Delta)^{-s} \), for \( s > 0 \), is defined as:
\begin{equation} \label{Prel:FracLap}
\rwh{((-\Delta)^{-s}f)} (\eta) = |\eta|^{-2s} \wh{f}(\eta) \quad \text{for} \quad \eta \neq 0.
\end{equation}
This multiplier, \( |\eta|^{-2s} \), must be a tempered distribution, which necessitates the condition \( 0 < s < \frac{n}{2} \) \cite[Section 3]{Stinga:2019:handbook}. To compute \( (-\Delta)^{-s} f(x) \), the inverse Fourier transform of \eqref{Prel:FracLap} needs to be calculated, a task that is complex due to the nature of the Fourier multiplier \( |\eta|^{-2s} \). However, this challenge can be circumvented by employing the method of semigroups, leading to the definition of the fractional Laplacian as:
\[
(-\Delta)^{-s} f(x) = \frac{1}{\Gamma(s)} \int_0^{\infty} e^{t \Delta} f(x) \frac{dt}{t^{1-s}}.
\]
The following theorem establishes a connection between the fractional Laplacian and the Riesz potential.

\begin{theorem}\cite[Theorem 5]{Stinga:2019:handbook}
    Let \( f \in \Sc(\R^n) \) and \( 0 < s < \frac{n}{2} \). Then
    \[
    (-\Delta)^{-s} f(x) = (I^{2s} f)(x) = h_n(2s) \int_{\mathbb{R}^n} |y|^{2s-n} f(x-y) \, dy,
    \]
    where \( h_n(2s) = \frac{\Gamma\left(\frac{n}{2} - s\right)}{4^s \Gamma(s) \pi^{\frac{n}{2}}}. \)
\end{theorem}

\begin{lemma}\label{lm:Delta:comp}
    Let \( s_1 \) and \( s_2 \) be two positive real numbers such that \( 0 < s_1 + s_2 < \frac{n}{2} \). The following composition formula holds for \( f \in \Sc(\R^n) \):
    \[
    (-\Delta)^{-s_1} \left( (-\Delta)^{-s_2} f \right) = (-\Delta)^{-(s_1+s_2)} f.
    \]
\end{lemma}

\noindent This follows from the semigroup property satisfied by the operator \( (-\Delta)^{-s} \). 
If $u\in \mathcal{S}^{\prime}(\mathbb{R}^n)$ is a tempered distribution and $s\geq 0$, the fractional Laplacian of order $s$ of $u$ is the Fourier multiplier
$
    (-\Delta)^su\vcentcolon = \rwh{}^{-1}(|\xi|^{2s}\rwh{u}),
$
whenever the right hand side is well-defined. 
\begin{theorem}[Boundedness properties of fractional Laplacian]\label{Bdd_frac_LAP}
If $s\geq 0$, the fractional Laplacian operator extends as a bounded linear map
\[
(-\Delta)^{s} : H^{t,p}(\mathbb{R}^n) \rightarrow H^{t-2s,p}(\mathbb{R}^n)
\]
whenever $t\in\mathbb{R}$ and $1\leq p<\infty$.
\end{theorem}
\begin{proof}
 The lemma follows from \cite[Lemma 2.1]{Ghosh:Salo:Uhlmann2020} when $p=2$, see also \cite[Proposition 6.5]{TaylorIII:1997} for general $p$. For the convenience of the reader, we will provide some details of the proof. If $u\in \mathcal{S}(\R^n)$, then
 $
 \|(-\Delta)^s u\|_{H^{t-2s,p}(\mathbb{R}^n)} = \|\rwh{}^{-1}\{m(\xi)\langle \xi\rangle^t\rwh{u}(\xi)\}\|_{L^p(\mathbb{R}^n)}
 $
 where $m(\xi) = \langle \xi\rangle^{-2s}|\xi|^{2s}$ is bounded and hence a Fourier multiplier on $L^p$, which implies
 \begin{align}\label{bdd_frac_Laplace}
     \|(-\Delta)^s u\|_{H^{t-2s,p}(\mathbb{R}^n)} \leq C \|u\|_{H^{t,p}(\mathbb{R}^n)}.
 \end{align}\qedhere 
\end{proof}

\subsection{Pseudodifferential calculus}
We recall few important facts from pseudodifferential Calculus, see \cite{Paternain:Salo:Uhlmann, Wong} for references. Let $m\in\mathbb{R}$. We define the set of {\it{symbols}} of order $m$ by the set of all $\sigma\in C^{\infty}(\mathbb{R}^n\times \mathbb{R}^n)$ such that $|D_{x}^{\alpha}D_{\xi}^{\beta}\sigma (x,\xi)|\leq C_{\alpha,\beta}(1+|\xi|)^{m-|\beta|}$, for any two multi-indices $\alpha,\beta$ and $x,\xi\in\mathbb{R}^n$, where $C_{\alpha,\beta}>0$ are constants depend on $\alpha$ and $\beta$. The pseudodifferential operator corresponding to the symbol $\sigma$ is defined by 
\[
(T_{\sigma}\phi)(x) = (2\pi)^{-n}\int_{\mathbb{R}^n}e^{\mathrm{i} \langle x, \xi\rangle}\sigma(x,\xi)\rwh{\phi}(\xi) d\xi, \ \ \ \phi\in\mathcal{S}(\mathbb{R}^n).
\]
Note that $T_{\sigma}$ maps the Schwartz space $\mathcal{S}(\mathbb{R}^n)$ to itself, see \cite[Proposition 6.7]{Wong}. Also, $T_{\sigma}$ is a linear map from tempered distribution $\mathcal{S}^{\prime}(\mathbb{R}^n)$ into $\mathcal{S}^{\prime}(\mathbb{R}^n)$, see \cite[Proposition 11.4]{Wong}. It has also nice mapping properties between $L^p$ spaces, that is in particular, if $\sigma$ is a symbol of order zero, then $T_{\sigma} : L^p(\mathbb{R}^n) \rightarrow L^p(\mathbb{R}^n) $ is a bounded linear operator for $1<p<\infty$, see \cite[Theorem 11.7]{Wong}. Moreover, when $\sigma$ belongs to the set of symbols of order $m$, the operator enjoys mapping properties between Sobolev spaces.
\begin{theorem}\cite[Theorem 12.9]{Wong} \label{bdd_pseudo} Let $\sigma$ be a symbol of order $m$. Then 
\[
T_{\sigma} : H^{s,p}(\mathbb{R}^n) \rightarrow H^{s-m,p}(\mathbb{R}^n)
\]
  is a bounded linear operator for all $s\in\mathbb{R}$ and $1<p<\infty$.  
\end{theorem}
Next theorem is important in proving the boundedness estimate of averaging operators. We denote $\mathcal{E}'(\mathbb{R}^n)$ by the space of all compactly supported distributions. Let $\Psi_{cl}^{m} = \{Op(a) ; a\in S_{cl}^{m}\}$ be the set of classical pseudodifferential operators of order $m$, where $S_{cl}^{m}$ be the set of all {\it{classical}} symbols of order $m$, see \cite[Definition 1.3.13]{Paternain:Salo:Uhlmann} for the precise definition.
\begin{theorem}\label{bdd_pseudo_2s}
    Let $n\geq 2$ be an integer. Let $f\in \mathcal{S}(\mathbb{R}^n)$. Define an operator
    \[
    T_{\sigma_{s}}f(x) := (2\pi)^{-n}\int_{\mathbb{R}^n} e^{\mathrm{i} \langle x,\xi\rangle}\sigma_{s}(\xi)\rwh{f}(\xi) d\xi,
    \]
where $\sigma_s(\xi) := |\xi|^{-2s}.$ Then $T_{\sigma_{s}}$
is a pseudodifferential operator of order $-2s$. Moreover, $T_{\sigma_{s}} = Q+S$, where $Q\in \Psi_{cl}^{-2s}$ is elliptic, and $S$ is a smoothing operator which maps $\mathcal{E}'(\mathbb{R}^n)$ to $C^{\infty}(\mathbb{R}^n)$. Furthermore
\[
 T_{\sigma_{s}} : H^{t,p}(\mathbb{R}^n) \rightarrow H^{t+2s,p}(\mathbb{R}^n)
\]
is a bounded linear operator, for all $s\in (0,1), 1<p<\infty, t\in \mathbb{R}.$
\end{theorem}
\begin{proof}
It can be observed that, the function $\sigma_s(\xi) = |\xi|^{-2s} $ is a symbol of order $-2s$, since, for any multi-indices $\alpha,\beta\in \mathbb{N}_{0}^{n}$, there exist $C_{\alpha,\beta}>0$ such that $|D_{x}^{\alpha}D_{\xi}^{\beta}\sigma_s (\xi)|\leq C_{\alpha,\beta}(1+|\xi|)^{-2s-|\beta|}$, for $x,\xi\in\mathbb{R}^n$.
To prove the second part of the theorem, we proceed as follows: Let $\psi\in C_{c}^{\infty}(\mathbb{R}^n)$ be such that $\psi(\xi)=1$ for $|\xi|\leq \frac{1}{2}$ and $\psi(\xi)=0$ for $|\xi|\geq 1$. Write $Qf = \rwh{}^{-1}\left\{\frac{1-\psi(\xi)}{|\xi|^{2s}} \rwh{(f)}\right\}$ and $Sf =  \rwh{}^{-1}\left\{\frac{\psi(\xi)}{|\xi|^{2s}} \rwh{(f)}\right\}$. Then $Q$ is a pseudodifferential operator in $\Psi_{cl}^{-2s}$ with symbol $\tilde{\sigma_s}(x,\xi) = \frac{1-\psi(\xi)}{|\xi|^{2s}}$, hence $Q$ is elliptic. Since $\frac{\psi(\xi)}{|\xi|^{2s}}\in L^1(\mathbb{R}^n)$ is compactly supported, by \cite[Lemma 1.3.17]{Paternain:Salo:Uhlmann}, the operator $S : f \rightarrow \rwh{}^{-1}\left\{\frac{\psi(\xi)}{|\xi|^{2s}} \rwh{(f)}\right\}$ is smoothing in the sense that it maps $\mathcal{E}'(\mathbb{R}^n)$ to $C^{\infty}(\mathbb{R}^n).$ Finally the operator $T_{\sigma_s}$ is bounded due to Theorem \ref{bdd_pseudo}.
\end{proof}

 \section{Weighted Divergent Beam Ray Transform: Reconstruction}\label{sec:wDBRT}
Before proving the Theorem \ref{tm:recon:dbrt}, we need one technical lemma. This is proved for $f \in \mathcal{S}\left(\mathbb{R}^3 ; S^m\left(\mathbb{R}^3\right)\right)$ in \cite[Lemma 3]{Mishra:Thakkar:2024} using mathematical induction.  
Here, we provide an alternate proof using the properties of homogeneous polynomials for any symmetric $m$-tensor fields.
 \begin{lemma}\label{lm:dbrt:workhorse}
Let $m\ge 1$ and $f \in \mathcal{S}\left(\mathbb{R}^n ; S^m\left(\mathbb{R}^n\right)\right)$ and $\xi_1,\cdots ,\xi_m\in \mathbb R^n\setminus \{0\}$. For each $1 \leq k \leq m$, define the index set
\[
J_k^m = \left\{ (j_1, \ldots, j_k) : 1 \leq j_i \leq m \text{ for } i = 1, \ldots, k \text{ and } j_1 < \cdots < j_k \right\}.
\]
Then, the following holds:
\begin{align*}
 \langle f, (\xi_1, \ldots, \xi_m) \rangle=\frac{1}{m !} \sum_{k=1}^m(-1)^{m-k}\sum_{J^m_k}  \left\langle f(x), \left(\xi_{j_1}+\cdots +\xi_{j_k} \right)^{\odot m}\right\rangle.
\end{align*}
\end{lemma}
\begin{proof}

From \cite[p. 14]{Mishra:Thakkar:2024}, we have 
$
\left\langle f, \left(\xi_1, \ldots, \xi_m\right)\right\rangle = \left\langle f, \xi_1 \odot \cdots \odot \xi_m \right\rangle
$
where \( \odot \) denotes the symmetric product. It is enough to show that
\begin{equation}\label{eq:int:work1}
    \xi_1 \odot \cdots \odot \xi_m = \sum_{k=1}^m \frac{(-1)^{m-k}}{m!} \left( \sum_{J_k^m} \left(\xi_{j_1} + \cdots + \xi_{j_k}\right)^{\odot m} \right).
\end{equation}
The complex vector space $S^m (\mathbb{R}^n)$ is generated by powers $z^m \left(z \in \mathbb{R}^n\right)$. That is, every statement on symmetric tensors can be translated into the language of polynomials, and vice versa. Thus, it is enough to prove that
 \begin{equation}\label{eq:int:work2}
    z_1z_2\cdots z_m= \sum_{k=1}^m \frac{(-1)^{m-k}}{m!} \left( \sum_{J_k^m} \left(z_{j_1} + \cdots + z_{j_k}\right)^{ m} \right).
\end{equation}
for $z=(z_1,\cdots, z_n)\in \mathbb R^n\setminus \{0\}$.
Let us consider function $g:\mathbb R^m\to \mathbb R$ given by 
\begin{align*}
    g(z_{i_1},\cdots, z_{i_m})=    \sum_{k=1}^m (-1)^{m-k} \left( \sum_{J_k^m} \left(z_{j_1} + \cdots + z_{j_k}\right)^{ m} \right)-m!z_{i_1}\cdots z_{i_m}.
\end{align*}
Notice that $g$ is a smooth function on $\mathbb{R}^m$.
Let $\alpha = (\alpha_1, \cdots, \alpha_m)$ be a multi-index such that $|\alpha| = m$. If there exists an $\alpha_i$ such that $\alpha_i \neq 0$ and $i \neq j_r$ for $j_r \in \{j_1, \cdots, j_k\}$, then
$
\partial^{\alpha}\left(z_{j_1} + \cdots + z_{j_k}\right)^{m} = 0,
$
where $\partial^{\alpha} = \partial_z^{\alpha} = \partial_{z_1}^{\alpha_1} \cdots \partial_{z_m}^{\alpha_m}$.
If no such $\alpha_i$ exists, then by homogeneity and $|\alpha| = m$, we have
$
\partial^{\alpha}\left(z_{j_1} + \cdots + z_{j_k}\right)^{m} = m!.
$
Note that elements in $J_k^m$ can be identified with elements in $\mathbb{R}^m$ with $k$ entries equal to 1 and the remaining $m-k$ entries equal to 0. For any multi-index $\alpha$, $n_{\alpha}$ counts the number of $\alpha_j$ such that $\alpha_j \neq 0$ (each index is counted only once). 
Note that 
\begin{equation}
    \frac{1}{m!}\partial_{z}^{\alpha}\left(\sum_{J_k^m}\left(z_{j_1} + \cdots + z_{j_k}\right)^{ m}\right)=\binom{m-n_\alpha}{k-n_\alpha}.
\end{equation}
The reasoning behind this is that one needs to count the number of entries $J_k^m$ identified in $\mathbb{R}^m$ where exactly $n_{\alpha}$ entries are fixed.
This implies, except $\alpha=(1,\cdots,1)$, we have 
\begin{align*}
    \partial^{\alpha} g(z_{i_1},\cdots, z_{i_m})= m!\sum_{k=0}^{m} (-1)^{m-k}\binom{m-n_{\alpha}}{k-n_{\alpha}}-0.
\end{align*}
After re-indexing and from the agreement that whenever $r< 0$, we have $\binom{s}{r}=0$, this implies  
\begin{align*}
   m !\sum_{k=0}^{m} (-1)^{m-k}\binom{m-n_{\alpha}}{k-n_{\alpha}}=m!(-1)^{m-n_{\alpha}}\sum_{r=0}^{m-n_{\alpha}} (-1)^{r}\binom{m-n_{\alpha}}{r}=0.
\end{align*}
In the case $\alpha=(1,\cdots,1)$, we have
$
     \partial^{\alpha} g(z_{i_1},\cdots, z_{i_m})=m! -m!=0.
$
This implies $\forall \alpha \in \mathbb Z_+^m$ such that $|\alpha|=m$, we have $\partial^{\alpha} g(z_{i_1},\cdots, z_{i_m})=0$. 
As $g$ is a homogeneous polynomial function of degree $m$, and combined with the condition $\partial^{\alpha} g(z_{i_1}, \ldots, z_{i_m}) = 0$ for all $\alpha \in \mathbb{Z}_+^m$ such that $|\alpha| = m$, we conclude that $g$ is a constant function equal to $0$, which proves the claim \eqref{eq:int:work2}.\qedhere


\end{proof}
We now prove Theorem \ref{tm:recon:dbrt}. 
\begin{proof}[Proof of Theorem \ref{tm:recon:dbrt}]
      We first show that $f$ can be recovered point-wise from the information of $D^{0,m}f$ and the case for $D^{k,m}f, k \geq 1,$ follows from the iteration. Let $x \in \mathbb{R}^n$ be the source of the beam $\left.\{x+t \xi\}\right|_{t \geq 0}$ and $\xi \in \mathbb{S}^{n-1}$ be the unit vector in the direction of the beam. Then $D^{0,m}f(x,\xi)$ is defined as 
        \begin{align}\label{eq:2:1}
            D^{0,m}f(x,\xi) = \sum_{j_1,\dots,j_m =1}^{n} \int_{0}^{\infty} f_{ j_1 \dots j_m}(x+t\xi) \xi^{j_1} \dots \xi^{j_m} dt,
        \end{align}
        where $f_{ j_1 \dots j_m}(x)$ is symmetric in all indices $j_1,\dots, j_m \in \{1,\dots,n\}$.
        Differentiating the above equation with respect to ${x_k}$ and multiplying it by $\xi^k$ and then summing over $k$ yields
        \begin{align}\label{eq:2:2}
            \sum_{k=1}^{n} \xi^k\partial_{x_k}D^{0,m}f(x,\xi) &= \sum_{k=1}^{n} \sum_{j_1,\dots,j_m =1}^{n}\int_{0}^{\infty} \xi^k\partial_{x_k}f_{j_1 \dots j_m}(x+t\xi) \xi^{j_1} \dots \xi^{j_m} dt\notag \\
            &= \sum_{j_1,\dots,j_m =1}^{n}\int_{0}^{\infty} \frac{d }{dt} f_{j_1 \dots j_m} (x+t\xi) \xi^{j_1} \dots \xi^{j_m} dt\notag \\
            &= -\sum_{j_1,\dots,j_m =1}^{n} f_{j_1 \dots j_m} (x) \xi^{j_1} \dots \xi^{j_m}
            =-\langle f(x),\xi^{\odot m} \rangle.
        \end{align}
       It remains to show the recovery of a symmetric \( m \)-tensor from \( \langle f(x), \xi^{\odot m} \rangle \) for \( \xi \in \mathbb{S}^{n-1} \).
Note that
\begin{equation}\label{eq:3:1}
f_{i_1 \ldots i_m}(x) = \left\langle f(x), \left(e_{i_1}, \ldots, e_{i_m}\right) \right\rangle,
\end{equation}
where \( e_i \) are elements of the standard basis of \( \mathbb{R}^n \). From Lemma \ref{lm:dbrt:workhorse}, if we know the information of
\begin{equation}\label{eq:3:2}
    \frac{1}{m !} \sum_{k=1}^m(-1)^{m-k}\sum_{J^m_k}  \left\langle f(x), \left(e_{j_1}+\cdots +e_{j_k} \right)^{\odot m}\right\rangle,
\end{equation}
we can recover $f$. From \eqref{eq:2:2},
\begin{align*}
    \sum_{r=1}^n \frac{\Theta_r}{\|\Theta\|} \partial_{x_r} D^{0, m} f\left(x, \frac{\Theta}{\|\Theta\|}\right)=-\left\langle f(x), \left(\frac{\Theta}{\|\Theta\|}\right)^{\odot m}\right\rangle.
\end{align*}
This implies, 
\begin{align*}
    \sum_{r=1}^n \frac{\Theta_r}{\|\Theta\|^{1-m}} \partial_{x_r} D^{0, m} f\left(x, \frac{\Theta}{\|\Theta\|}\right)=-\left\langle f(x), \Theta^{\odot m} \right\rangle.
\end{align*}
From \eqref{eq:3:1} and \eqref{eq:3:2},
\begin{align*}
    &f_{i_1 \ldots i_m}(x) \\&=- \frac{1}{m!} \sum_{k=1}^m(-1)^{m-k} \sum_{J_k^m} \sum_{r=1}^n \frac{\left(e_{j_1}+\cdots+e_{j_k}\right)_r}{\|\left(e_{j_1}+\cdots+e_{j_k}\right)\|^{1-m}} \partial_{x_r} D^{0, m} f\left(x, \frac{\left(e_{j_1}+\cdots+e_{j_k}\right)}{\|\left(e_{j_1}+\cdots+e_{j_k}\right)\|}\right).
\end{align*}
   Now, for the case $l \geq 1,$ we have 
     \Beq\label{sec3:eqn3}
         \begin{aligned}
            \sum_{k=1}^{n} \xi^k\partial_{x_k}D^{l,m}f(x,\xi) &= \sum_{k=1}^{n} \sum_{j_1,\dots,j_m =1}^{n}\int_{0}^{\infty} t^l \xi^k\partial_{x_k}f_{j_1 \dots j_m}(x+t\xi) \xi^{j_1} \dots \xi^{j_m} dt \\
            &= \sum_{j_1,\dots,j_m =1}^{n}\int_{0}^{\infty} t^l \frac{d }{dt} f_{j_1 \dots j_m}  (x+t\xi) \xi^{j_1} \dots \xi^{j_m} dt \\
            &= - l D^{l-1,m}f(x,\xi).
        \end{aligned}
        \Eeq
       We proceed as above, iteratively, and obtain $D^{0,m}f(x,\xi)$ from which we recover $f(x)$.  
\end{proof}

\section{Fractional Divergent Beam Ray: Reconstruction, Unique Continuation}

\label{sec:fracDBRT}

 In this section, we consider fractional divergent beam ray transform for vector fields and 2-tensor fields and prove reconstruction, unique continuation and stability results. 
 Note that, the restrictions on the respective exponents
 for fractional ($\chi_{s,m} f$) and $k$-weighted ($D^{k,m} f$) divergent beam ray transforms are $s\in \mathbb{R}^+$ and $k \in \mathbb{Z}^+\cup \{0\}$.
 In fact, for any $k \in \mathbb{R}^+$, such that \( k = 2s-1 \), one can generalize Definition \ref{def:k_weight_DBT} to Definition \ref{def:frac_DBT}, such that \( D^{k,m} f = \chi_{\frac{k+1}{2},m} f \). 
 One important difference between the fractional and $k$-weighted divergent beam ray transforms is that, in general, the unique continuation principle holds for the former but not for the latter. This is elaborated in Remark \ref{sec4:rem2}.
The computations in the proof of Theorem \ref{tm:recon:dbrt} suggest that
 it is enough to define $\chi_{s,m} f$ for $s \in (0,1)$. Whenever $s \in \mathbb{R}^+$, we apply the procedure in \eqref{sec3:eqn3} $\lfloor 2s -1\rfloor$ times to obtain $\chi_{r,m} f$ from $\chi_{s,m} f$, where $r = \frac{2s - \lfloor 2s -1\rfloor}{2}  \in (0,1)$. Therefore, without loss of generality, we assume that $s \in (0,1)$. 
\begin{lemma}\label{sec4:Lemma1.2}  
    For $f \in \mathcal{S}(\mathbb{R}^n;S^m(\mathbb{R}^n)) $, the fractional divergent beam ray transform $\chi_{s,m} f(\cdot,\xi)$ is a smooth tempered distribution on $\R^n$. Specifically, for any given $\delta >0,$
    $\chi_{s,m} f$ satisfies the following estimate:
    \[
      \left|\partial_x^\alpha (\chi_{s,m} f)(x,\xi) \right| \leq C_{\alpha,s,m} \langle x \rangle^{2s+\delta},
    \]
    where $C_{\alpha,s,m}$ is a positive constant, $\alpha=\left(\alpha_1 \ldots \alpha_n\right)$ be a multi-index and $\langle x \rangle = \left( 1+ |x|^2\right)^{\frac{1}{2}}$.
\end{lemma}
\begin{proof}
    Noting that $f_{ j_1 \dots j_m} \in \Sc (\R^n)$ and $\left|\left(\partial_x^\alpha f_{ j_1 \dots j_m} \right) (x+t\xi)\right| \leq C_{\alpha,N} \langle x+t\xi\rangle^{-N}$ for any $N>2s$, then for each $\xi \in\mathbb S^{n-1}$, we get
    \begin{align*}
        \left| \partial_x^\alpha (\chi_{s,m} f)(x,\xi) \right| &= \left| \int_0^{\infty} t^{2s-1} \partial_x^\alpha f_{i_1 \cdots i_m}(x + t\xi) \xi^{i_1} \cdots \xi^{i_m} \, dt \right|\\
        & \leq  \int_0^{\infty} t^{2s-1} \left| \partial_x^\alpha f_{i_1 \cdots i_m}(x + t\xi)\right| dt  \leq  C_{\alpha,N} \int_0^{\infty} t^{2s-1} \langle x + t\xi\rangle^{-N} dt.
\end{align*}
From Peetre's inequality \cite[p. 17, (2.21)]{Treves:book:1980:1}, we have, for some $C>0$, that
   \begin{align*}      
  \left| \partial_x^\alpha (\chi_{s,m} f)(x,\xi) \right| 
        & \leq C \int_0^{\infty} t^{2s-1} \langle x \rangle^{N} \langle  t\rangle^{-N} dt\\
       & = C \langle x \rangle^{N} \left[ \int_0^{1} t^{2s-1}  \langle  t\rangle^{-N} dt  + \int_1^{\infty} t^{2s-1}  \langle  t\rangle^{-N} dt\right] \\
       & \leq C \langle x \rangle^{N} \left[ \int_0^{1} t^{2s-1}  2^{-\frac{N}{2}} dt  + \int_1^{\infty} t^{2s-1} 2^{-\frac{N}{2}} t^{-N} dt\right]\\
       & \leq C 2^{-\frac{N}{2}} \left[ \frac{1}{2s} + \frac{1}{N-2s}\right]  \langle x \rangle^{N}.
    \end{align*}
    For any $\delta >0$, we choose $N = 2s+\delta$. This proves the lemma.
\end{proof}
We remark from the above Lemma \ref{sec4:Lemma1.2} that, for $f \in \mathcal{S}(\mathbb{R}^n;S^m(\mathbb{R}^n)) $, the fractional divergent beam ray transform $\chi_{s,m} f(\cdot,\xi)$ is a  tempered distribution on $\R^n$ as well as it is a smooth tensor. Thus the Fourier transform of $\chi_{s,m} f(\cdot,\xi)$ in the $\xi$ variable makes sense and therefore it allows one to apply the usual definition of Fourier transform defined for function. See the details in Lemma \ref{stab_ave_momentum}.

We define the following $m+1$ averages of the fractional divergent beam ray transform $\chi_{s,m}f$ for an $m$-tensor $f\in \mathcal{S}(\mathbb{R}^n;S^m(\mathbb{R}^n))$, over the sphere $\mathbb{S}^{n-1}$:
 \[
    \mathcal{A}_{m,s}^{0}, \;
    \mathcal{A}_{m,s}^{1} = (\mathcal{A}_{m,s}^{1,i})_{i=1}^n, \;
    \dots \;,
    \mathcal{A}_{m,s}^{m} = (\mathcal{A}_{m,s}^{m,i_1, \dots,i_m})_{i_1,\dots,i_m=1}^n
 \]
 where
\begin{align*}
    (\mathcal{A}_{m,s}^{0}f)(x) &= c_{n,s}^{m,0} \int_{\mathbb{S}^{n-1}} (\chi_{s,m}f)(x,\xi) dS_{\xi}, \\
    \left(\mathcal{A}_{m,s}^{1,i}f\right)(x) &= c_{n,s}^{m,1} \int_{\mathbb{S}^{n-1}} \xi^i(\chi_{s,m}f)(x,\xi) dS_{\xi},\\
    & \vdots\\
    \left(\mathcal{A}_{m,s}^{m,i_1,\dots,i_m}f\right)(x) &= c_{n,s}^{m,m} \int_{\mathbb{S}^{n-1}} \xi^{i_1}\cdots \xi^{i_m}(\chi_{s,m}f)(x,\xi) dS_{\xi}.
\end{align*}
Here the constants $c_{n,s}^{m,k},\; 0 \leq k \leq m,$ are defined as
\[
    c_{n, s}^{m,k} = \frac{-\Gamma\left(\frac{n+m+k-2s}{2} \right) }{2^{2s - \left\lfloor \frac{m+k}{2} \right\rfloor}\pi^{\frac{n}{2}}\Gamma\left(s + \frac{(m+k)\operatorname{mod}2}{2}\right)}.
\]
\begin{lemma}\label{Sec4:Lemma1}
    For any $0\le k\le m$, we have
    \beq
        \left( \mathcal A^{k, i_1,\cdots,i_k}_{m,s} f\right)(x) = c_{n,s}^{m,k} (-1)^{m+k} \left( |x|^{2s-n-m-k} \; x^{i_1}\cdots x^{i_k} x^{j_1} \cdots x^{j_m} \right) * f_{j_1 \cdots j_m}(x)
    \eeq
\end{lemma}
\begin{proof}
It follows from the definition of the averages of $\chi_{s,m}f$ that
   \[
   \begin{split}
  & \left( c_{n,s}^{m,k} \right)^{-1}\left( \mathcal A^{k, i_1,\cdots,i_k}_{m,s} f\right)(x)
   = \int_{\mathbb{S}^{n-1}} \xi^{i_1}\cdots \xi^{i_k}(\chi_{s,m}f)(x,\xi) dS_{\xi}\\
   &\qquad = \int_{\mathbb{S}^{n-1}}\int_0^{\infty} \xi^{i_1}\cdots \xi^{i_k} t^{2s-1} f_{j_1 \cdots j_m}(x+t \xi) \xi^{j_1} \cdots \xi^{j_m} d t  dS_{\xi}\\
   &\qquad = \int_{\mathbb{S}^{n-1}}\int_0^{\infty} \xi^{i_1}\cdots \xi^{i_k} t^{2s-n} f_{j_1 \cdots j_m}(x+t \xi) \xi^{j_1} \cdots \xi^{j_m} t^{n-1}d t  dS_{\xi}
   \end{split}
   \]
   Simplifying, we obtain
   \[
   \begin{split}
    & \left( c_{n,s}^{m,k} \right)^{-1}\left( \mathcal A^{k, i_1,\cdots,i_k}_{m,s} f\right)(x)  = \int_{\mathbb{R}^{n}} |y|^{2s-n-m-k}y^{i_1}\cdots y^{i_k} y^{j_1} \cdots y^{j_m} f_{j_1 \cdots j_m}(x+y)    dy\\
   &\qquad =  \int_{\mathbb{R}^{n}} |z-x|^{2s-n-m-k} (z-x)^{i_1}\cdots (z-x)^{i_k} (z-x)^{j_1} \cdots (z-x)^{j_m} f_{j_1 \cdots j_m}(z)    dz\\
   &\qquad = (-1)^{m+k} \left( |x|^{2s-n-m-k} \; x^{i_1}\cdots x^{i_k} x^{j_1} \cdots x^{j_m} \right) * f_{j_1 \cdots j_m}(x). \qedhere
   \end{split}
   \]
   
\end{proof}
\begin{lemma}\label{Sec4:Lemma2}
    For $f \in \mathcal{S}(\mathbb{R}^n;S^m(\mathbb{R}^n)) $, the average $\mathcal{A}_{m,s}^{k} f, \;  0 \leq k \leq m$ is a smooth tempered distribution. Specifically, for any given $\delta >0,$
$\mathcal{A}_{m,s}^{k} f$ satisfies the following estimate:
    \[
      \left|\partial_x^\alpha (\mathcal{A}_{m,s}^{k}f)(x) \right| \leq C_{\alpha,s,n,m} \langle x \rangle^{2s+\delta},
    \]
    where $C_{\alpha,s,n,m}$ is a positive constant, $\alpha=\left(\alpha_1 \ldots \alpha_n\right)$ be a multi-index and $\langle x \rangle = \left( 1+ |x|^2\right)^{\frac{1}{2}}$.
    \end{lemma}
\begin{proof}
Using the convolution theorem, we observe that $\mathcal{A}_{m,s}^{k} f$ is smooth and $\partial^\alpha\mathcal{A}_{m,s}^{k} f = \mathcal{A}_{m,s}^{k} \partial^\alpha f$. Noting that $f_{ j_1 \dots j_m} \in \Sc (\R^n)$ and $\left|\left(\partial_x^\alpha f_{ j_1 \dots j_m} \right) (x+t\xi)\right| \leq C_{\alpha,N} \langle x+t\xi\rangle^{-N}$. Thus, for any $N>2s$, we get       
 \begin{align*}
         \left|\partial_x^\alpha (\mathcal{A}_{m,s}^{k,i_1,\cdots, i_k}f)(x) \right| &= c_{n,s}^{m,k} \left| \int_{\mathbb{S}^{n-1}} \int_{0}^{\infty} t^{2s-1} \xi^{i_1}\cdots \xi^{i_k} \left(\partial_x^\alpha f_{ j_1 \dots j_m} \right)(x+t\xi) \xi^{j_1} \dots \xi^{j_m} dt dS_{\xi} \right|\\
        & \leq c_{n,s}^{m,k} {\sum_{j_1,\cdots,j_m=0}^n }\int_{\mathbb{S}^{n-1}} \int_{0}^{\infty} t^{2s-1}  \left|\left(\partial_x^\alpha f_{ j_1 \dots j_m} \right) (x+t\xi)\right| dt dS_{\xi}\\
        & \leq c_{n,s}^{m,k} { n^m} C_\alpha \int_{\mathbb{S}^{n-1}} \int_{0}^{\infty} t^{2s-1}  \langle x+ t \xi\rangle^{-N} dt dS_{\xi}.
    \end{align*}
From Peetre's inequality \cite[p. 17, (2.21)]{Treves:book:1980:1}, we have
    \begin{align*}
 & \left|\partial_x^\alpha (\mathcal{A}_{m,s}^{k,i_1,\cdots, i_k}f)(x) \right|  \leq c_{n,s}^{m,k}2^N  n^mC_\alpha \int_{\mathbb{S}^{n-1}} \int_{0}^{\infty} t^{2s-1}  \langle x\rangle^{N} \langle t \xi\rangle^{-N}dt dS_{\xi}\\
& \qquad= c_{n,s}^{m,k}2^N  n^m  C_\alpha \langle x\rangle^{N} \int_{\mathbb{S}^{n-1}} \int_{0}^{\infty} t^{2s-1}   \langle t \rangle^{-N}dt dS_{\xi}\\
        & \qquad= c_{n,s}^{m,k}2^N  n^m  C_\alpha \langle x\rangle^{N} \int_{\mathbb{S}^{n-1}}  \left[    \int_{0}^{1} t^{2s-1}   \langle t \rangle^{-N}dt  + \int_{1}^{\infty} t^{2s-1}   \langle t \rangle^{-N}dt \right] dS_{\xi}\\
        &\qquad \leq c_{n,s}^{m,k}2^N  n^m  C_\alpha \langle x\rangle^{N} \int_{\mathbb{S}^{n-1}}  \left[    \int_{0}^{1} t^{2s-1}  + \int_{1}^{\infty} t^{2s-1}    t^{-N}dt \right] dS_{\xi}\\
        &\qquad = c_{n,s}^{m,k}2^N  n^m  C_\alpha \langle x\rangle^{N} \frac{2 \pi^{\frac{n}{2}}}{\Gamma(n/2)}\left[ \frac{1}{2s} + \frac{1}{N-2s}  \right].
    \end{align*}
    This completes the lemma.
\end{proof}
We remark that, the first term in the right-hand side of the tensor field expression in Lemma \ref{Sec4:Lemma1} is a function locally integrable over $\mathbb{R}^n$ and bounded for $|x|>1$. This first term can be considered as an element of the space $\mathcal{S}^{\prime}(\mathbb{R}^n)$ of tempered distributions. Indeed, the second term $f_{j_1 \cdots j_m}$ belongs to $\mathcal{S}(\mathbb{R}^n)$. It is well known from \cite{Vladimirov:1979} that, for $u\in \mathcal{S}(\mathbb{R}^n)$ and $v\in \mathcal{S}^{\prime}(\mathbb{R}^n)$, the convolution $u*v$ is defined and belongs to the space of smooth functions whose every derivative increases at most as a polynomial at infinity. In this case the standard formula is valid: $\wh{(u*v)} = \wh{u} * \wh{v} $. Thus the averaging operator 
\[
\mathcal{A}_{m,s}^{k} : \mathcal{S}(\mathbb{R}^n;S^m(\mathbb{R}^n)) \rightarrow C^{\infty}(\mathbb{R}^n;S^m(\mathbb{R}^n))
\]
is a continuous operator. For $f\in \mathcal{S}(\mathbb{R}^n;S^m(\mathbb{R}^n))$, therefore it allows one to apply the usual definition of Fourier transform on the tensor field $\mathcal{A}_{m,s}^{k} f$. See the details in Section \ref{stability_es}.

In the following, we study the reconstruction and unique continuation for the tensor field $f \in \mathcal{S}\left(\mathbb{R}^n ; S^m\left(\mathbb{R}^n\right)\right)$ given the vanishing condition on the averages $\mathcal{A}_{m,s}^{k}, {0} \leq k \leq m$.
 In \cite[Theorem III]{ilmavirta2023unique}, authors have addressed the $m=0$ case with $c_{n,s}^{0,0} = h_n(2s)$ and obtained following unique continuation result.
\begin{theorem}\cite[Theorem III]{ilmavirta2023unique}
    Let $n \geq 2$, $s \in \left( 0, \frac{n}{4}\right)$, $s \notin \mathbb{Z}$ and let $f \in L^{\frac{2n}{n+4s}}\left(\mathbb{R}^n\right)$. If there exists a non-empty open set $U$ in $\mathbb{R}^n$ such that
$$
f=\mathcal{A}_{m,s}^{0} f=0  \text { in } U,  \text { then } f=0 \text { in }  \mathbb{R}^n.
$$
\end{theorem} 
\noindent Our goal is to obtain a similar unique continuation result for fractional divergent beam ray transform of vector fields and symmetric 2-tensor fields. Before proceeding further, we recall important results related to the unique continuation principle and the measurable unique continuation principle for fractional Laplacian operators.
\begin{lemma}\cite[Theorem 2.2]{Kar:Railo:Zimmermann:2023}(Unique continuation for the fractional Laplacian)\label{UCP_FL}
    Let $n \geq 1$ be an integer and $\alpha>0$ with $\alpha \notin \mathbb{Z}$. Let $u \in H^r\left(\mathbb{R}^n\right)$ for some $r \in \mathbb{R}$. If
    $$
        u=(-\Delta)^\alpha u=0 \quad \text { in some open set in } \mathbb{R}^n,
    $$
    then $u \equiv 0$ in $\mathbb{R}^n$.
\end{lemma}
\begin{lemma}[Measurable unique continuation for the fractional Laplacian] \label{measurable_UCP_fractional_Laplace}
Let $n\geq 1$ and $\Omega\subset\mathbb{R}^n$ be an open set. Let $q\in L^{\infty}(\Omega)$, and assume that $u\in H^s(\mathbb{R}^n)$ with $s\in [\frac{1}{4}, 1)$ satisfies
\[
((-\Delta)^s+q)u = 0 \ \ \text{in}\ \ \Omega.
\]
If there exists a measurable set $E\subset \Omega$ with positive measure such that $u=0$ in $E$, then $u\equiv 0$ in $\mathbb{R}^n$.     
\end{lemma}
\noindent A detailed proof of the above lemma can be found in \cite[Proposition 5.1]{Ghosh:Salo:Ruland:Uhlmann2020}. However, a more general result corresponding to the measurable unique continuation principle for fractional conductivity operator was established by Garcia-Ferrero and R\"uland in \cite[Theorem 4]{Garcia:Ruland:2019}. They proved the following result:
\begin{lemma}[Measurable UCP for the fractional Conductivity operator]\label{sharp_measurable_UCP}
Let $\gamma\in \mathbb{R}_{+}\setminus\mathbb{N}$, and let $u\in \operatorname{Dom}((-\nabla\cdot\tilde{a}\nabla)^{\gamma})$ be a solution to 
\[
|(-\nabla\cdot\tilde{a}\nabla)^{\gamma} u(x)|
\leq |q(x)||u(x)| \ \ \text{in}\ \mathbb{R}^n,
\]
where $\tilde{a} = \tilde{a}^{ij}$ is a bounded symmetric and strictly positive-definite matrix with the property that $\tilde{a}\in C_{\rm {loc}}^{2\lfloor \gamma\rfloor,1}(\mathbb{R}^n, \mathbb{R}_{sym}^{n\times n})$ and $\tilde{a}^{ij}(0)=\delta^{ij}$ with $q\in L^{\infty}(\mathbb{R}^n)$. If there exists a measurable set $E\subset \mathbb{R}^n$ with $|E|>0$ and density one at $x_0=0$ such that $u|_{E}=0$, then $u\equiv 0$ in $\mathbb{R}^n$.    
\end{lemma}
\noindent In our case, we only implement the Lemma \ref{sharp_measurable_UCP} for $\tilde{a}=\delta^{ij}, q=0, \gamma=s$, and the measurable UCP holds for all $s\in \mathbb{R}_{+}\setminus\mathbb{N}$. 

\subsection{Vector field case}
Consider a vector field $f = (f_j)_{j=1}^{n} \in \mathcal{S}(\R^n;S^1(\mathbb{R}^n)).$ 
The following theorem provides the reconstruction formula for the vector field $f$ in terms of the associated averages $\Ac_{1,s}^{0} f$ and $\Ac_{1,s}^{1} f  = (\Ac_{1,s}^{1,i} f)_{i=1}^{n}.$
   
\begin{theorem}\label{recon_vec_thm}
Let $s \in (0,1) \setminus  \left\{ \frac{1}{2} \right\}$ and
    $f = (f_j)_{i=1}^{n} \in \mathcal{S}(\R^n; S^1(\mathbb{R}^n))$. Then
    \[
        f_i(x) = (-\Delta)^{s}\left[ -2s \RT_i\Ac_{1,s}^{0} f(x) -
     \mathcal{A}_{1, s}^{1, i} f(x)
    \right], \quad 1 \leq i \leq n.
    \]
\end{theorem}
\noindent To prove the above formula, we need to take the Fourier transform of the averages, as represented in Lemma \ref{Sec4:Lemma1}:
    \Beq\label{sec4:eqn3}
        \begin{aligned}
            \lb \Ac_{1,s}^{0} f\rb(x)&= -c^{1,0}_{n,s} \left(x^{j}|x|^{2s-n-1} \right) * f_{j}(x), \\
            \lb \Ac_{1,s}^{1,i} f\rb(x)&= c_{n, s}^{1,1}  \left(x^ix^j|x|^{2s-n-2} \right) * f_j(x), \quad 1 \leq i \leq n.
        \end{aligned}
     \Eeq
\begin{lemma}\label{fourier_lemm}
    For $s \in \lb 0,1\rb \setminus  \left\{ \frac{1}{2} \right\}$
    and $y \in \R^n\setminus\{0\},$ the Fourier transforms of the averages $\Ac_{1,s}^{0} f$ and $\Ac_{1,s}^{1} f$ are as follows:
    \begin{align*}
        \wh{ \lb \Ac_{1,s}^{0} f\rb}(y) &= |y|^{-2s} \wh{\lb \RT_j f_j\rb}(y),\\
        \wh{( \Ac_{1,s}^{1,i} f)}(y) &= - |y|^{-2s} \wh{f_i}(y) - 2s\rwh{ \lb \RT_i\Ac_{1,s}^{0} f\rb }(y), \quad 1 \leq i \leq n.
    \end{align*}
\end{lemma}
\begin{proof}
    We apply the Fourier transform to the expressions given in \eqref{sec4:eqn3}.
For ${1-2s \notin 2 \mathbb{Z}^+ \cup \{0\}},$ we have
    \Beq
    \begin{aligned}
    \wh{ \lb \Ac_{1,s}^{0} f\rb}(y)&=- \I c^{1,0}_{n,s} \frac{2^{2s-1}\pi^{\frac{n}{2}} \Gamma\left(\frac{2s-1}{2} \right)}{\Gamma\left(\frac{n-2s{+1}}{2}\right)}  \partial_j|y|^{1-2s}\wh{f}_{j}(y)\\
    &=-\I c^{1,0}_{n,s} \frac{ 2^{2s-1}\pi^{\frac{n}{2}} \Gamma\left(\frac{2s-1}{2} \right)}{\Gamma\left(\frac{n-2s{+1}}{2}\right)}  (1-2s) y_j|y|^{-2s-1}\wh{f}_{j}(y)\\
    &=- \I c^{1,0}_{n,s} \lb \frac{ -2^{2s}\pi^{\frac{n}{2}} \left( \frac{2s-1}{2} \right) \Gamma\left(\frac{2s-1}{2} \right)}{\Gamma\left(\frac{n-2s{+1}}{2}\right)} \rb y_j |y|^{-2s-1} \wh{f}_{j}(y)\\
    &= -\I c^{1,0}_{n,s} \lb \frac{ -2^{2s}\pi^{\frac{n}{2}} \Gamma\left(\frac{2s+1}{2} \right)}{\Gamma\left(\frac{n-2s{+1}}{2}\right)} \rb  y_j |y|^{-2s-1} \wh{f}_{j}(y).
    \end{aligned}
    \Eeq
    Substituting the value of $c^{1,0}_{n,s}$, we obtain
    \Beq\label{vec_FT}
    \wh{ \lb \Ac_{1,s}^{0} f\rb}(y) = -\I |y|^{-2s-1} y_j  \wh{f}_j(y) = |y|^{-2s} \wh{\lb \RT_j f_j\rb}(y).
\Eeq
Next, for ${2-2s \notin 2\mathbb{Z}^+ \cup \{0\}}$ we have
\begin{align*}
   & \wh{( \Ac_{1,s}^{1,i} f)}(y)= -c_{n, s}^{1,1}\frac{2^{2s-2}\pi^{n/2}\Gamma(s-1)}{\Gamma(\frac{2+n-2s}{2})}\partial_{ij}|y|^{2-2s}\wh{f_j}(y)\\
    &\qquad =- c_{n, s}^{1,1}
    \frac{2^{2s-2}\pi^{n/2}\Gamma(s-1)}{\Gamma(\frac{2+n-2s}{2})}\partial_{ij}|y|^{2-2s}\wh{f_j}(y) \\
    &\qquad = -c_{n, s}^{1,1}
    \frac{2^{2s-2}\pi^{n/2}\Gamma(s-1)}{\Gamma(\frac{2+n-2s}{2})}(2-2s)\left[\delta_{ij}|y|^{-2s}-2sy_iy_j|y|^{-2-2s}\right]\wh{f_j}(y)\\ 
    &\qquad = - c_{n, s}^{1,1} \left(
    \frac{-2^{2s-1}\pi^{n/2}\Gamma(s)}{\Gamma(\frac{2+n-2s}{2})}\right)
    \left[\wh{f_i}(y)|y|^{-2s}-2s\frac{y_i}{|y|}y_j|y|^{-1-2s}\wh{f_j}(y)\right]\\
    &\qquad = - c_{n, s}^{1,1} \left(
    \frac{-2^{2s-1}\pi^{n/2}\Gamma(s)}{\Gamma(\frac{2+n-2s}{2})}\right)
    \left[ |y|^{-2s} \wh{f_i}(y) + 2s\rwh{ \lb \RT_i\Ac_{1,s}^{0} f\rb}(y)\right] \quad \lb \text{using \ref{vec_FT}} \rb.
\end{align*}
Substituting the value of $c^{1,1}_{n,s}$, we obtain
\Beq\label{vec_FT2}
    \wh{( \Ac_{1,s}^{1,i} f)}(y) = - |y|^{-2s} \wh{f_i}(y) - 2s\rwh{ \lb \RT_i\Ac_{1,s}^{0} f\rb}(y). 
\Eeq
    Note that for $s \in \lb 0,1\rb$, we have
$ 
 {2-2s}, {1-2s \notin 2 \mathbb{Z}^+ \cup \{0\}} \implies { s \notin \left\{\frac{1}{2},1 \right\}}.
$ This completes the proof.
\end{proof}

\begin{proof}[Proof of Theorem \ref{recon_vec_thm}]
 Lemma \ref{fourier_lemm} implies 
\begin{align}\label{vec_expression}
    \wh f_i(y) = |y|^{2s}\left[ -2s\wh{ \lb \RT_i\Ac_{1,s}^{0} f\rb}(y) -
    \wh{\lb \mathcal{A}_{1, s}^{1, i} f\rb}(y)
    \right] , \quad 
    {s \in \lb 0,1\rb \setminus  \left\{ \frac{1}{2} \right\} }.
\end{align}
By taking the inverse Fourier transform, we obtain
\Beq \label{vec_expression2}
    f_i(x) = (-\Delta)^{s}\left[ - 2s \RT_i\Ac_{1,s}^{0} f(x) -
     \mathcal{A}_{1, s}^{1, i} f(x)
    \right] , \quad 
    {s \in \lb 0,1\rb \setminus  \left\{ \frac{1}{2} \right\} }
\Eeq
This completes the proof.
\end{proof}
We now proceed to establish the unique continuation result for the vector field and note that Lemma \ref{UCP_FL} is crucial in this regard.
From \eqref{vec_FT}, we have
$
    |y|^{2s+1} \wh{ \lb \Ac_{1,s}^{0} f\rb}(y) = -\I  y \cdot  \wh{f}(y).
$
That is precisely we obtain
\Beq\label{vec_expression4}
    (-\Delta)^{s+\frac{1}{2}} \Ac_{1,s}^{0} f(x) = \nabla \cdot  {f}(x).
\Eeq
In view of \eqref{vec_FT2}, we have
\[
    |y|^{2s} \wh{( \Ac_{1,s}^{1,i} f)}(y)= -  \wh{f_i}(y) + 2s |y|^{2s} \rwh{ \lb \RT_i\Ac_{1,s}^{0} f\rb}(y),
\]
which is precisely
\Beq\label{vec_expression3}
    (-\Delta)^{s}  \Ac_{1,s}^{1,i} f(x)= -  f_i(x) + 2s (-\Delta)^{s} \RT_i\Ac_{1,s}^{0} f(x). 
\Eeq
\begin{theorem}[Unique continuation for vector fields]
    Let $n \geq 2$ be an integer, and let
    $s \in  \left(0, 1 \right)\setminus \left\{ \frac{1}{2} \right\}$,
    and $f = (f_j)_{j=1}^{n} \in \mathcal{S}(\R^n; S^1(\mathbb{R}^n))$. If there exists a non-empty open set $U \subset \R^n$ such that
    \[
        f = \mathcal{A}_{1,s}^{0}f =  \mathcal{A}_{1,s}^{1} f = 0 \text{ in } U, \text{ then } f=0 \text{ in } \R^n.
    \]
\end{theorem}
\begin{proof}
    We first apply Lemma \ref{UCP_FL} to $\Ac_{1,s}^{0} f.$ In light of \eqref{vec_expression4}, we have
    \[
        \Ac_{1,s}^{0} f(x) = (-\Delta)^{s+\frac{1}{2}} \Ac_{1,s}^{0} f(x) = 0 \quad \text{ in } U.
    \]
    For $s+\frac{1}{2} \notin \Z$, Lemma \ref{UCP_FL} (with $u=\Ac_{1,s}^{0} f$ and $\alpha = s+\frac{1}{2}$) suggests that
    \Beq\label{vec_expression5}
        \Ac_{1,s}^{0} f \equiv 0 \quad \text{ in } \R^n.
    \Eeq
    Substituting \eqref{vec_expression5} in \eqref{vec_expression3}, we obtain $(-\Delta)^{s}  \Ac_{1,s}^{1,i} f(x)= -  f_i(x)$. Using Lemma \ref{UCP_FL} (with $u = \Ac_{1,s}^{1,i} f$ and $s =\alpha$) for $s \notin \Z$, we have
    \[
      \Ac_{1,s}^{1,i} f = (-\Delta)^{s}  \Ac_{1,s}^{1,i} f = 0 \;\; \text{ in } U 
    \]
    which implies $\Ac_{1,s}^{1,i} f = 0$  in $\R^n.$
    In summary, we have $\Ac_{1,s}^{0} f = \Ac_{1,s}^{1,i} f = 0 $ in $\R^n$ for 
    $s \in \lb 0, 1 \rb \setminus  \left\{ \frac{1}{2} \right\} .$
    This fact, along with \eqref{vec_expression2}, it follows that
    $f \equiv 0 \text{ in } \R^n.$
\end{proof}

\begin{theorem}[Measurable unique continuation principle for vector fields]
Let $n\geq 2$ be an integer and $s\in (0,1)\setminus \{\frac{1}{2}\}$ and let $f=(f_j)_{j=1}^{n}\in \mathcal{S}(\mathbb{R}^n;S^1(\mathbb{R}^n))$. Also, assume that $U\subset\mathbb{R}^n$ be any non-empty open set. If $f|_{U}=0$ and there exists a positive measure set $E\subset U$ 
 such that $\Ac_{1,s}^{0} f=\Ac_{1,s}^{1}f =0$ in $E$, then $f\equiv 0$ in $\mathbb{R}^n$.    
\end{theorem}
\begin{proof}
    Since $f=0$ in $U$, it follows from \eqref{vec_expression4} that
    $
    (-\Delta)^{s+\frac{1}{2}}\Ac_{1,s}^{0} f = 0 \ \ \text{in} \ \ U.
    $
By the measurable UCP, Lemma \ref{sharp_measurable_UCP}, we have $\Ac_{1,s}^{0} f = 0$ in $\mathbb{R}^n$. Thus the equation \eqref{vec_expression3} reduces to $(-\Delta)^{s}  \Ac_{1,s}^{1,i} f(x)= -  f_i(x)$ for all $x\in \mathbb{R}^n$. Since $f_{i}|_{U}=0$, we infer that $(-\Delta)^{s}  \Ac_{1,s}^{1,i} f(x)=0$ for all $x\in U$. Note that $\Ac_{1,s}^{1,i} f(x)=0$ in $E$. Applying the measurable UCP, Lemma \ref{sharp_measurable_UCP}, we obtain $\Ac_{1,s}^{1,i} f(x)=0$ in $\mathbb{R}^n$. Finally, it follows from the equation \eqref{vec_expression3} that $f\equiv 0$ in $\mathbb{R}^n$.
\end{proof}
\subsection{Symmetric 2-tensor fields}
Consider a symmetric 2-tensor field $f\in \mathcal{S}(\mathbb{R}^n; S^2(\mathbb{R}^n)).$ Using Lemma \ref{Sec4:Lemma1}, the convolution form of the associated averages are given as follows: for $1 \leq i_1,i_2 \leq n,$
    \Beq\label{sec4:eqn4}
    \begin{aligned}
    \left(\mathcal{A}_{2, s}^0 f\right)(x) &= c_{n, s}^{2,0} \left( |x|^{2s-n-2}x_{j_1}x_{j_2} \right) *  f_{j_1j_2} (x). \\
    \left(\mathcal{A}_{2, s}^{1,i_1} f\right)(x) &= -c_{n, s}^{2,1} \left( |x|^{2s-n-3}x_{i_1}x_{j_1}x_{j_2
    } \right) *  f_{j_1j_2} (x). \\
     \left(\mathcal{A}_{2, s}^{2,i_1,i_2} f\right)(x) &= c_{n, s}^{2,2} \left(|x|^{2s-n-4}x_{i_1} x_{i_2} x_{j_1}x_{j_2
    }\right) *  f_{j_1j_2}(x).
    \end{aligned}
    \Eeq
\noindent To obtain a reconstruction formula for the symmetric 2-tensor field $f$, we first apply the Fourier transform to the averages given in \eqref{sec4:eqn4}.
\begin{lemma}\label{2TF_lem2}
     For 
     $s \in \lb 0,1\rb \setminus  \left\{ \frac{1}{2} \right\}$
     and $y \in \R^n\setminus\{0\},$ the Fourier transforms of the averages $\Ac_{2,s}^{0} f$, $\Ac_{2,s}^{1} f$ and $\Ac_{2,s}^{2} f$ are as follows:
     \begin{align*}
     \wh {\left(\mathcal{A}_{2, s}^0 f\right)}(y) &= - |y|^{-2s}\left[\wh f_{j_1j_1} (y) + 2s \rwh{\lb  \RT_{j_1} \RT_{j_2} f_{j_1j_2} \rb } (y) \right].\\
         \wh {\left(\mathcal{A}_{2, s}^{1,i_1} f\right)} (y) 
        &=  -\rwh{\lb \RT_{i_1} \mathcal{A}_{2, s}^{0} f \rb} (y) + |y|^{-2s} \left[  2 \rwh{ \lb \RT_{j_1} f_{j_1i_1}  \rb} (y)  + \rwh{ \lb \RT_{i_1} \RT_{j_1}\RT_{j_2}f_{j_1j_2} \rb }  (y) \right].\\
        \wh {\left(\mathcal{A}_{2, s}^{2,i_1i_2} f\right)} (y)
        &=-
        2 |y|^{-2s} \wh{f}_{i_1i_2}(y) + \delta_{i_1i_2} \wh {\left(\mathcal{A}_{2, s}^0 f\right)}(y) {-2s} \rwh{ \lb \RT_{i_1}\RT_{i_2} \mathcal{A}_{2, s}^0 f \rb }(y) \\
    & \qquad -2s \lb 
    \rwh{ \lb \RT_{i_1} \mathcal{A}_{2, s}^{1,i_2} f \rb }(y) 
    + \rwh{ \lb \RT_{i_2} \mathcal{A}_{2, s}^{1,i_1} f \rb }(y) \rb.
     \end{align*}
\end{lemma}
\begin{proof}
We apply the Fourier transform to the expression given in \eqref{sec4:eqn4}.

(i) For $ 2-2s \notin 2\mathbb{Z}^+ \cup \{0\}$:
    \begin{align*}
   & \wh {\left(\mathcal{A}_{2, s}^0 f\right)}(y) = -c_{n, s}^{2,0} \frac{2^{2s-2}\pi^{n/2}\Gamma(s-1)}{\Gamma\left(\frac{n+2-2s}{2} \right) }\partial_{j_1j_2}|y|^{2-2s} \wh f_{j_1j_2} (y)\\
     &\qquad = -c_{n, s}^{2,0} \lb \frac{2^{2s-2}\pi^{n/2}\Gamma(s-1)}{\Gamma\left(\frac{n+2-2s}{2} \right) } \rb (2-2s)|y|^{-2s}\left[\delta_{j_1j_2}+(-2s)\frac{y_{j_1}y_{j_2}}{|y|^2}\right] \wh f_{j_1j_2} (y)\\
     &\qquad = -c_{n, s}^{2,0} \lb \frac{-2^{2s-1}\pi^{n/2}\Gamma(s)}{\Gamma\left(\frac{n+2-2s}{2} \right) } \rb |y|^{-2s}\left[\wh f_{j_1j_1} (y) +(-2s)\frac{y_{j_1}y_{j_2}}{|y|^2} \wh f_{j_1j_2} (y) \right] \\
     &\qquad = -c_{n, s}^{2,0} \lb \frac{-2^{2s-1}\pi^{n/2}\Gamma(s)}{\Gamma\left(\frac{n+2-2s}{2} \right) } \rb |y|^{-2s}\left[\wh f_{j_1j_1} (y) + 2s \rwh{\lb  \RT_{j_1} \RT_{j_2} f_{j_1j_2} \rb } (y) \right].
\end{align*}
Substituting the value of $c^{2,0}_{n,s}$, we obtain
\[
   \wh {\left(\mathcal{A}_{2, s}^0 f\right)}(y) = - |y|^{-2s}\left[\wh f_{j_1j_1} (y) + 2s \rwh{\lb  \RT_{j_1} \RT_{j_2} f_{j_1j_2} \rb } (y) \right].
\]
(ii) For $3-2s \notin 2\mathbb{Z}^+ \cup \{0\}$:
\begin{align*}
   & \wh {\left(\mathcal{A}_{2, s}^{1,i_1} f\right)} (y) = ic_{n, s}^{2,1} \frac{2^{2s-3}\pi^{n/2}\Gamma\left(\frac{2s-3}{2}\right)}{\Gamma\left(\frac{n+3-2s}{2} \right) }\partial_{i_1j_1j_2}|y|^{3-2s} \wh f_{j_1j_2} (y)\\
    &\qquad= ic_{n, s}^{2,1} \frac{2^{2s-3}\pi^{n/2}\Gamma\left(\frac{2s-3}{2}\right)}{\Gamma\left(\frac{n+3-2s}{2} \right) } 
    (3-2s)(1-2s)|y|^{-1-2s}\\
&\qquad\qquad \times \left[y_{i_1}\delta_{j_1j_2}+y_{j_1}\delta_{i_1j_2}+y_{j_2}\delta_{i_1j_1}+(-1-2s)y_{i_1}y_{j_1}y_{j_2}|y|^{-2} \right] \wh f_{j_1j_2} (y)\\
&\qquad= -ic_{n, s}^{2,1} \lb \frac{-2^{2s-1}\pi^{n/2}\Gamma\left(\frac{2s+1}{2}\right)}{\Gamma\left(\frac{n+3-2s}{2} \right) } \rb 
    |y|^{-2s} \bigg[ \frac{y_{i_1}}{|y|} \wh f_{j_1j_1} (y) \\
&\qquad\qquad  + \frac{y_{j_1}}{|y|} \wh f_{j_1i_1} (y) + \frac{y_{j_2}}{|y|} \wh f_{i_1j_2} (y) -(1+2s)\frac{y_{i_1}y_{j_1}y_{j_2}}{|y|^{3}} \wh f_{j_1j_2} (y) \bigg]\\
&\qquad= c_{n, s}^{2,1} \lb \frac{-2^{2s-1}\pi^{n/2}\Gamma\left(\frac{2s+1}{2}\right)}{\Gamma\left(\frac{n+3-2s}{2} \right) } \rb 
    |y|^{-2s} \Big[ \rwh{ \lb \RT_{i_1}f_{j_1j_1} \rb } (y)\\
&\qquad\qquad +  \rwh{ \lb \RT_{j_1} f_{j_1i_1}  \rb} (y) +  \rwh{ \lb \RT_{j_2} f_{i_1j_2} \rb }(y) +(1+2s) \rwh{ \lb \RT_{i_1} \RT_{j_1}\RT_{j_2}f_{j_1j_2} \rb }  (y) \Big].
\end{align*}
Substituting the value of $c^{2,1}_{n,s}$, we obtain
\begin{align*}
   & \wh {\left(\mathcal{A}_{2, s}^{1,i_1} f\right)} (y) \\
    &= |y|^{-2s} \left[ \rwh{ \lb \RT_{i_1}f_{j_1j_1} \rb } (y)+  2 \rwh{ \lb \RT_{j_1} f_{j_1i_1}  \rb} (y)  +(1+2s) \rwh{ \lb \RT_{i_1} \RT_{j_1}\RT_{j_2}f_{j_1j_2} \rb }  (y) \right]\\
    &=  -\rwh{\lb \RT_{i_1} \mathcal{A}_{2, s}^{0} f \rb} (y) + |y|^{-2s} \left[  2 \rwh{ \lb \RT_{j_1} f_{j_1i_1}  \rb} (y)  + \rwh{ \lb \RT_{i_1} \RT_{j_1}\RT_{j_2}f_{j_1j_2} \rb }  (y) \right].
\end{align*}
(iii) For $4-2s \notin 2\mathbb{Z}^+ \cup \{0\}$:
\begin{align*}
     \wh {\left(\mathcal{A}_{2, s}^{2,i_1i_2} f\right)} (y) &= c_{n, s}^{2,2} \frac{2^{2s-4}\pi^{n/2}\Gamma\left(\frac{2s-4}{2}\right)}{\Gamma\left(\frac{n+4-2s}{2} \right) }\partial_{i_1i_2j_1j_2}|y|^{4-2s} \wh f_{j_1j_2}(y).
\end{align*}
Note that
\[
\begin{split}
     &   \partial_{i_1i_2j_1j_2}|y|^{4-2s} \wh f_{j_1j_2}(y) \\
    &\quad=(4-2s)(2-2s)|y|^{-2s}\Bigg[\delta_{j_1j_2}\delta_{i_1i_2} +\delta_{j_1i_2}\delta_{i_1j_2}+\delta_{i_1j_1}\delta_{i_2j_2} \\
    &\quad\quad +(-2s)\left(\delta_{j_1j_2}y_{i_1}y_{i_2}+\delta_{j_1i_2}y_{i_1}y_{j_2}+\delta_{i_1j_1}y_{i_2}y_{j_2}+y_{j_1}y_{j_2}\delta_{i_1i_2}+y_{j_1}y_{i_2}\delta_{i_1j_2}+y_{i_1}y_{j_1}\delta_{i_2j_2} \right)\frac{1}{|y|^2}\\
    &\quad\quad +(-2s)(-2s-2) \frac{y_{i_1}y_{i_2}y_{j_1}y_{j_2}}{|y|^4} \Bigg] \wh f_{j_1j_2}(y)\\
    &\quad= (4-2s)(2-2s)|y|^{-2s}\Bigg[\delta_{i_1i_2} \wh{f}_{j_1j_1}+ \wh{f}_{i_2i_1} + \wh{f}_{i_1i_2} \\
    &\quad\quad +(-2s)\left(y_{i_1}y_{i_2} \wh{f}_{j_1j_1} + y_{i_1}y_{j_2} \wh{f}_{i_2j_2}+ y_{i_2}y_{j_2} \wh{f}_{i_1j_2} +y_{j_1}y_{i_2}\wh{f}_{j_1i_1}+y_{i_1}y_{j_1}\wh{f}_{j_1i_2} \right)\frac{1}{|y|^2}\\
     &\quad\quad+(-2s)y_{j_1}y_{j_2}\delta_{i_1i_2}( \wh{f}_{j_1j_2})\frac{1}{|y|^2}+(-2s)(-2s-2) \frac{y_{i_1}y_{i_2}y_{j_1}y_{j_2}}{|y|^4} \wh{f}_{j_1j_2}\Bigg]. 
\end{split}
\]
Simplifying further, we have
\[
\begin{split}
   &   \partial_{i_1i_2j_1j_2}|y|^{4-2s} \wh f_{j_1j_2}(y)  \\
    &\quad= (4-2s)(2-2s)|y|^{-2s}\Bigg[\delta_{i_1i_2} \lb \wh{f}_{j_1j_1} +2s \rwh{\lb \RT_{j_1}\RT_{j_2} f_{j_1j_2} \rb } \rb + 2 \wh{f}_{i_1i_2} \\
    &\quad\quad + 2s\left( \rwh{\lb \RT_{i_1}\RT_{i_2}f_{j_1j_1} \rb} +  \rwh{\lb \RT_{i_1}\RT_{j_2}f_{i_2j_2} \rb}+  \rwh{\lb \RT_{i_2}\RT_{j_2} f_{i_1j_2} \rb} +  \rwh{\lb \RT_{j_1}\RT_{i_2}f_{j_1i_1} \rb } \right)\\
    &\quad\quad +2s\rwh{\lb \RT_{i_1}\RT_{j_1} f_{j_1i_2} \rb}+2s(2+2s) \rwh{\lb \RT_{i_1}\RT_{i_2}\RT_{j_1}\RT_{j_2} f_{j_1j_2} \rb }\Bigg]. 
\end{split}
\]
Hence, we have
\[
\begin{split}
        &\wh {\left(\mathcal{A}_{2, s}^{2,i_1i_2} f\right)} (y) = c_{n, s}^{2,2} \frac{2^{2s-4}\pi^{n/2}\Gamma\left(\frac{2s-4}{2}\right)}{\Gamma\left(\frac{n+4-2s}{2} \right) }\partial_{i_1i_2j_1j_2}|y|^{4-2s} \wh f_{j_1j_2}(y)\\
    &\quad= c_{n, s}^{2,2} \frac{2^{2s-4}\pi^{n/2}\Gamma\left(\frac{2s-4}{2}\right)}{\Gamma\left(\frac{n+4-2s}{2} \right) } (4-2s)(2-2s)
    |y|^{-2s}\Bigg[\delta_{i_1i_2} \lb \wh{f}_{j_1j_1} +2s \rwh{\lb \RT_{j_1}\RT_{j_2} f_{j_1j_2} \rb } \rb + 2 \wh{f}_{i_1i_2} \\
    &\quad\quad + 2s\left( \rwh{\lb \RT_{i_1}\RT_{i_2}f_{j_1j_1} \rb} +  \rwh{\lb \RT_{i_1}\RT_{j_2}f_{i_2j_2} \rb}+  \rwh{\lb \RT_{i_2}\RT_{j_2} f_{i_1j_2} \rb} +  \rwh{\lb \RT_{j_1}\RT_{i_2}f_{j_1i_1} \rb }\right)\\
    &\quad\quad+2s \rwh{\lb \RT_{i_1}\RT_{j_1} f_{j_1i_2} \rb}  +2s(2+2s) \rwh{\lb \RT_{i_1}\RT_{i_2}\RT_{j_1}\RT_{j_2} f_{j_1j_2} \rb }\Bigg] \\
     &\quad= -c_{n, s}^{2,2} \lb \frac{-2^{2s-2}\pi^{n/2}\Gamma\left(s\right)}{\Gamma\left(\frac{n+4-2s}{2} \right) } \rb 
    |y|^{-2s}\Bigg[\delta_{i_1i_2} \lb \wh{f}_{j_1j_1} +2s \rwh{\lb \RT_{j_1}\RT_{j_2} f_{j_1j_2} \rb } \rb + 2 \wh{f}_{i_1i_2} \\
    &\quad\quad + 2s\left( \rwh{\lb \RT_{i_1}\RT_{i_2}f_{j_1j_1} \rb} +2s \rwh{\lb \RT_{i_1}\RT_{i_2}\RT_{j_1}\RT_{j_2} f_{j_1j_2} \rb } \right)\\ 
    &\quad\quad  + 4s\left(   \rwh{\lb \RT_{i_1}\RT_{j_2}f_{i_2j_2} \rb}+  \rwh{\lb \RT_{i_2}\RT_{j_2} f_{i_1j_2} \rb}  
    +\rwh{\lb \RT_{i_1}\RT_{i_2}\RT_{j_1}\RT_{j_2} f_{j_1j_2} \rb } \right) \Bigg]
    \end{split}
    \]
which we simplify further to obtain
\[
\begin{split}
     &\wh {\left(\mathcal{A}_{2, s}^{2,i_1i_2} f\right)} (y)  \\  &\quad= -
    2 |y|^{-2s} \wh{f}_{i_1i_2} + \delta_{i_1i_2} \wh {\left(\mathcal{A}_{2, s}^0 f\right)}(y) + 2s\rwh{ \lb \RT_{i_1}\RT_{i_2} \mathcal{A}_{2, s}^0 f \rb }(y)\\
    &\quad\quad- 4s
    |y|^{-2s}\Bigg[ 
    \left(   \rwh{\lb \RT_{i_1}\RT_{j_2}f_{i_2j_2} \rb}+  \rwh{\lb \RT_{i_2}\RT_{j_2} f_{i_1j_2} \rb} + \rwh{\lb \RT_{i_1}\RT_{i_2}\RT_{j_1}\RT_{j_2} f_{j_1j_2} \rb } \right) \Bigg]\\
    &\quad= -
    2 |y|^{-2s} \wh{f}_{i_1i_2}(y) + \delta_{i_1i_2} \wh {\left(\mathcal{A}_{2, s}^0 f\right)}(y) + 2s\rwh{ \lb \RT_{i_1}\RT_{i_2} \mathcal{A}_{2, s}^0 f \rb }(y) \\
    &\quad\quad - 2s \lb 2\rwh{ \lb \RT_{i_1}\RT_{i_2} \mathcal{A}_{2, s}^0 f \rb }(y)
    + \rwh{ \lb \RT_{i_1} \mathcal{A}_{2, s}^{1,i_2} f \rb }(y) 
    + \rwh{ \lb \RT_{i_2} \mathcal{A}_{2, s}^{1,i_1} f \rb }(y) \rb\\
    &\quad= -
    2 |y|^{-2s} \wh{f}_{i_1i_2}(y) + \delta_{i_1i_2} \wh {\left(\mathcal{A}_{2, s}^0 f\right)}(y) -2s \rwh{ \lb \RT_{i_1}\RT_{i_2} \mathcal{A}_{2, s}^0 f \rb }(y) \\
    &\quad\quad -2s \lb 
    \rwh{ \lb \RT_{i_1} \mathcal{A}_{2, s}^{1,i_2} f \rb }(y) 
    + \rwh{ \lb \RT_{i_2} \mathcal{A}_{2, s}^{1,i_1} f \rb }(y) \rb.
\end{split}
\]
Note that for $s \in \lb 0,1\rb$, we have 
$
2-2s,3-2s,4-2s \notin 2\mathbb{Z}^+ \cup \{0\} $ which implies $ s \notin \left\{ \frac{1}{2}, 1, \frac{3}{2} ,2\right\}.
$
This completes the proof.
\end{proof}

\begin{theorem}
    For 
    $s \in \lb 0,1\rb \setminus  \left\{ \frac{1}{2} \right\}$,
    the reconstruction formula for the symmetric 2-tensor field $f$ in terms of the above averages is given by
     \Beq \label{2TF_exp4}
        \begin{aligned}
    {f}_{i_1i_2}(x) &= \frac{1}{2} (-\Delta)^{s} \Bigg[ 
       \delta_{i_1i_2}  {\left(\mathcal{A}_{2, s}^0 f\right)}(x) {-2s} { \lb \RT_{i_1}\RT_{i_2} \mathcal{A}_{2, s}^0 f \rb }(x) - {\left(\mathcal{A}_{2, s}^{2,i_1i_2} f\right)} (x)  \\
    &\qquad\qquad\qquad\quad - 2s \lb 
     { \lb \RT_{i_1} \mathcal{A}_{2, s}^{1,i_2} f \rb }(x) 
    + { \lb \RT_{i_2} \mathcal{A}_{2, s}^{1,i_1} f \rb }(x) \rb 
    \Bigg].
    \end{aligned}
     \Eeq
\end{theorem}
\begin{proof}
    From the expression for $\wh {\lb \mathcal{A}_{2, s}^{2,i_1i_2} f \rb }$ we have
\begin{align*}
    \wh{f}_{i_1i_2}(y) &= \frac{|y|^{2s}}{2} \Bigg[ 
       \delta_{i_1i_2} \wh {\left(\mathcal{A}_{2, s}^0 f\right)}(y) {-2s} \rwh{ \lb \RT_{i_1}\RT_{i_2} \mathcal{A}_{2, s}^0 f \rb }(y) - \wh {\left(\mathcal{A}_{2, s}^{2,i_1i_2} f\right)} (y)  \\
    &\qquad\qquad\quad - 2s \lb 
     \rwh{ \lb \RT_{i_1} \mathcal{A}_{2, s}^{1,i_2} f \rb }(y) 
    + \rwh{ \lb \RT_{i_2} \mathcal{A}_{2, s}^{1,i_1} f \rb }(y) \rb 
    \Bigg].
\end{align*}
Applying the inverse Fourier transform to the above identity yields \eqref{2TF_exp4}.
\end{proof}
\noindent Note that, from Lemma \ref{2TF_lem2},
\begin{align*}
        \wh {\left(\mathcal{A}_{2, s}^0 f\right)}(y) &=  |y|^{-2s-2}\left[ |y|^2 \wh f_{j_1j_1} (y) +(-2s)y_{j_1}y_{j_2}\wh f_{j_1j_2} (y) \right]\\
        \wh {\left(\mathcal{A}_{2, s}^{1,i_1} f\right)} (y) &= |y|^{-2s-3}\big[ y_{i_1} |y|^2  \wh f_{j_1j_1} (y)+ y_{j_1}|y|^2 \wh f_{j_1i_1} (y)  \\
        & \qquad\qquad + y_{j_2}|y|^2 \wh f_{i_1j_2} (y) -(1+2s)y_{i_1}y_{j_1}y_{j_2} \wh f_{j_1j_2} (y) \big].
\end{align*}
\noindent This implies
\Beq\label{2TF_exp1}
\begin{rcases}
    \begin{aligned}
     |y|^{2s+2} \wh {\left(\mathcal{A}_{2, s}^0 f\right)}(y) &=  |y|^2 \wh f_{j_1j_1} (y) +(-2s)y_{j_1}y_{j_2}\wh f_{j_1j_2} (y) \\
     |y|^{2s+3}\wh {\left(\mathcal{A}_{2, s}^{1,i_1} f\right)} (y) &=  y_{i_1} |y|^2  \wh f_{j_1j_1} (y)+ y_{j_1}|y|^2 \wh f_{j_1i_1} (y) + y_{j_2}|y|^2 \wh f_{i_1j_2} (y) \\
    &\qquad -(1+2s)y_{i_1}y_{j_1}y_{j_2} \wh f_{j_1j_2} (y) 
\end{aligned}
\end{rcases}
 \Eeq 
which is precisely 
\Beq\label{2TF_exp2}
\begin{rcases}
    \begin{aligned}
    (-\Delta)^{s+1}  {\left(\mathcal{A}_{2, s}^0 f\right)}(x) &=  -\Delta  f_{j_1j_1} (x) + 2s \partial_{x_{j_1}} \partial_{x_{j_2}}  f_{j_1j_2} (x)\\
    (-\Delta)^{s+\frac{3}{2}} {\left(\mathcal{A}_{2, s}^{1,i_1} f\right)} (x) &=  \I\partial_{x_{i_1}} (-\Delta)  f_{j_1j_1} (x)+ \I\partial_{x_{j_1}}(-\Delta)  f_{j_1i_1} (x) + \I\partial_{x_{j_2}}(-\Delta)  f_{i_1j_2} (x) \\
    &\qquad +\I (1+2s) \partial_{x_{i_1}} \partial_{x_{j_1}} \partial_{x_{j_2}}  f_{j_1j_2} (x).
\end{aligned}
\end{rcases}
\Eeq 
\begin{theorem}[Unique continuation for 2-tensor fields]\label{2TF_UCP}
    Let $n \geq 2$ be an integer and 
    $s \in  \left(0, 1 \right)\setminus \left\{ \frac{1}{2} \right\}$
    and let $f\in \mathcal{S}(\mathbb{R}^n; S^2(\mathbb{R}^n))$. If there exists a non-empty open set $U \subset \R^n$ such that
    \[
        f = \mathcal{A}_{2,s}^{0}f =  \mathcal{A}_{2,s}^{1} f = \mathcal{A}_{2,s}^{2} f = 0 \text{ in } U, \text{ then } f=0 \text{ in } \R^n.
    \]
\end{theorem}
\begin{proof}
       We use \eqref{2TF_exp2} and obtain
        \[
        \begin{rcases}
            \begin{aligned}
               &\Ac_{2,s}^{0} f = (-\Delta)^{s+1} \Ac_{2,s}^{0} f =0\\
               &\Ac_{2,s}^{1,i_1} f = (-\Delta)^{s+\frac{3}{2}}\Ac_{2,s}^{1,i_1} f = 0
            \end{aligned}
        \end{rcases} \text{ in } U.
        \]
       Now for $s+1, s+\frac{3}{2} \notin \Z,$ Lemma \ref{UCP_FL} suggests that
       \Beq\label{2TF_exp3}
            \Ac_{2,s}^{0} f = \Ac_{2,s}^{1,i_1} f = 0 \quad \text{ 
             in } \R^n.
       \Eeq
       Substituting \eqref{2TF_exp3} in \eqref{2TF_exp4}, we have $(-\Delta)^s\Ac_{2,s}^{2,i_1,i_2} f = 2f_{i_1,i_2}$ and 
       $
           \Ac_{2,s}^{2,i_1,i_2} f = (-\Delta)^s\Ac_{2,s}^{2,i_1,i_2} f = 0 
       $ in $U$.
       Again using Lemma \ref{UCP_FL}, we see that for $s \notin \Z,$ $\Ac_{2,s}^{2,i_1,i_2} f = 0$ in $\R^n.$ This implies that $f_{i_1i_2} = 0$ in $\R^n.$
\end{proof}
\begin{theorem}[Measurable unique continuation principle for $2$-tensor fields]
Let $n\geq 2$ be an integer and $s\in (0,1)\setminus \{\frac{1}{2}\}$ and let $f\in \mathcal{S}(\mathbb{R}^n; S^2(\mathbb{R}^n))$. Also, assume that $U\subset\mathbb{R}^n$ be any non-empty open set. If $f|_{U}=0$ and there exists a positive measure set $E\subset U$
 such that $\Ac_{2,s}^{0} f=\Ac_{2,s}^{1}f =\Ac_{2,s}^{2}f=0$ in $E$. Then $f\equiv 0$ in $\mathbb{R}^n$.    
\end{theorem}
\begin{proof}
Since $f|_{U}=0$, it follows from \eqref{2TF_exp2} that $(-\Delta)^{s+1} \Ac_{2,s}^{0} f =0$ and $(-\Delta)^{s+\frac{3}{2}}\Ac_{2,s}^{1,i_1} f = 0$ in $U$.
By the measurable UCP, Lemma \ref{sharp_measurable_UCP}, we obtain
 \Beq\label{2TF_exp3measurable}
            \Ac_{2,s}^{0} f = \Ac_{2,s}^{1,i_1} f = 0 \quad \text{ 
             in } \R^n.
       \Eeq
 Substituting \eqref{2TF_exp3measurable} in \eqref{2TF_exp4}, we see that $(-\Delta)^s\Ac_{2,s}^{2,i_1,i_2} f = 2f_{i_1,i_2}$ in $\mathbb{R}^n$. Since $f|_{U}=0$ and $\Ac_{2,s}^{2}f=0$ in $E$, it follows from the measurable UCP, Lemma \ref{sharp_measurable_UCP} that $\Ac_{2,s}^{2}f=0$ in $\mathbb{R}^n$. Hence, the equation \eqref{2TF_exp4} implies $f\equiv 0$ in $\mathbb{R}^n$.
\end{proof}

\section{Stability and forward estimates for fractional divergent beam ray transform}\label{stability_es}
In this section, we derive a stability estimate for the fractional divergence beam momentum ray transform for vector fields and 2-tensor fields. First, we derive a stability estimate for the averaging operators, and then we deduce a related estimate for fractional divergence beam momentum ray transform. Additionally, we establish a continuity estimate for the related transform. The basic idea of the proof relies on explicit reconstruction formula derived in Section \ref{sec:fracDBRT} and the boundedness properties of the Riesz potential, Riesz transforms and fractional Laplace operators. 
We will frequently use the following formula. For $\alpha>0,$
\begin{equation}\label{fourier_fractional}
  |y|^{\alpha}\wh{\phi} = \wh{\left( (-\Delta)^{\alpha/2}\phi\right)}   
\end{equation}
holds for all $\phi\in\mathcal{S}'(\mathbb{R}^n).$ 
\subsection{Stability and forward estimate: vector field case}
We start with the forward estimates. In particular, we first restrict our attention to derive boundedness estimates for averaging operators for Schwartz class vector fields. Then we extend the boundedness estimate result for the Sobolev class vector fields using the density argument. These boundedness estimates are essential to deduce the stability estimates for the corresponding averaging operators.   
\begin{theorem}[Boundedness estimate for averaging operators]\label{bdd_ave_schwartz_Lp}
Let $n\geq 2$ be an integer and $s \in \lb 0,1\rb \setminus  \left\{ \frac{1}{2} \right\} $. Assume that $f\in \mathcal{S}(\mathbb{R}^n; S^1(\mathbb{R}^n))$. Then exists a constant $C>0$ such that
\begin{equation}\label{es_bdd_A01ss}
  \|\mathcal{A}_{1,s}^{0}f\|_{H^{t+2s,p}(\mathbb{R}^n; S^1(\mathbb{R}^n))} \leq C \|f\|_{H^{t,p}(\mathbb{R}^n; S^1(\mathbb{R}^n))}  
\end{equation}
and
\[
\|\mathcal{A}_{1,s}^{1}f\|_{H^{t+2s,p}(\mathbb{R}^n; S^1(\mathbb{R}^n))} \leq C \|f\|_{H^{t,p}(\mathbb{R}^n; S^1(\mathbb{R}^n))},
\]
where $1<p<\infty$ and $t\geq 0$.
\end{theorem}
\begin{proof}
In view of Lemma \ref{fourier_lemm}, we have 
\begin{align*}
 (\mathcal{A}_{1,s}^{0}f)(x)
&= (2\pi)^{-n}\int_{\mathbb{R}^n} e^{\mathrm{i} \langle x,\xi\rangle}\wh{(\mathcal{A}_{1,s}^{0}f)}(\xi) d\xi \\
& = (2\pi)^{-n}\int_{\mathbb{R}^n} e^{\mathrm{i} \langle x,\xi\rangle}|\xi|^{-2s} \wh{\lb \RT_j f_j\rb}(\xi) d\xi = (T_{\sigma_s}(\RT_j f_j))(x),
\end{align*}
where $\sigma_s(\xi) = |\xi|^{-2s}$. Applying Theorem \ref{bdd_pseudo_2s} and Theorem \ref{bdd_riesz}, we obtain
\begin{align}
\|\mathcal{A}_{1,s}^{0}f\|_{H^{t+2s,p}(\mathbb{R}^n; S^1(\mathbb{R}^n))}
& = \|T_{\sigma_s}(\RT_j f_j)\|_{H^{t+2s,p}(\mathbb{R}^n)}
\leq C \|\RT_j f_j\|_{H^{t,p}(\mathbb{R}^n)}\notag\\
& \leq C \|f_j\|_{H^{t,p}(\mathbb{R}^n)} \leq C \|f\|_{H^{t,p}(\mathbb{R}^n; S^1(\mathbb{R}^n))}.\label{eq:int_step_1}
\end{align}
Similarly, using the Fourier inversion formula and applying Lemma \ref{fourier_lemm}, we have
\begin{align*}
    (\mathcal{A}_{1,s}^{1,i}f)(x)
&= (2\pi)^{-n}\int_{\mathbb{R}^n} e^{\mathrm{i} \langle x,\xi\rangle}\wh{(\mathcal{A}_{1,s}^{1,i}f)}(\xi) d\xi\\
& = (2\pi)^{-n}\int_{\mathbb{R}^n} e^{\mathrm{i} \langle x,\xi\rangle}\left[- |\xi|^{-2s} \wh{f_i}(\xi) - 2s\rwh{ \lb \RT_i\Ac_{1,s}^{0} f\rb }(\xi)   \right]d\xi \\
&= - T_{\sigma_s}(f_i)(x) - 2s T_0\lb \RT_i\Ac_{1,s}^{0} f\rb (x),
\end{align*}
where  $\sigma_s(\xi) = |\xi|^{-2s}$. Here $T_{\sigma_s}$ is an elliptic pseudodifferential operator of order $-2s$ and $T_0$ is also an elliptic pseudodifferential operator of order $0$. By Theorem \ref{bdd_pseudo_2s}, Theorem \ref{bdd_riesz} and \eqref{eq:int_step_1}, we have
\begin{align*}
 \|\mathcal{A}_{1,s}^{1,i}f\|_{H^{t+2s,p}(\mathbb{R}^n)}
& \leq \|T_{\sigma_s}(f_i)\|_{H^{t+2s,p}(\mathbb{R}^n)} + 2s \|T_0\lb \RT_i\Ac_{1,s}^{0} f\rb\|_{H^{t+2s,p}(\mathbb{R}^n)} \\
& \leq C \|f_i\|_{H^{t,p}(\mathbb{R}^n)} + 2sC \| \RT_i\Ac_{1,s}^{0} f\|_{H^{t+2s,p}(\mathbb{R}^n)} \\
& \leq C \|f_i\|_{H^{t,p}(\mathbb{R}^n)} + 2sC \|\Ac_{1,s}^{0} f\|_{H^{t+2s,p}(\mathbb{R}^n)} \\
& \leq C(1+2s)\|f\|_{H^{t,p}(\mathbb{R}^n; S^1(\mathbb{R}^n))},
\end{align*}
for all $t>-2s, 1<p<\infty$ and $s\in (0,1)\setminus\{\frac{1}{2}\}.$ In other words, we have
\[
 \|\mathcal{A}_{1,s}^{1}f\|_{H^{t+2s,p}(\mathbb{R}^n; S^1(\mathbb{R}^n))}
\leq \sum_{i=1}^{n}   \|\mathcal{A}_{1,s}^{1,i}f\|_{H^{t+2s,p}(\mathbb{R}^n)} \leq C(s,n) \|f\|_{H^{t,p}(\mathbb{R}^n; S^1(\mathbb{R}^n))},
\]
for all $t>-2s, 1<p<\infty$ and $s\in (0,1)\setminus\{\frac{1}{2}\}.$ 
Hence the proof follows.

\end{proof}
We now define averaging operator for $f \in H^{t, p}\left(\mathbb{R}^n ; S^1\left(\mathbb{R}^n\right)\right)$.
Since $ \mathcal{S}(\mathbb{R}^n ; S^1)$ is dense in $H^{t,p}(\mathbb{R}^n ; S^1(\mathbb{R}^n))$, for any given $f\in H^{t,p}(\mathbb{R}^n ; S^1(\mathbb{R}^n))$ there exists a sequence $\{f_{\mathrm{k}}\}\in \mathcal{S}(\mathbb{R}^n ; S^1)$ such that $f_{\mathrm{k}}$ converges to $f$ in $H^{t,p}(\mathbb{R}^n ; S^1(\mathbb{R}^n))$ topology. By Theorem \ref{bdd_ave_schwartz_Lp}, 
\[
 \|\mathcal{A}_{1,s}^{0}f_{\mathrm{k}} -  \mathcal{A}_{1,s}^{0}f_{\mathrm{l}}\|_{H^{t+2s,p}(\mathbb{R}^n; S^1(\mathbb{R}^n))} \leq C \|f_{\mathrm{k}} - f_{\mathrm{l}}\|_{H^{t,p}(\mathbb{R}^n; S^1(\mathbb{R}^n))}  
\]
which implies the sequence $\{\mathcal{A}_{1,s}^{0}f_{\mathrm{k}}\} \ (\mathrm{k} = 1,2,\cdots)$ is a Cauchy sequence in 
$H^{t,p}(\mathbb{R}^n ; S^1(\mathbb{R}^n))$.
Since the Sobolev space $H^{t+2s,p}(\mathbb{R}^n ; S^1(\mathbb{R}^n))$ is complete, the Cauchy sequence converges to some $g\in H^{t+2s,p}(\mathbb{R}^n ; S^1(\mathbb{R}^n)).$ 

For $f\in H^{t,p}(\mathbb{R}^n ; S^1(\mathbb{R}^n))$, we now define $\mathcal{A}_{1,s}^{0}f :=g$. Now, we claim that, the above definition of averaging operator on Sobolev spaces is well-defined. Indeed, if $f_1, f_2 \in H^{t,p}(\mathbb{R}^n ; S^1(\mathbb{R}^n))$ such that $f_1=f_2$ almost everywhere, then we show that $\mathcal{A}_{1,s}^{0}f_1 = \mathcal{A}_{1,s}^{0}f_2$ almost everywhere. To see this, we have, for given $f_1\in H^{t,p}(\mathbb{R}^n ; S^1(\mathbb{R}^n))$, there exists $\{f_{\mathrm{k1}}\}\in \mathcal{S}(\mathbb{R}^n ; S^1)$ such that $f_{\mathrm{k1}} \rightarrow f_1$ in $ H^{t,p}(\mathbb{R}^n ; S^1(\mathbb{R}^n))$ topology. Hence $\{\mathcal{A}_{1,s}^{0}f_{\mathrm{k1}}\}$ is Cauchy by Theorem \ref{bdd_ave_schwartz_Lp} and there exists $g_1$ such that $\mathcal{A}_{1,s}^{0}f_{\mathrm{k1}} \rightarrow g_1$ in  $ H^{t+2s,p}(\mathbb{R}^n ; S^1(\mathbb{R}^n))$ norm. We define $\mathcal{A}_{1,s}^{0}f_{1} := g_1$. Similarly, given $f_2 \in  H^{t,p}(\mathbb{R}^n ; S^1(\mathbb{R}^n))$, there exists $\{f_{\mathrm{k2}}\}\in \mathcal{S}(\mathbb{R}^n ; S^1)$ such that $f_{\mathrm{k2}} \rightarrow f_2$ in $ H^{t,p}(\mathbb{R}^n ; S^1(\mathbb{R}^n))$ norm and eventually by the boundedness theorem \ref{bdd_ave_schwartz_Lp}, $\{\mathcal{A}_{1,s}^{0}f_{\mathrm{k2}}\}$ is Cauchy. Hence it converges to some $g_2$ and we define $\mathcal{A}_{1,s}^{0}f_{2} := g_2$. In view of Theorem \ref{bdd_ave_schwartz_Lp}, for every $\epsilon>0$, there exists a natural number $ N$, such that 
\[
\begin{split}
 &  \|g_2 - g_1\|_{H^{t+2s,p}(\mathbb{R}^n ; S^1(\mathbb{R}^n))}\\
& \qquad=  \|\mathcal{A}_{1,s}^{0}f_{2}-\mathcal{A}_{1,s}^{0}f_{1}\|_{H^{t+2s,p}(\mathbb{R}^n ; S^1(\mathbb{R}^n))} \\
&\qquad \leq \|\mathcal{A}_{1,s}^{0}f_{2}-\mathcal{A}_{1,s}^{0}f_{\mathrm{k2}}\|_{H^{t+2s,p}(\mathbb{R}^n ; S^1(\mathbb{R}^n))} + \|\mathcal{A}_{1,s}^{0}f_{\mathrm{k2}} - \mathcal{A}_{1,s}^{0}f_{\mathrm{k1}}\|_{H^{t+2s,p}(\mathbb{R}^n ; S^1(\mathbb{R}^n))} \\
&\qquad\qquad+ \|\mathcal{A}_{1,s}^{0}f_{\mathrm{k1}}-\mathcal{A}_{1,s}^{0}f_{1}\|_{H^{t+2s,p}(\mathbb{R}^n ; S^1(\mathbb{R}^n))} \\
& \qquad < \epsilon + C \|f_{\mathrm{k2}} - f_{\mathrm{k1}}\|_{H^{t,p}(\mathbb{R}^n ; S^1(\mathbb{R}^n))} \\
&\qquad < \epsilon + C \|f_{\mathrm{k2}} - f_2\|_{H^{t,p}(\mathbb{R}^n ; S^1(\mathbb{R}^n))} + C \|f_1 - f_{\mathrm{k1}}\|_{H^{t,p}(\mathbb{R}^n ; S^1(\mathbb{R}^n))}\\
& \qquad
< \epsilon.
\end{split}
\]
Since, $\epsilon>0$ is arbitrary, we conclude that $g_1=g_2$ a.e. Therefore, the operator $\mathcal{A}_{1,s}^{0} :  H^{t,p}(\mathbb{R}^n ; S^1(\mathbb{R}^n)) \rightarrow  H^{t+2s,p}(\mathbb{R}^n ; S^1(\mathbb{R}^n))$ defined by
$\mathcal{A}_{1,s}^{0} f := g =\lim_{\mathrm{k}\rightarrow\infty}\mathcal{A}_{1,s}^{0}f_{\mathrm{k}}$ is well defined. 

\begin{theorem}[Boundedness estimate for averaging operators on $L^p$ based Sobolev spaces]\label{bdd_sobol_pp}
Let $n\geq 2$ be an integer and $s \in \lb 0,1\rb \setminus  \left\{ \frac{1}{2} \right\} $. Assume that $f\in H^{t,p}(\mathbb{R}^n; S^1(\mathbb{R}^n))$. Then there exists a constant $C>0$ such that
\begin{align*}
      \|\mathcal{A}_{1,s}^{0}f\|_{H^{t+2s,p}(\mathbb{R}^n; S^1(\mathbb{R}^n))} &\leq C \|f\|_{H^{t,p}(\mathbb{R}^n; S^1(\mathbb{R}^n))}
\end{align*}
and 
\begin{align*}
    \|\mathcal{A}_{1,s}^{1}f\|_{H^{t+2s,p}(\mathbb{R}^n; S^1(\mathbb{R}^n))} &\leq C \|f\|_{H^{t,p}(\mathbb{R}^n; S^1(\mathbb{R}^n))},
\end{align*}
where $1<p<\infty$ and $t\geq 0$. Moreover, the operator norm $\|\mathcal{A}_{1, s}^{k}\|_{H^{t,p}\rightarrow H^{t+2s,p}}\leq C$, where $C>0$ is constant, $k=0,1$ and the operator norm is defined as
\[
\|\mathcal{A}_{1, s}^{k}\|_{H^{t,p}\rightarrow H^{t+2s,p}} := \sup_{f\in {H^{t,p}(\mathbb{R}^n; S^1(\mathbb{R}^n))},\ f\neq 0}\frac{\|\mathcal{A}_{1, s}^{k} f\|_{H^{t+2s,p}(\mathbb{R}^n; S^1(\mathbb{R}^n))}}{\|f\|_{H^{t,p}(\mathbb{R}^n; S^1(\mathbb{R}^n))}}.
\]
\end{theorem}
\begin{proof}
For every $\epsilon>0$, there exists a natural number $N$ such that
\[
\begin{split}
    \|\mathcal{A}_{1,s}^{0}f\|_{ H^{t+2s,p}(\mathbb{R}^n ; S^1(\mathbb{R}^n))}
 &   \leq \|\mathcal{A}_{1,s}^{0}f-\mathcal{A}_{1,s}^{0}f_{\mathrm{k}}\|_{ H^{t+2s,p}(\mathbb{R}^n ; S^1(\mathbb{R}^n))} + \|\mathcal{A}_{1,s}^{0}f_{\mathrm{k}}\|_{ H^{t+2s,p}(\mathbb{R}^n ; S^1(\mathbb{R}^n))} \\
 &\leq \epsilon + C\|f_{\mathrm{k}}\|_{ H^{t,p}(\mathbb{R}^n ; S^1(\mathbb{R}^n))} \\
 & \leq \epsilon + C \|f_{\mathrm{k}}-f\|_{ H^{t,p}(\mathbb{R}^n ; S^1(\mathbb{R}^n))} + \|f\|_{ H^{t,p}(\mathbb{R}^n ; S^1(\mathbb{R}^n))} \\
 & \leq\epsilon + C \|f\|_{ H^{t,p}(\mathbb{R}^n ; S^1(\mathbb{R}^n))}
\end{split}
\]
holds for all $f\in H^{t,p}(\mathbb{R}^n ; S^1(\mathbb{R}^n))$ and $\mathrm{k}\geq N$. Since $\epsilon>0$ is arbitrary small, we have
\[
\|\mathcal{A}_{1,s}^{0}f\|_{ H^{t+2s,p}(\mathbb{R}^n ; S^1(\mathbb{R}^n))}\leq C \|f\|_{H^{t,p}(\mathbb{R}^n ; S^1(\mathbb{R}^n))}.
\]
The same reasoning applies to the case for the averaging operator $\mathcal{A}_{1,s}^{1}$ to obtain the boundedness estimates on the Sobolev class vector fields. 
\end{proof}

\begin{theorem}[$L^p-L^q$ Boundedness estimate for averaging operators]\label{LpLqvectror}
Let $n\geq 2$ be an integer and $s \in \lb 0,1\rb \setminus  \left\{ \frac{1}{2} \right\} $. Assume that $f\in \mathcal{S}(\mathbb{R}^n; S^1(\mathbb{R}^n))$. Then exists a constant $C>0$ such that
\begin{equation}\label{es_bdd_A01ss:2}
  \|\mathcal{A}_{1,s}^{0}f\|_{L^q(\mathbb{R}^n; S^1(\mathbb{R}^n))} \leq C \|f\|_{L^p(\mathbb{R}^n; S^1(\mathbb{R}^n))}  
\end{equation}
and
\[
\|\mathcal{A}_{1,s}^{1}f\|_{L^q(\mathbb{R}^n; S^1(\mathbb{R}^n))} \leq C \|f\|_{L^p(\mathbb{R}^n; S^1(\mathbb{R}^n))},
\]
where $p$ and $q$ satisfy $\frac{1}{q} = \frac{1}{p} - \frac{2s}{n}.$
\end{theorem}
\begin{proof}
Let $f = (f_j)_{j=1}^{n} \in \mathcal{S}(\R^n)$. In view of Lemma \ref{fourier_lemm}, we have 
\begin{equation}\label{fourier_A-01}
    \wh{ \lb \Ac_{1,s}^{0} f\rb}(y) = |y|^{-2s} \wh{\lb \RT_j f_j\rb}(y),
\end{equation}
holds for all $s \in \lb 0,1\rb \setminus  \left\{ \frac{1}{2} \right\}$ and $y \in \R^n\setminus\{0\}.$
Since $f_j \in \mathcal{S}(\R^n),$ it is immediate that $f_j\in L^p(\mathbb{R}^n)$ and the fact that the Riesz transform $\RT_j:  L^p(\mathbb{R}^n) \rightarrow L^p(\mathbb{R}^n), 1<p<\infty$ is bounded ensures that $\RT_j f_j\in L^p(\mathbb{R}^n)$ and hence $\RT_j f_j$ is a tempered distribution. Therefore the Fourier transform of $\RT_j f_j$ is well-defined and the above identity makes sense. 
Taking the inverse Fourier transform and using \eqref{Prel:FracLap}, the identity \eqref{fourier_A-01} becomes
\[
\Ac_{1,s}^{0} f = \mathscr{F}^{-1}(|y|^{-2s} \wh{\lb \RT_j f_j\rb}) = (-\Delta)^{-s}\lb \RT_j f_j\rb = I^{2s}\lb \RT_j f_j\rb.
\]
Boundedness properties of Riesz potential operator, see \eqref{estimate_Riesz}, and the fact that the Riesz transform $\RT_j : L^p(\mathbb{R}^n) \rightarrow L^p(\mathbb{R}^n), 1<p<\infty$ is bounded, we obtain 
\begin{align*}
\|\mathcal{A}_{1,s}^{0}f\|_{L^q(\mathbb{R}^n; S^1(\mathbb{R}^n))}
&\leq \sum_{j=1}^{n}\|I^{2s}(\RT_j f_j)\|_{L^q(\mathbb{R}^n)} \\
& \leq C \sum_{j=1}^{n}\|\RT_j f_j\|_{L^p(\mathbb{R}^n)} \leq C \sum_{j=1}^{n}\|f_j\|_{L^p(\mathbb{R}^n)} = C \|f\|_{L^p(\mathbb{R}^n; S^1(\mathbb{R}^n))},
\end{align*}
where $p$ and $q$ satisfy $\frac{1}{q} = \frac{1}{p} - \frac{2s}{n}.$

To prove an estimate for the other averaging operator, we recall from Lemma \ref{fourier_lemm}, that
\[
\wh{( \Ac_{1,s}^{1,i} f)}(y) = - |y|^{-2s} \wh{f_i}(y) - 2s\rwh{ \lb \RT_i\Ac_{1,s}^{0} f\rb }(y), \quad 1 \leq i \leq n.
\]
After performing inverse Fourier transform on the above identity, we get
\[
\Ac_{1,s}^{1,i} f = - I^{2s}(f_i) - 2s \RT_i\Ac_{1,s}^{0} f.
\]
Finally, boundedness properties of Riesz potential and Riesz transform operator ensure the following estimate
\begin{align*}
\|\mathcal{A}_{1,s}^{1,i}f\|_{L^q(\mathbb{R}^n)}
&\leq \|I^{2s}f_i\|_{L^q(\mathbb{R}^n)} + 2s\|\RT_i\Ac_{1,s}^{0} f\|_{L^q(\mathbb{R}^n)} \\
& \leq C \|f_i\|_{L^p(\mathbb{R}^n)}  + C\|\mathcal{A}_{1,s}^{0}f\|_{L^q(\mathbb{R}^n)}\\
& \leq C \|f_i\|_{L^p(\mathbb{R}^n)}
\end{align*}
where $p$ and $q$ satisfy $\frac{1}{q} = \frac{1}{p} - \frac{2s}{n}.$ The last inequality in the above estimate follows from \eqref{es_bdd_A01ss:2}, which completes the proof. 
\end{proof}
\begin{theorem}
    [Boundedness estimate for averaging operators on $L^p-L^q$ based Sobolev spaces]\label{bddLPQSOBO}
Let $n\geq 2$ be an integer and $s \in \lb 0,1\rb \setminus  \left\{ \frac{1}{2} \right\} $. Assume that $f\in H^{t,p}(\mathbb{R}^n; S^1(\mathbb{R}^n))$. Then there exists a constant $C>0$ such that
\begin{equation}\label{es_bdd_A01ss:3}
  \|\mathcal{A}_{1,s}^{0}f\|_{H^{t,q}(\mathbb{R}^n; S^1(\mathbb{R}^n))} \leq C \|f\|_{H^{t,p}(\mathbb{R}^n; S^1(\mathbb{R}^n))},  \ \ \|\mathcal{A}_{1,s}^{1}f\|_{H^{t,q}(\mathbb{R}^n; S^1(\mathbb{R}^n))} \leq C \|f\|_{H^{t,p}(\mathbb{R}^n; S^1(\mathbb{R}^n))}
\end{equation}
where $p$ and $q$ satisfy $\frac{1}{q} = \frac{1}{p} - \frac{2s}{n}$ and $t\geq 0.$
\end{theorem}
\begin{proof}
If $f\in  \mathcal{S}(\mathbb{R}^n; S^1(\mathbb{R}^n))$, then $\partial^{\alpha}f \in  \mathcal{S}(\mathbb{R}^n; S^1(\mathbb{R}^n))$ for all multi-indices $\alpha\in \mathbb{N}^n$ and $\mathcal{A}_{1,s}^{k}f\in  \mathcal{S}'(\mathbb{R}^n; S^1(\mathbb{R}^n))$ for $k=0,1$. By the convolution theorem, we have $\partial^{\alpha}\lb \mathcal{A}_{1,s}^{k}f\rb = \mathcal{A}_{1,s}^{k} \lb \partial^{\alpha}f\rb $. Therefore, by Theorem \ref{LpLqvectror}, we obtain
\begin{align*}
\|\mathcal{A}_{1,s}^{k}f\|_{W^{l,q}(\mathbb{R}^n; S^1(\mathbb{R}^n))}  
& = \sum_{|\alpha|\leq l} \|\partial^{\alpha}\lb \mathcal{A}_{1,s}^{k}f\rb\|_{L^{q}(\mathbb{R}^n; S^1(\mathbb{R}^n))}  
= \sum_{|\alpha|\leq l} \| \mathcal{A}_{1,s}^{k} \lb \partial^{\alpha}f\rb\|_{L^{q}(\mathbb{R}^n; S^1(\mathbb{R}^n))}\\
& \leq C \sum_{|\alpha|\leq l} \| \partial^{\alpha}f\|_{L^{p}(\mathbb{R}^n; S^1(\mathbb{R}^n))} 
= C \|f\|_{W^{l,p}(\mathbb{R}^n; S^1(\mathbb{R}^n))}
\end{align*}
holds for all $l\in \mathbb{N}\cup\{0\}$ and $\frac{1}{q} = \frac{1}{p} - \frac{2s}{n}$. Since $\mathcal{S}(\mathbb{R}^n; S^1(\mathbb{R}^n))$ is dense in $W^{l,q}(\mathbb{R}^n; S^1(\mathbb{R}^n))$, the averaging operator $\mathcal{A}_{1,s}^{k}$ can be extended as a linear operator $ \mathcal{A}_{1,s}^{k} : W^{l,p}(\mathbb{R}^n; S^1(\mathbb{R}^n)) \rightarrow W^{l,q}(\mathbb{R}^n; S^1(\mathbb{R}^n))$ which is bounded. By the interpolation theorem, see Theorem \ref{interpolation}, the operator $\mathcal{A}_{1,s}^{k} : H^{t,p}(\mathbb{R}^n; S^1(\mathbb{R}^n)) \rightarrow H^{t,q}(\mathbb{R}^n; S^1(\mathbb{R}^n))$ is bounded, where $p$ and $q$ satisfy $\frac{1}{q} = \frac{1}{p} - \frac{2s}{n}$ and $t\geq 0$
for $k=0,1.$
\end{proof} 
Next, we prove the corresponding stability estimates.
\begin{theorem}[Stability for averaging operators]\label{stability_ave}
Let $n\geq 2$ be an integer and $s \in \lb 0,1\rb \setminus  \left\{ \frac{1}{2} \right\} $. Assume that $f\in \mathcal{S}(\mathbb{R}^n; S^1(\mathbb{R}^n))$. Then there exists a constant $C>0$ such that
\[
\|f\|_{H^{t,p}(\mathbb{R}^n; S^1\left(\mathbb{R}^n\right))}
\leq C \left[\|\mathcal{A}_{1, s}^{1} f\|_{H^{t+2s,p}(\mathbb{R}^n; S^1\left(\mathbb{R}^n\right))}  + 2s \|\mathcal{A}_{1, s}^{0} f\|_{H^{t+2s,p}(\mathbb{R}^n; S^1\left(\mathbb{R}^n\right))}\right],
\]
holds for all $1<p<\infty$ and $t>-2s$.
\end{theorem}
\begin{proof}
Recall from \eqref{vec_expression} that
\begin{align}\label{vec_expression_stability}
    \wh f_i(y) = |y|^{2s}\left[ -2s\wh{ \lb \RT_i\Ac_{1,s}^{0} f\rb}(y) -
    \wh{\lb \mathcal{A}_{1, s}^{1, i} f\rb}(y)
    \right] , \quad 
    {s \in \lb 0,1\rb \setminus  \left\{ \frac{1}{2} \right\} }
\end{align}
where $f = (f_1, \cdots f_n)$. Combining \eqref{vec_expression_stability} with Minkowski inequality, we have  
\begin{align*}
&\|f\|_{H^{t,p}(\mathbb{R}^n; S^1\left(\mathbb{R}^n\right))}
= \sum_{i=1}^{n}\|f_i\|_{H^{t,p}(\mathbb{R}^n)}\\
&= \sum_{i=1}^{n}\|J_{-t}f_i\|_{L^{p}(\mathbb{R}^n)} = \sum_{i=1}^{n}\|{\wh{}}^{-1}[(1+|\xi|^2)^{\frac{t}{2}}\wh{f_i}]\|_{L^{p}(\mathbb{R}^n)}\\
& \leq  \sum_{i=1}^{n}[2s\|{\wh{}}^{-1}[(1+|\xi|^2)^{\frac{t}{2}}|\xi|^{2s}\wh{\lb \RT_i\Ac_{1,s}^{0} f\rb}]\|_{L^{p}(\mathbb{R}^n)} +  \|{\wh{}}^{-1}[(1+|\xi|^2)^{\frac{t}{2}}|\xi|^{2s}\wh{\lb \mathcal{A}_{1, s}^{1, i} f\rb}]\|_{L^{p}(\mathbb{R}^n)}].
\end{align*}
Taking $\alpha = 2s$, $\phi = \mathcal{A}_{1, s}^{1,i} f$ in the identity \eqref{fourier_fractional} and boundedness properties of the fractional Laplace operator (see \eqref{bdd_frac_Laplace}), we get
\begin{align*}
 &\|{\wh{}}^{-1}[(1+|\xi|^2)^{\frac{t}{2}}|\xi|^{2s}\wh{\lb \mathcal{A}_{1, s}^{1, i} f\rb}]\|_{L^{p}(\mathbb{R}^n)}
  = \|J_{-t}((-\Delta)^s\lb \mathcal{A}_{1, s}^{1, i} f\rb)\|_{L^{p}(\mathbb{R}^n)}\\
 &\qquad = \|(-\Delta)^s\lb \mathcal{A}_{1, s}^{1, i} f\rb\|_{H^{t,p}(\mathbb{R}^n)} \leq C  \| \mathcal{A}_{1, s}^{1, i} f\|_{H^{t+2s,p}(\mathbb{R}^n)}
 \leq C  \| \mathcal{A}_{1, s}^{1} f\|_{H^{t+2s,p}(\mathbb{R}^n; S^1\left(\mathbb{R}^n\right))}.
\end{align*}
Similarly, taking $\alpha = 2s$, $\phi = \RT_i\mathcal{A}_{1, s}^{0} f$ in the identity \eqref{fourier_fractional} together with Theorem \ref{Bdd_frac_LAP} and Theorem \ref{bdd_riesz}, we get
\begin{align*}
 &\|{\wh{}}^{-1}[(1+|\xi|^2)^{\frac{t}{2}}|\xi|^{2s}\wh{\lb \RT_i\mathcal{A}_{1, s}^{0} f\rb}]\|_{L^{p}(\mathbb{R}^n)}
  = \|J_{-t}((-\Delta)^s\lb \RT_i\mathcal{A}_{1, s}^{0} f\rb)\|_{L^{p}(\mathbb{R}^n)}\\
 &\qquad = \|(-\Delta)^s\lb \RT_i\mathcal{A}_{1, s}^{0} f\rb\|_{H^{t,p}(\mathbb{R}^n)} \leq C  \| \RT_i\mathcal{A}_{1, s}^{0} f\|_{H^{t+2s,p}(\mathbb{R}^n)}
 \leq C  \| \mathcal{A}_{1, s}^{0} f\|_{H^{t+2s,p}(\mathbb{R}^n; S^1\left(\mathbb{R}^n\right))}
\end{align*}
holds for all $1<p<\infty$, $s \in \lb 0,1\rb \setminus  \left\{ \frac{1}{2} \right\} $ and $t>-2s$.
Combining all the above estimates, the proof follows.
\end{proof}

\begin{theorem}[Stability for averaging operators for Sobolev class vector fields]
    Let $n\geq 2$ be an integer and $s \in \lb 0,1\rb \setminus  \left\{ \frac{1}{2} \right\} $. Assume that $f\in H^{t,p}(\mathbb{R}^n; S^1(\mathbb{R}^n))$. Then there exists a constant $C>0$ such that
\[
\|f\|_{H^{t,p}(\mathbb{R}^n; S^1\left(\mathbb{R}^n\right))}
\leq C \left[\|\mathcal{A}_{1, s}^{1} f\|_{H^{t+2s,p}(\mathbb{R}^n; S^1\left(\mathbb{R}^n\right))}  + 2s \|\mathcal{A}_{1, s}^{0} f\|_{H^{t+2s,p}(\mathbb{R}^n; S^1\left(\mathbb{R}^n\right))}\right],
\]
holds for all $1<p<\infty$ and $t>-2s$.
\end{theorem}
\begin{proof}
Since $ \mathcal{S}(\mathbb{R}^n ; S^1)$ is dense in $H^{t,p}(\mathbb{R}^n ; S^1(\mathbb{R}^n))$, for any given $f\in H^{t,p}(\mathbb{R}^n ; S^1(\mathbb{R}^n))$ there exists a sequence $\{f_{\mathrm{k}}\}\in \mathcal{S}(\mathbb{R}^n ; S^1)$ such that $f_{\mathrm{k}}$ converges to $f$ in $H^{t,p}(\mathbb{R}^n ; S^1(\mathbb{R}^n))$ topology. In other words, for $\epsilon>0$, there exists a natural number $N$ such that
$\|f_{\mathrm{k}} - f\|_{H^{t,p}(\mathbb{R}^n ; S^1(\mathbb{R}^n))}<\epsilon$ for all $\mathrm{k}\geq N$.
Applying Minkowski inequality and Theorem \ref{stability_ave}, we have
\begin{align*}
 \|f\|_{H^{t,p}(\mathbb{R}^n; S^1\left(\mathbb{R}^n\right))}
&\leq    
  \|f-f_{\mathrm{k}}\|_{H^{t,p}(\mathbb{R}^n; S^1\left(\mathbb{R}^n\right))} + 
  \|f_{\mathrm{k}}\|_{H^{t,p}(\mathbb{R}^n; S^1\left(\mathbb{R}^n\right))}\\
& \leq  \epsilon + C \left[\|\mathcal{A}_{1, s}^{1} f_{\mathrm{k}}\|_{H^{t+2s,p}(\mathbb{R}^n; S^1\left(\mathbb{R}^n\right))}  + \|\mathcal{A}_{1, s}^{0} f_{\mathrm{k}}\|_{H^{t+2s,p}(\mathbb{R}^n; S^1\left(\mathbb{R}^n\right))}\right]\\
& \leq  \epsilon + C [\|\mathcal{A}_{1, s}^{1} f_{\mathrm{k}}-\mathcal{A}_{1, s}^{1} f\|_{H^{t+2s,p}(\mathbb{R}^n; S^1\left(\mathbb{R}^n\right))}+ \|\mathcal{A}_{1, s}^{1} f\|_{H^{t+2s,p}(\mathbb{R}^n; S^1\left(\mathbb{R}^n\right))} \\
&\qquad\qquad+ \|\mathcal{A}_{1, s}^{0} f_{\mathrm{k}}-\mathcal{A}_{1, s}^{0} f\|_{H^{t+2s,p}(\mathbb{R}^n; S^1\left(\mathbb{R}^n\right))}+\|\mathcal{A}_{1, s}^{0} f\|_{H^{t+2s,p}(\mathbb{R}^n; S^1\left(\mathbb{R}^n\right))}  ].
\end{align*}
Note that the operator norm $\|\mathcal{A}_{1, s}^{k}\|_{H^{t,p}\rightarrow H^{t+2s,p}}\leq C$, see Theorem \ref{bdd_sobol_pp}, where $C>0$ is constant and $k=0,1$,
we have
\[
\begin{split}
 \|\mathcal{A}_{1, s}^{k} f_{\mathrm{k}}-\mathcal{A}_{1, s}^{k} f\|_{H^{t+2s,p}(\mathbb{R}^n; S^1 
\left(\mathbb{R}^n\right))}
& \leq \|\mathcal{A}_{1, s}^{k}\|_{H^{t,p}\rightarrow H^{t+2s,p}} \|f-f_{\mathrm{k}}\|_{H^{t,p}(\mathbb{R}^n; S^1 
 \left(\mathbb{R}^n\right))} 
\\& \leq C\epsilon, \ \forall \ \ \mathrm{k}\geq N, \ k=0,1.
\end{split}
\]
Combining all the above estimates, we obtain the required result.
\end{proof}
\begin{theorem}\label{stab_ave_momentum}
Let $n\geq 2$ be an integer and $s \in \lb 0,1\rb \setminus  \left\{ \frac{1}{2} \right\} $. Assume $f\in \mathcal{S}(\mathbb{R}^n; S^m(\mathbb{R}^n))$ such that 
\[
\int_{\mathbb{R}^n} \abs{\mathscr{F}^{-1}\left[(1+|\xi|^2)^{\frac{t+2s}{2}} \mathscr{F}_{\xi}(\chi_{s,m}f)(\xi,\eta)\right]}^pd\xi 
<\infty.
\]
Then 
\begin{align*}
    \|\mathcal{A}_{m,s}^{k,i_1,\dots,i_k}f\|_{H^{t+2s,p}(\mathbb{R}^n)}
& \leq  c_{n,s}^{m,k} \int_{\mathbb{S}^{n-1}} \left[ \int_{\mathbb{R}^n} \abs{\mathscr{F}^{-1}\left[(1+|\xi|^2)^{\frac{t+2s}{2}} \mathscr{F}_{\xi}(\chi_{s,m}f)(\xi,\eta)\right]}^pd\xi\right]^{\frac{1}{p}} dS_{\eta},\\
 \|\mathcal{A}_{m,s}^{0}f\|_{H^{t+2s,p}(\mathbb{R}^n)}
 &\leq  c_{n,s}^{m,0} \int_{\mathbb{S}^{n-1}} \left[ \int_{\mathbb{R}^n} \abs{\mathscr{F}^{-1}\left[(1+|\xi|^2)^{\frac{t+2s}{2}} \mathscr{F}_{\xi}(\chi_{s,m}f)(\xi,\eta)\right]}^pd\xi\right]^{\frac{1}{p}} dS_{\eta},
\end{align*}
  where $\mathscr{F}_{\xi}$ denotes the Fourier transform with respect to $\xi$ variable and $1\le k\le m$.
\end{theorem}
\begin{proof}
We recall $m+1$ averages of the fractional divergent beam ray transform $\chi_{s,m}f$ for an $m$-tensor $f\in \mathcal{S}(\mathbb{R}^n;S^m(\mathbb{R}^n))$, over the sphere $\mathbb{S}^{n-1}$ as
\[
(\chi_{s,m} f)(x,\xi) = \int_0^{\infty} t^{2s-1} f_{i_1 \cdots i_m}(x + t\xi) \xi^{i_1} \cdots \xi^{i_m} \, dt 
\]
and
\[
    \left(\mathcal{A}_{m,s}^{k,i_1,\dots,i_k}f\right)(x) = c_{n,s}^{m,k} \int_{\mathbb{S}^{n-1}} \xi^{i_1}\cdots \xi^{i_k}(\chi_{s,m}f)(x,\xi) dS_{\xi}.
\]
where $x\in\mathbb{R}^n$ and $\xi\in\mathbb{S}^{n-1}$. For $f\in \mathcal{S}(\mathbb{R}^n; S^m(\mathbb{R}^n))$, it is immediate from Lemma \ref{sec4:Lemma1.2}, that 
$\chi_{s,m}f(\cdot,\eta)$ is a smooth tempered distribution on $\mathbb{R}^n$ and that allows us to take Fourier transform of $\chi_{s,m}f$ with respect to $\xi$ variable.
Applying the Fubini theorem, we notice that
\begin{align*}
\mathscr{F}(\mathcal{A}_{m,s}^{k,i_1,\dots,i_k}f)(\xi) 
&= \int_{\mathbb{R}^n} e^{-\mathrm{i} \langle x, \xi \rangle} (\mathcal{A}_{m,s}^{k,i_1,\dots,i_k}f)(x) dx\\
&= c_{n,s}^{m,k}\int_{\mathbb{R}^n} e^{-\mathrm{i} \langle x, \xi \rangle} \left[\int_{\mathbb{S}^{n-1}}\eta^{i_1}\cdots \eta^{i_k}(\chi_{s,m}f)(x,\eta) dS_{\eta} \right] dx \\
& = c_{n,s}^{m,k}\int_{\mathbb{S}^{n-1}} \eta^{i_1}\cdots \eta^{i_k}\left[\int_{\mathbb{R}^n} e^{-\mathrm{i} \langle x, \xi \rangle}(\chi_{s,m}f)(x,\eta) dx \right]dS_{\eta}\\
&= c_{n,s}^{m,k}\int_{\mathbb{S}^{n-1}}\eta^{i_1}\cdots \eta^{i_k}
\mathscr{F}_{\xi}({\chi_{s,m}f})(\xi,\eta)dS_{\eta}.
\end{align*}
We also observe using Fubini theorem that
\begin{align*}
&\mathscr{F}^{-1}\left[(1+|\xi|^2)^{\frac{t+2s}{2}} \mathscr{F}(\mathcal{A}_{m,s}^{k,i_1,\dots,i_k}f)(\xi) \right] = \int_{\mathbb{R}^n} e^{\mathrm{i} \langle \xi, \tau \rangle} (1+|\tau|^2)^{\frac{t+2s}{2}} \mathscr{F}(\mathcal{A}_{m,s}^{k,i_1,\dots,i_k}f)(\tau) d\tau\\
& \qquad= c_{n,s}^{m,k}\int_{\mathbb{R}^n} e^{\mathrm{i} \langle \xi, \tau \rangle} (1+|\tau|^2)^{\frac{t+2s}{2}} \left[ \int_{\mathbb{S}^{n-1}}\eta^{i_1}\cdots \eta^{i_k}
\mathscr{F}_{\xi}({\chi_{s,m}f})(\tau,\eta)dS_{\eta}\right]  d\tau\\
& \qquad=c_{n,s}^{m,k}  \int_{\mathbb{S}^{n-1}} \eta^{i_1}\cdots \eta^{i_k}\left[\int_{\mathbb{R}^n} e^{\mathrm{i} \langle \xi, \tau \rangle} (1+|\tau|^2)^{\frac{t+2s}{2}}
\mathscr{F}_{\xi}({\chi_{s,m}f})(\tau,\eta)d\tau\right]dS_{\eta}\\
& \qquad= c_{n,s}^{m,k}  \int_{\mathbb{S}^{n-1}} \eta^{i_1}\cdots \eta^{i_k}\left[\mathscr{F}^{-1}\left[(1+|\xi|^2)^{\frac{t+2s}{2}} \mathscr{F}_{\xi}(\chi_{s,m}f)(\xi,\eta) \right] \right]dS_{\eta}.
\end{align*}
Finally, using Fubini, Minkowski integral inequality and the fact that $|\eta^{i_k}|\leq 1$, we obtain
\begin{align*}
&\|\mathcal{A}_{m,s}^{k,i_1,\dots,i_k}f\|_{H^{t+2s,p}(\mathbb{R}^n)}= \|J_{-t-2s}\left(\mathcal{A}_{m,s}^{k,i_1,\dots,i_k}f\right)\|_{L^p(\mathbb{R}^n)}\\
& \qquad= \|\mathscr{F}^{-1}\left[(1+|\xi|^2)^{\frac{t+2s}{2}} \mathscr{F}(\mathcal{A}_{m,s}^{k,i_1,\dots,i_k}f)\right]\|_{L^p(\mathbb{R}^n)}\\
& \qquad= \left[\int_{\mathbb{R}^n} |\mathscr{F}^{-1}\left[(1+|\xi|^2)^{\frac{t+2s}{2}} \mathscr{F}(\mathcal{A}_{m,s}^{k,i_1,\dots,i_k}f)(\xi) \right]|^pd\xi\right]^{\frac{1}{p}}\\
& \qquad= c_{n,s}^{m,k} \left[\int_{\mathbb{R}^n} \abs{\int_{\mathbb{S}^{n-1}} \eta^{i_1}\cdots \eta^{i_k}\left[\mathscr{F}^{-1}\left[(1+|\xi|^2)^{\frac{t+2s}{2}} \mathscr{F}_{\xi}(\chi_{s,m}f)(\xi,\eta) \right] \right]dS_{\eta}}^pd\xi\right]^{\frac{1}{p}}
\\
& \qquad\leq  c_{n,s}^{m,k} \int_{\mathbb{S}^{n-1}} \left[ \int_{\mathbb{R}^n} \abs{\mathscr{F}^{-1}\left[(1+|\xi|^2)^{\frac{t+2s}{2}} \mathscr{F}_{\xi}(\chi_{s,m}f)(\xi,\eta)\right]}^pd\xi\right]^{\frac{1}{p}} dS_{\eta}.
\end{align*}
Similarly, we can prove the estimate for operator $\mathcal{A}_{m, s}^0$.
\end{proof}
 \begin{theorem}[Stability for the fractional divergent beam ray transform]
 Let $n\geq 2$ be an integer and $s \in \lb 0,1\rb \setminus  \left\{ \frac{1}{2} \right\} $. Assume $f\in \mathcal{S}(\mathbb{R}^n; S^1(\mathbb{R}^n))$ such that
 \[
\int_{\mathbb{R}^n} \abs{\mathscr{F}^{-1}\left[(1+|\xi|^2)^{\frac{t+2s}{2}} \mathscr{F}_{\xi}(\chi_{s,1}f)(\xi,\eta)\right]}^pd\xi 
<\infty.
\]
Then there exists a constant $C>0$ such that
 \[
 \|f\|_{H^{t,p}(\mathbb{R}^n; S^1(\mathbb{R}^n))}
 \leq C \int_{\mathbb{S}^{n-1}}\left[ \int_{\mathbb{R}^n} \abs{\mathscr{F}^{-1}\left[(1+|\xi|^2)^{\frac{t+2s}{2}} \mathscr{F}_{\xi}(\chi_{s,1}f)(\xi,\eta)\right]}^pd\xi\right]^{\frac{1}{p}}dS_{\eta},
 \]    
  where $\mathscr{F}_{\xi}$ denotes the Fourier transform with respect to $\xi$ variable.
 \end{theorem}
 \begin{proof}
The proof follows from Theorem \ref{stab_ave_momentum} and Theorem \ref{stability_ave}.
\end{proof}

\subsection{Stability and forward estimate: 2-tensor field case}
We start with the boundedness estimates for the averaging operators as it requires to prove the stability estimates for the averaging operators.
\begin{theorem}[Boundedness estimate for averaging operators]\label{es_bdd_A01ss_2ten_sobo}
Let $n\geq 2$ be an integer and $s \in \lb 0,1\rb \setminus  \left\{ \frac{1}{2} \right\} $. Assume that $f\in \mathcal{S}(\mathbb{R}^n; S^2(\mathbb{R}^n))$. Then there exists a constant $C>0$ such that
\[
  \|\mathcal{A}_{2,s}^{0}f\|_{H^{t+2s,p}(\mathbb{R}^n; S^2(\mathbb{R}^n))} \leq C \|f\|_{H^{t,p}(\mathbb{R}^n; S^2(\mathbb{R}^n))},  \ \ \|\mathcal{A}_{2,s}^{1}f\|_{H^{t+2s,p}(\mathbb{R}^n; S^2(\mathbb{R}^n))} \leq C \|f\|_{H^{t,p}(\mathbb{R}^n; S^2(\mathbb{R}^n))}
\]
and
\[
 \|\mathcal{A}_{2,s}^{2}f\|_{H^{t+2s,p}(\mathbb{R}^n; S^2(\mathbb{R}^n))} \leq C \|f\|_{H^{t,p}(\mathbb{R}^n; S^2(\mathbb{R}^n))},
\]
where $1<p<\infty$ and $t>0$.
\end{theorem}
\begin{proof}
By the Fourier inversion formula and applying Lemma \ref{2TF_lem2}, we can write 
\begin{align*}
\left(\mathcal{A}_{2,s}^{0}f  \right)(x)
&= (2\pi)^{-n}\int_{\mathbb{R}^n} e^{\mathrm{i} \langle x,\xi\rangle}\wh{(\mathcal{A}_{2,s}^{0}f)}(\xi) d\xi \\
& = (2\pi)^{-n}\int_{\mathbb{R}^n}  e^{\mathrm{i} \langle x,\xi\rangle} \left[- |\xi|^{-2s}\left[\wh f_{j_1j_1} (\xi) + 2s \rwh{\lb  \RT_{j_1} \RT_{j_2} f_{j_1j_2} \rb } (\xi) \right] \right]d \xi\\
&= -T_{\sigma_s}(f_{j_1j_1})(x) - 2s T_{\sigma_s}\lb  \RT_{j_1} \RT_{j_2} f_{j_1j_2} \rb (x),
\end{align*}
where $\sigma_s(\xi) = |\xi|^{-2s}$ and $T_{\sigma_s}$ is an elliptic pseudodifferential operator of order $-2s$. By Theorem \ref{bdd_pseudo_2s} and Theorem \ref{bdd_riesz}, we obtain
\begin{align}\label{eq:int_step_2}
    \begin{split}
\|\mathcal{A}_{2,s}^{0}f\|_{H^{t+2s,p}(\mathbb{R}^n)} 
&\leq \|T_{\sigma_s}(f_{j_1j_1})\|_{H^{t+2s,p}(\mathbb{R}^n)} +2s \|T_{\sigma_s}\lb  \RT_{j_1} \RT_{j_2} f_{j_1j_2} \rb\|_{H^{t+2s,p}(\mathbb{R}^n)}\\
& \leq C\left[\|f_{j_1j_1}\|_{H^{t,p}(\mathbb{R}^n)} + \|\RT_{j_1} \RT_{j_2} f_{j_1j_2}\|_{H^{t,p}(\mathbb{R}^n)}\right]\\
& \leq C \|f\|_{H^{t,p}(\mathbb{R}^n)}
  \end{split}
\end{align}
where $C$ depends on $n,p,s$.
Next we estimate the term $\mathcal{A}_{2,s}^{1}f$ in the Bessel potential spaces. By the Fourier inversion formula and Lemma \ref{2TF_lem2}, we can write
\begin{align*}
&\left(\mathcal{A}_{2,s}^{1,i_1}f  \right)(x) \\
&= (2\pi)^{-n}\int_{\mathbb{R}^n} e^{\mathrm{i} \langle x,\xi\rangle}\wh{(\mathcal{A}_{2,s}^{1,i_1}f)}(\xi) d\xi \\
& =(2\pi)^{-n}\int_{\mathbb{R}^n} e^{\mathrm{i} \langle x,\xi\rangle}\left[ -\rwh{\lb \RT_{i_1} \mathcal{A}_{2, s}^{0} f \rb} (\xi) + |\xi|^{-2s} \left[  2 \rwh{ \lb \RT_{j_1} f_{j_1i_1}  \rb} (\xi)  + \rwh{ \lb \RT_{i_1} \RT_{j_1}\RT_{j_2}f_{j_1j_2} \rb }  (\xi) \right]\right]d\xi \\
& = - \lb \RT_{i_1} \mathcal{A}_{2, s}^{0} f \rb(x)+ 2T_{\sigma_s} \lb \RT_{j_1} f_{j_1i_1}  \rb(x) + T_{\sigma_s} {\lb \RT_{i_1} \RT_{j_1}\RT_{j_2}f_{j_1j_2} \rb}(x)
\end{align*}
where  $\sigma_s(\xi) = |\xi|^{-2s}$. Here $T_{\sigma_s}$ is an elliptic pseudodifferential operator of order $-2s$. By Theorem \ref{bdd_pseudo_2s}, Theorem \ref{bdd_riesz} and \eqref{eq:int_step_2}, we have
\begin{align*}
&\|\mathcal{A}_{2,s}^{1,i_1}f \|_{H^{t+2s,p}(\mathbb{R}^n)} \\
& \leq \|\RT_{i_1} \mathcal{A}_{2, s}^{0} f \|_{H^{t+2s,p}(\mathbb{R}^n)} + 2\|T_{\sigma_s} \lb \RT_{j_1} f_{j_1i_1}\rb\|_{H^{t+2s,p}(\mathbb{R}^n)} + \|T_{\sigma_s} {\lb \RT_{i_1} \RT_{j_1}\RT_{j_2}f_{j_1j_2}\rb}\|_{H^{t+2s,p}(\mathbb{R}^n)} \\
& \leq C \left[\|\mathcal{A}_{2, s}^{0} f\|_{H^{t+2s,p}(\mathbb{R}^n)}
+\|\RT_{j_1} f_{j_1i_1}\|_{H^{t,p}(\mathbb{R}^n)} + \|\RT_{i_1} \RT_{j_1}\RT_{j_2}f_{j_1j_2}\|_{H^{t,p}(\mathbb{R}^n)}\right] \\
& \leq C(n,s) \|f\|_{H^{t,p}(\mathbb{R}^n)}
\end{align*}
holds for all $t>0, 1<p<\infty$ and $s\in (0,1)\setminus\{\frac{1}{2}\}$. Finally by the Fourier inversion formula and Lemma \ref{2TF_lem2}, we have
\begin{align*}
    &\left(\mathcal{A}_{2,s}^{2,i_1,i_2}f  \right)(x) \\
&= (2\pi)^{-n}\int_{\mathbb{R}^n} e^{\mathrm{i} \langle x,\xi\rangle}\wh{(\mathcal{A}_{2,s}^{2,i_1,i_2}f)}(\xi) d\xi \\
& =(2\pi)^{-n}\int_{\mathbb{R}^n} e^{\mathrm{i} \langle x,\xi\rangle}[ -
        2 |\xi|^{-2s} \wh{f}_{i_1i_2}(\xi) + \delta_{i_1i_2} \wh {\left(\mathcal{A}_{2, s}^0 f\right)}(\xi) {-2s} \rwh{ \lb \RT_{i_1}\RT_{i_2} \mathcal{A}_{2, s}^0 f \rb }(\xi) \\
    & \qquad \qquad\qquad -2s \lb 
    \rwh{ \lb \RT_{i_1} \mathcal{A}_{2, s}^{1,i_2} f \rb }(\xi) 
    + \rwh{ \lb \RT_{i_2} \mathcal{A}_{2, s}^{1,i_1} f \rb }(\xi) \rb]d\xi \\
    & = -2T_{\sigma_s}\lb {f}_{i_1i_2}\rb(x) + \delta_{i_1i_2} \left(\mathcal{A}_{2, s}^0 f\right)(x)
    -2s \lb \RT_{i_1}\RT_{i_2} \mathcal{A}_{2, s}^0 f \rb(x) \\
    &\qquad  -2s \left[\lb \RT_{i_1} \mathcal{A}_{2, s}^{1,i_2} f \rb (x) + \lb \RT_{i_2} \mathcal{A}_{2, s}^{1,i_1} f \rb (x) \right]
\end{align*}
where $\sigma_s(\xi)=|\xi|^{-2s}$ and $T_{\sigma_s}$ is an elliptic pseudodifferential operator of order $-2s$. By the boundedness properties of the elliptic pseudodifferential operator, see Theorem \ref{bdd_pseudo_2s}, and the boudedness properties of the Riesz transform, see Theorem \ref{bdd_riesz}, we have
\begin{align*}
&\|\mathcal{A}_{2,s}^{2,i_1,i_2}f \|_{H^{t+2s,p}(\mathbb{R}^n)} \\
& \leq 2\|T_{\sigma_s}\lb {f}_{i_1i_2}\rb\|_{H^{t+2s,p}(\mathbb{R}^n)}+ \|\mathcal{A}_{2, s}^0 f\|_{H^{t+2s,p}(\mathbb{R}^n)} +2s \| \RT_{i_1}\RT_{i_2} \mathcal{A}_{2, s}^0 f\|_{H^{t+2s,p}(\mathbb{R}^n)} \\
&\qquad + 2s\|\RT_{i_1} \mathcal{A}_{2, s}^{1,i_2} f\|_{H^{t+2s,p}(\mathbb{R}^n)} + 2s\|\RT_{i_2} \mathcal{A}_{2, s}^{1,i_1} f\|_{H^{t+2s,p}(\mathbb{R}^n)} \\
& \leq 2C\|{f}_{i_1i_2}\|_{H^{t,p}(\mathbb{R}^n)} + C\|f\|_{H^{t,p}(\mathbb{R}^n)} + 2sC\|\mathcal{A}_{2, s}^0 f\|_{H^{t+2s,p}(\mathbb{R}^n)} \\
&\qquad + 2sC\|\mathcal{A}_{2, s}^{1,i_2} f\|_{H^{t+2s,p}(\mathbb{R}^n)} + 2sC\|\mathcal{A}_{2, s}^{1,i_1} f\|_{H^{t+2s,p}(\mathbb{R}^n)} \\
& \leq C(n,s) \|f\|_{H^{t,p}(\mathbb{R}^n)}
\end{align*}
holds for all $t>0, 1<p<\infty$ and $s\in (0,1)\setminus\{\frac{1}{2}\}$. Hence the theorem follows.
\end{proof}
\begin{theorem}[Boundedness estimate for averaging operators on $L^p$ based Sobolev spaces]\label{es_bdd_A01ss_2ten_sobo:2}
Let $n\geq 2$ be an integer and $s \in \lb 0,1\rb \setminus  \left\{ \frac{1}{2} \right\} $. Assume that $f\in H^{t,p}(\mathbb{R}^n; S^2(\mathbb{R}^n))$. Then exists a constant $C>0$ such that
\[
  \|\mathcal{A}_{2,s}^{0}f\|_{H^{t+2s,p}(\mathbb{R}^n; S^2(\mathbb{R}^n))} \leq C \|f\|_{H^{t,p}(\mathbb{R}^n; S^2(\mathbb{R}^n))},  \ \ \|\mathcal{A}_{2,s}^{1}f\|_{H^{t+2s,p}(\mathbb{R}^n; S^2(\mathbb{R}^n))} \leq C \|f\|_{H^{t,p}(\mathbb{R}^n; S^2(\mathbb{R}^n))}
\]
and
\[
 \|\mathcal{A}_{2,s}^{2}f\|_{H^{t+2s,p}(\mathbb{R}^n; S^2(\mathbb{R}^n))} \leq C \|f\|_{H^{t,p}(\mathbb{R}^n; S^2(\mathbb{R}^n))},
\]
where $1<p<\infty$ and $t\geq 0$. Moreover, the operator norm $\|\mathcal{A}_{2, s}^{k}\|_{H^{t,p}\rightarrow H^{t+2s,p}}\leq C$, where $C>0$ is constant, $k=0,1,2$ and the operator norm is defined as
\[
\|\mathcal{A}_{2, s}^{k}\|_{H^{t,p}\rightarrow H^{t+2s,p}} := \sup_{f\in {H^{t,p}(\mathbb{R}^n; S^2(\mathbb{R}^n))},\ f\neq 0}\frac{\|\mathcal{A}_{2, s}^{k} f\|_{H^{t+2s,p}(\mathbb{R}^n; S^2(\mathbb{R}^n))}}{\|f\|_{H^{t,p}(\mathbb{R}^n; S^2(\mathbb{R}^n))}}.
\]
\end{theorem}
\begin{proof}
Since $ \mathcal{S}(\mathbb{R}^n ; S^2)$ is dense in $H^{t,p}(\mathbb{R}^n ; S^2(\mathbb{R}^n))$, for any given $f\in H^{t,p}(\mathbb{R}^n ; S^2(\mathbb{R}^n))$ there exists a sequence $\{f_{\mathrm{k}}\}\in \mathcal{S}(\mathbb{R}^n ; S^2)$ such that $f_{\mathrm{k}}$ converges to $f$ in $H^{t,p}(\mathbb{R}^n ; S^2(\mathbb{R}^n))$ topology. By Theorem \ref{bdd_ave_schwartz_Lp}, sequence $\{\mathcal{A}_{2,s}^{0}f_{\mathrm{k}}\} \ (\mathrm{k} = 1,2,\cdots)$ is a Cauchy sequence in 
$H^{t,p}(\mathbb{R}^n ; S^2(\mathbb{R}^n))$ and hence it converges to some $g\in H^{t,p}(\mathbb{R}^n ; S^2(\mathbb{R}^n)).$ For $f\in H^{t,p}(\mathbb{R}^n ; S^2(\mathbb{R}^n))$, we now define $\mathcal{A}_{2,s}^{0}f :=g$. It is now easy to check that the definition of averaging operator on Sobolev spaces is well-defines and the operator satisfies the required boundedness estimate for the vector fields with Sobolev class. Estimates corresponding to the other averaging operators follow by the similar argument.      
\end{proof}

\begin{theorem}[$L^p-L^q$ Boundedness estimate for averaging operators]\label{bdd_2tensor}
Let $n\geq 2$ be an integer and $s \in \lb 0,1\rb \setminus  \left\{ \frac{1}{2} \right\} $. Assume that $f\in \mathcal{S}(\mathbb{R}^n, S^2(\mathbb{R}^n))$. Then there exists a constant $C>0$ such that
\begin{equation}\label{es_bdd_A01s}
  \|\mathcal{A}_{2,s}^{0}f\|_{L^q(\mathbb{R}^n)} \leq C \|f\|_{L^p(\mathbb{R}^n)},  \ \ \ \|\mathcal{A}_{2,s}^{1}f\|_{L^q(\mathbb{R}^n)} \leq C \|f\|_{L^p(\mathbb{R}^n)},
\end{equation}
and
\[
\|\mathcal{A}_{2,s}^{2}f\|_{L^q(\mathbb{R}^n)} \leq C \|f\|_{L^p(\mathbb{R}^n)},
\]
where $p$ and $q$ satisfy $\frac{1}{q} = \frac{1}{p} - \frac{2s}{n}$ and $C=C(s,n)$ be a positive constant depending on $n$ and $s$.
\end{theorem}

\begin{proof}
Let $f = \left( f_{i_1i_2}\right)_{i_1,i_2=1}^{n}$ be a symmetric $2$-tensor field such that $f\in \mathcal{S}(\mathbb{R}^n, S^2(\mathbb{R}^n))$.
In view of Lemma \ref{2TF_lem2}, we have
\begin{equation}\label{bdd_ave11}
 \wh {\left(\mathcal{A}_{2, s}^0 f\right)}(y) = - |y|^{-2s}\left[\wh f_{j_1j_1} (y) + 2s \rwh{\lb  \RT_{j_1} \RT_{j_2} f_{j_1j_2} \rb } (y) \right] 
\end{equation}
holds for all $s \in \lb 0,1\rb \setminus  \left\{ \frac{1}{2} \right\} $ and $y\in\mathbb{R}^n\setminus\{0\}.$ Combining \eqref{Prel:FracLap} together with applying inverse Fourier transform on the equation \eqref{bdd_ave11}, we have
\begin{align*}
 \mathcal{A}_{2, s}^0 f(y)
 & = - \mathscr{F}^{-1}\left[  |y|^{-2s}\wh {f_{j_1j_1}} (y)\right] -2s \mathscr{F}^{-1}\left[ |y|^{-2s} \rwh{\lb  \RT_{j_1} \RT_{j_2} f_{j_1j_2} \rb } (y) \right] \\
 & = - (-\Delta)^{-s} {f_{j_1j_1}}- 2s (-\Delta)^{-s}\lb  \RT_{j_1} \RT_{j_2} f_{j_1j_2} \rb   \\
 & = - I^{2s}\lb f_{j_1j_1}\rb -2s I^{2s} \lb  \RT_{j_1} \RT_{j_2} f_{j_1j_2} \rb.
\end{align*}
Since, the Riesz potential is bounded from $L^p$ to $L^q$, see \eqref{estimate_Riesz}, and the Riesz transform $\RT_j : L^p(\mathbb{R}^n) \rightarrow L^p(\mathbb{R}^n), 1<p<\infty$ is bounded, it is immediate to see that
\begin{align*}
    \|\mathcal{A}_{2,s}^{0}f\|_{L^q(\mathbb{R}^n)}
&  \leq C \|I^{2s}\lb f_{j_1j_1}\rb\|_{L^q(\mathbb{R}^n)} + 2s\|I^{2s} \lb  \RT_{j_1} \RT_{j_2} f_{j_1j_2} \rb\|_{L^q(\mathbb{R}^n)}  \\
& \leq C \|f_{j_1j_1}\|_{L^p(\mathbb{R}^n)}  + 2sC\|\RT_{j_1} \RT_{j_2} f_{j_1j_2} \|_{L^p(\mathbb{R}^n)} \\
& \leq C \|f\|_{L^p(\mathbb{R}^n)} + 2sC\|f\|_{L^p(\mathbb{R}^n)} \\
& = C(1+2s)\|f\|_{L^p(\mathbb{R}^n)}, 
\end{align*}
where $p, q$ satisfy $\frac{1}{q}=\frac{1}{p}-\frac{2s}{n}$ and $C$ depends on $n.$
To provide an estimate for other averaging operators, let us recall the expression for $\mathcal{A}_{2,s}^{1}f$ for Lemma \ref{2TF_lem2}:
\begin{align*}
     & \wh {\left(\mathcal{A}_{2, s}^{1,i_1} f\right)} (y) \\
&\qquad =  -\rwh{\lb \RT_{i_1} \mathcal{A}_{2, s}^{0} f \rb} (y) + |y|^{-2s} \left[  2 \rwh{ \lb \RT_{j_1} f_{j_1i_1}  \rb} (y)  + \rwh{ \lb \RT_{i_1} \RT_{j_1}\RT_{j_2}f_{j_1j_2} \rb }  (y) \right].
\end{align*}
Taking the inverse Fourier transform to the above identity, we have
\begin{align*}
\mathcal{A}_{2,s}^{1,i_1}f
& = -  \RT_{i_1} \lb\mathcal{A}_{2, s}^{0} f \rb + \mathscr{F}^{-1}\left[ |y|^{-2s} \left[  2 \rwh{ \lb \RT_{j_1} f_{j_1i_1}  \rb} (y)  + \rwh{ \lb \RT_{i_1} \RT_{j_1}\RT_{j_2}f_{j_1j_2} \rb }  (y) \right]\right] \\
& = -  \RT_{i_1} \lb\mathcal{A}_{2, s}^{0} f \rb + 2 (-\Delta)^{-s}\lb \RT_{j_1} f_{j_1i_1}\rb + (-\Delta)^{-s}\lb \RT_{i_1} \RT_{j_1}\RT_{j_2}f_{j_1j_2} \rb  \\
& = - \RT_{i_1} \lb\mathcal{A}_{2, s}^{0} f \rb + 2 I^{2s}\lb \RT_{j_1} f_{j_1i_1}\rb + I^{2s}\lb \RT_{i_1} \RT_{j_1}\RT_{j_2}f_{j_1j_2} \rb. 
\end{align*}
By the boundedness properties of Riesz potential operator and Riesz transform, we infer that
\begin{align*}
&\|\mathcal{A}_{2,s}^{1}f\|_{L^q(\mathbb{R}^n)} 
 = \sum_{i_1=1}^{n} \|\mathcal{A}_{2,s}^{1, i_1}f\|_{L^q(\mathbb{R}^n)} \\
& \leq \sum_{i_1=1}^{n} \|\RT_{i_1} \lb\mathcal{A}_{2, s}^{0} f \rb\|_{L^q(\mathbb{R}^n)} +  \sum_{i_1, j_1=1}^{n} \|I^{2s}\lb \RT_{j_1} f_{j_1i_1}\rb\|_{L^q(\mathbb{R}^n)}\\
&\qquad + \sum_{i_1, j_1, j_2=1}^{n} \|I^{2s}\lb \RT_{i_1} \RT_{j_1}\RT_{j_2}f_{j_1j_2} \rb\|_{L^q(\mathbb{R}^n)} \\
& \leq C \|\mathcal{A}_{2, s}^{0} f \|_{L^q(\mathbb{R}^n)}  + 2C\sum_{i_1, j_1=1}^{n} \| \RT_{j_1} f_{j_1i_1}\|_{L^p(\mathbb{R}^n)} + C\sum_{i_1, j_1,j_2=1}^{n} \|\RT_{i_1} \RT_{j_1}\RT_{j_2}f_{j_1j_2} \|_{L^p(\mathbb{R}^n)} \\
& \leq C(1+2s)\|f\|_{L^p(\mathbb{R}^n)} +2C\sum_{i_1,j_1=1}^{n} \|f_{j_1i_1}\|_{L^p(\mathbb{R}^n)} + C\sum_{j_1,j_2=1}^{n} \|f_{j_1j_2}\|_{L^p(\mathbb{R}^n)} \\
& \leq C(1+2s)\|f\|_{L^p(\mathbb{R}^n)}.
\end{align*}
Finally to estimate the term $\mathcal{A}_{2,s}^{2}f$, we recall from Lemma \ref{2TF_lem2}, that
\begin{align*}
    \wh {\left(\mathcal{A}_{2, s}^{2,i_1i_2} f\right)} (y)
        &=-
        2 |y|^{-2s} \wh{f}_{i_1i_2}(y) + \delta_{i_1i_2} \wh {\left(\mathcal{A}_{2, s}^0 f\right)}(y) {-2s} \rwh{ \lb \RT_{i_1}\RT_{i_2} \mathcal{A}_{2, s}^0 f \rb }(y) \\
    & \qquad -2s \lb 
    \rwh{ \lb \RT_{i_1} \mathcal{A}_{2, s}^{1,i_2} f \rb }(y) 
    + \rwh{ \lb \RT_{i_2} \mathcal{A}_{2, s}^{1,i_1} f \rb }(y) \rb. 
\end{align*}
Taking inverse Fourier transform to the above identity, we get
\begin{align*}
  \mathcal{A}_{2, s}^{2,i_1i_2} f
  & =- 2 \mathscr{F}^{-1}\left[ |y|^{-2s} \wh{f}_{i_1i_2}(y)\right] +  \delta_{i_1i_2}\left(\mathcal{A}_{2, s}^0 f\right)(y) {-2s} \lb \RT_{i_1}\RT_{i_2} \mathcal{A}_{2, s}^0 f \rb (y) \\
  & \qquad -2s \left[ { \lb\RT_{i_1} \mathcal{A}_{2, s}^{1,i_2} \rb }f (y) +  { \lb \RT_{i_2} \mathcal{A}_{2, s}^{1,i_1} f \rb }(y)\right] \\
  & = -2 (-\Delta)^{s}\lb f_{i_1i_2}\rb(y) + \delta_{i_1i_2}\left(\mathcal{A}_{2, s}^0 f\right)(y) {-2s} \lb \RT_{i_1}\RT_{i_2} \mathcal{A}_{2, s}^0 f \rb (y) \\
  & \qquad -2s \left[ { \lb\RT_{i_1} \mathcal{A}_{2, s}^{1,i_2} \rb }f (y) +  { \lb \RT_{i_2} \mathcal{A}_{2, s}^{1,i_1} f \rb }(y)\right]. 
\end{align*}
By the boundedness properties of Riesz potential operator and Riesz transform, we get
\begin{align*}
 &\|\mathcal{A}_{2,s}^{2}f\|_{L^q(\mathbb{R}^n)} 
 = \sum_{i_1, i_2=1}^{n} \|\mathcal{A}_{2,s}^{1, i_1, i_2}f\|_{L^q(\mathbb{R}^n)} \\
& \leq 2\sum_{i_1, i_2=1}^{n} \|I^{2s}({f}_{i_1i_2})\|_{L^q(\mathbb{R}^n)} + \sum_{i_1, i_2=1}^{n} \|\mathcal{A}_{2, s}^0 f\|_{L^q(\mathbb{R}^n)} + 2s\sum_{i_1, i_2=1}^{n} \|\RT_{i_1}\RT_{i_2} \mathcal{A}_{2, s}^0 f\|_{L^q(\mathbb{R}^n)} \\
&\qquad + 2s\sum_{i_1, i_2=1}^{n} \left[\|\RT_{i_1} \mathcal{A}_{2, s}^{1,i_2} \|_{L^q(\mathbb{R}^n)} + \|\RT_{i_2} \mathcal{A}_{2, s}^{1,i_1} f\|_{L^q(\mathbb{R}^n)} \right]  \\
& \leq 2C\sum_{i_1, i_2=1}^{n} \|{f}_{i_1i_2}\|_{L^p(\mathbb{R}^n)} + n^2 \|\mathcal{A}_{2, s}^0 f\|_{L^q(\mathbb{R}^n)} + 2sC\sum_{i_1, i_2=1}^{n} \|\mathcal{A}_{2, s}^0 f\|_{L^q(\mathbb{R}^n)}  \\
&\qquad + \sum_{i_1, i_2=1}^{n} \left[\|\mathcal{A}_{2, s}^{1,i_2} \|_{L^q(\mathbb{R}^n)} + \| \mathcal{A}_{2, s}^{1,i_1} f\|_{L^q(\mathbb{R}^n)} \right]  \\
& \leq C\|f\|_{L^p(\mathbb{R}^n)} + C(1+2s)\|f\|_{L^p(\mathbb{R}^n)} + Cs(1+2s)n^2\|f\|_{L^p(\mathbb{R}^n)}+ Csn\|f\|_{L^q(\mathbb{R}^n)}\\
& = C(s,n) \|f\|_{L^p(\mathbb{R}^n)}
\end{align*}
where $C(s,n)$ be a positive constant depending on $n$ and $s$ and $p, q$ satisfy $\frac{1}{q}=\frac{1}{p}-\frac{2s}{n}$.    
\end{proof}
\begin{theorem}
    [Boundedness estimate for averaging operators on $L^p-L^q$ based Sobolev spaces]
Let $n\geq 2$ be an integer and $s \in \lb 0,1\rb \setminus  \left\{ \frac{1}{2} \right\} $. Assume that $f\in H^{t,p}(\mathbb{R}^n, S^2(\mathbb{R}^n))$. Then there exists a constant $C>0$ such that
\begin{equation*}
  \|\mathcal{A}_{2,s}^{2}f\|_{H^{t,q}(\mathbb{R}^n)} \leq C \|f\|_{H^{t,p}(\mathbb{R}^n)},  \ \ \|\mathcal{A}_{2,s}^{1}f\|_{H^{t,q}(\mathbb{R}^n)} \leq C \|f\|_{H^{t,p}(\mathbb{R}^n)}
\end{equation*}
and
\[
\|\mathcal{A}_{2,s}^{0}f\|_{H^{t,q}(\mathbb{R}^n)} \leq C \|f\|_{H^{t,p}(\mathbb{R}^n)}
\]
where $p$ and $q$ satisfy $\frac{1}{q} = \frac{1}{p} - \frac{2s}{n}$ and $t\geq 0.$
\end{theorem}
\begin{proof}
The proof follows from interpolation argument, Theorem \ref{bdd_2tensor} and argument similar to the proof of Theorem \ref{bddLPQSOBO}.    
\end{proof}

Now, we begin with the stability estimates for the averaging operators for Schwartz class $2$-tensor fields and Sobolev class tensor fields.
\begin{theorem}[Stability for averaging operator]\label{stability_ave_2 tensor}
Let $n\geq 2$ be an integer and $s \in \lb 0,1\rb \setminus  \left\{ \frac{1}{2} \right\} $. Assume that $f\in \mathcal{S}(\mathbb{R}^n; S^2(\mathbb{R}^n))$. Then there exists a constant $C>0$ such that
\[
\begin{split}
 &\|f\|_{H^{t,p}(\mathbb{R}^n; S^2(\mathbb{R}^n))} \\
&\quad  \leq C(n,s)\left[  \|\mathcal{A}_{2, s}^{2} f\|_{H^{t+2s,p}(\mathbb{R}^n; S^2(\mathbb{R}^n))} +  \|\mathcal{A}_{2, s}^{1} f\|_{H^{t+2s,p}(\mathbb{R}^n; S^2(\mathbb{R}^n))} + \|\mathcal{A}_{2, s}^{0} f\|_{H^{t+2s,p}(\mathbb{R}^n; S^2(\mathbb{R}^n))}\right],
\end{split}
\]
holds for all $1<p<\infty$, $t>-2s$.
\end{theorem}
\begin{proof}
Recall from Lemma \ref{2TF_lem2}, for $ \in \lb 0,1\rb \setminus  \left\{ \frac{1}{2} \right\} $ and $y\in\mathbb{R}^n\setminus\{0\}$, we have
\begin{align*}
    \wh {\left(\mathcal{A}_{2, s}^{2,i_1i_2} f\right)} (y)
        &=-
        2 |y|^{-2s} \wh{f}_{i_1i_2}(y) + \delta_{i_1i_2} \wh {\left(\mathcal{A}_{2, s}^0 f\right)}(y) {-2s} \rwh{ \lb \RT_{i_1}\RT_{i_2} \mathcal{A}_{2, s}^0 f \rb }(y) \\
    & \qquad -2s \lb 
    \rwh{ \lb \RT_{i_1} \mathcal{A}_{2, s}^{1,i_2} f \rb }(y) 
    + \rwh{ \lb \RT_{i_2} \mathcal{A}_{2, s}^{1,i_1} f \rb }(y) \rb. 
\end{align*}
By Minkowski inequality, we obtain
\begin{align*}
&\|f\|_{H^{t,p}(\mathbb{R}^n; S^2(\mathbb{R}^n))} = \sum_{i_1, i_2=1}^{n} \|f_{i_1, i_2}\|_{H^{t,p}(\mathbb{R}^n)}\\
&= \sum_{i_1, i_2=1}^{n}\|J_{-t}f_{i_1, i_2}\|_{L^{p}(\mathbb{R}^n)} = \sum_{i_1, i_2=1}^{n}\|{\wh{}}^{-1}[(1+|\xi|^2)^{\frac{t}{2}}\wh{f_{i_1, i_2}}]\|_{L^{p}(\mathbb{R}^n)}\\
&\leq C \sum_{i_1, i_2=1}^{n}[ \|{\wh{}}^{-1}[(1+|\xi|^2)^{\frac{t}{2}}|\xi|^{2s}\wh{(\mathcal{A}_{2, s}^{2,i_1i_2} f)}]\|_{L^{p}(\mathbb{R}^n)} + \|{\wh{}}^{-1}[(1+|\xi|^2)^{\frac{t}{2}}|\xi|^{2s}\delta_{i_1i_2} \wh {\left(\mathcal{A}_{2, s}^0 f\right)}]\|_{L^{p}(\mathbb{R}^n)} \\
&\qquad+\|{\wh{}}^{-1}[(1+|\xi|^2)^{\frac{t}{2}}|\xi|^{2s}\rwh{ \lb \RT_{i_1}\RT_{i_2} \mathcal{A}_{2, s}^0 f \rb }]\|_{L^{p}(\mathbb{R}^n)}+\|{\wh{}}^{-1}[(1+|\xi|^2)^{\frac{t}{2}}|\xi|^{2s} \rwh{ \lb \RT_{i_1} \mathcal{A}_{2, s}^{1,i_2} f \rb }]\|_{L^{p}(\mathbb{R}^n)}
\\
&\qquad \qquad+\|{\wh{}}^{-1}[(1+|\xi|^2)^{\frac{t}{2}}|\xi|^{2s} \rwh{ \lb \RT_{i_2} \mathcal{A}_{2, s}^{1,i_1} f \rb }]\|_{L^{p}(\mathbb{R}^n)} ]\\
&:= C \sum_{i_1, i_2=1}^{n}[ I +II+III+IV+V].
\end{align*}
We will now estimate all the above terms one by one.
Taking $\alpha = 2s$, $\phi=\mathcal{A}_{2, s}^{2,i_1i_2} f$ in the identity \eqref{fourier_fractional} and boundedness properties of the fractional Laplace operator (see \eqref{bdd_frac_Laplace}), we get
\begin{align*}
I 
&=   \|{\wh{}}^{-1}[(1+|\xi|^2)^{\frac{t}{2}}|\xi|^{2s}\wh{(\mathcal{A}_{2, s}^{2,i_1i_2} f)}]\|_{L^{p}(\mathbb{R}^n)}  \\
& =\|{\wh{}}^{-1}[(1+|\xi|^2)^{\frac{t}{2}}\wh {\left((-\Delta)^{s}\left( \mathcal{A}_{2, s}^{2,i_1i_2} f\right) \right)}]\|_{L^{p}(\mathbb{R}^n)} 
 = \|J_{-t}\left((-\Delta)^{s}\left( \mathcal{A}_{2, s}^{2,i_1i_2} f\right) \right)\|_{L^p(\mathbb{R}^n)} \\
& = \|(-\Delta)^{s}\left( \mathcal{A}_{2, s}^{2,i_1i_2} f \right)\|_{H^{t,p}(\mathbb{R}^n)}
\leq C \|\mathcal{A}_{2, s}^{2,i_1i_2} f\|_{H^{t+2s,p}(\mathbb{R}^n)}.
\end{align*}
Similarly, taking $\alpha = 2s$, $\phi=\mathcal{A}_{2, s}^{0} f$ in the identity \eqref{fourier_fractional} and boundedness properties of the fractional Laplace operator (see \eqref{bdd_frac_Laplace}), we get
\begin{align*}
II 
&= \|{\wh{}}^{-1}[(1+|\xi|^2)^{\frac{t}{2}}|\xi|^{2s}\delta_{i_1i_2} \wh {\left(\mathcal{A}_{2, s}^0 f\right)}]\|_{L^{p}(\mathbb{R}^n)} \\
& = \|{\wh{}}^{-1}[(1+|\xi|^2)^{\frac{t}{2}}\delta_{i_1i_2} \wh{\left((-\Delta)^{s}\left( \mathcal{A}_{2, s}^{0} f\right) \right) }]\|_{L^{p}(\mathbb{R}^n)} 
\leq \|J_{-t}\left((-\Delta)^{s}\left( \mathcal{A}_{2, s}^{0} f\right) \right)\|_{L^p(\mathbb{R}^n)} \\
& = \|(-\Delta)^{s}\left( \mathcal{A}_{2, s}^{0} f \right)\|_{H^{t,p}(\mathbb{R}^n)}
\leq C \|\mathcal{A}_{2, s}^{0} f\|_{H^{t+2s,p}(\mathbb{R}^n)}.
\end{align*}
Similarly, taking $\alpha = 2s$, $\phi = \RT_{i_1}\RT_{i_2}\mathcal{A}_{2, s}^{0} f$ in the identity \eqref{fourier_fractional} together with Theorem \ref{Bdd_frac_LAP} and Theorem \ref{bdd_riesz}, we get
\begin{align*}
III
&= \|{\wh{}}^{-1}[(1+|\xi|^2)^{\frac{t}{2}}|\xi|^{2s}\rwh{ \lb \RT_{i_1}\RT_{i_2} \mathcal{A}_{2, s}^0 f \rb }]\|_{L^{p}(\mathbb{R}^n)}\\
& =  \|{\wh{}}^{-1}[(1+|\xi|^2)^{\frac{t}{2}}\wh{\left((-\Delta)^{s}\lb \RT_{i_1}\RT_{i_2} \mathcal{A}_{2, s}^0 f \rb  \right) }]\|_{L^{p}(\mathbb{R}^n)} \\
&= \|J_{-t}\left((-\Delta)^{s}\lb \RT_{i_1}\RT_{i_2} \mathcal{A}_{2, s}^0 f \rb \right)\|_{L^p(\mathbb{R}^n)} 
 = \|(-\Delta)^{s}\lb \RT_{i_1}\RT_{i_2} \mathcal{A}_{2, s}^0 f \rb \|_{H^{t,p}(\mathbb{R}^n)}\\
&\leq C \| \RT_{i_1}\RT_{i_2} \mathcal{A}_{2, s}^0 f \|_{H^{t+2s,p}(\mathbb{R}^n)}
\leq  C \|\mathcal{A}_{2, s}^0 f \|_{H^{t+2s,p}(\mathbb{R}^n)}, \ \ \text{for all}\ t>-2s, 1<p<\infty.
\end{align*}
Similarly, taking $\alpha = 2s$, $\phi = \RT_{i_1}\mathcal{A}_{2, s}^{1,i_2} f$ in the identity \eqref{fourier_fractional} together with Theorem \ref{Bdd_frac_LAP} and Theorem \ref{bdd_riesz}, we get
\begin{align*}
IV
&= \|{\wh{}}^{-1}[(1+|\xi|^2)^{\frac{t}{2}}|\xi|^{2s}\rwh{ \lb \RT_{i_1}\mathcal{A}_{2, s}^{1,i_2} f \rb }]\|_{L^{p}(\mathbb{R}^n)}\\
& =  \|{\wh{}}^{-1}[(1+|\xi|^2)^{\frac{t}{2}}\wh{\left((-\Delta)^{s}\lb \RT_{i_1}\mathcal{A}_{2, s}^{1,i_2} f \rb  \right) }]\|_{L^{p}(\mathbb{R}^n)} \\
&= \|J_{-t}\left((-\Delta)^{s}\lb \RT_{i_1}\mathcal{A}_{2, s}^{1,i_2} f\rb \right)\|_{L^p(\mathbb{R}^n)} 
 = \|(-\Delta)^{s}\lb \RT_{i_1}\mathcal{A}_{2, s}^{1,i_2} f\rb \|_{H^{t,p}(\mathbb{R}^n)}\\
&\leq C \| \RT_{i_1}\mathcal{A}_{2, s}^{1,i_2} f \|_{H^{t+2s,p}(\mathbb{R}^n)}
\leq  C \|\mathcal{A}_{2, s}^{1,i_2} f\|_{H^{t+2s,p}(\mathbb{R}^n)}, \ \ \text{for all}\ t>-2s, 1<p<\infty.
\end{align*}
Finally, taking $\alpha = 2s$, $\phi = \RT_{i_2}\mathcal{A}_{2, s}^{1,i_1} f$ in the identity \eqref{fourier_fractional} together with Theorem \ref{Bdd_frac_LAP} and Theorem \ref{bdd_riesz}, we get
\begin{align*}
V
&= \|{\wh{}}^{-1}[(1+|\xi|^2)^{\frac{t}{2}}|\xi|^{2s}\rwh{ \lb \RT_{i_2}\mathcal{A}_{2, s}^{1,i_1} f \rb }]\|_{L^{p}(\mathbb{R}^n)}\\
& =  \|{\wh{}}^{-1}[(1+|\xi|^2)^{\frac{t}{2}}\wh{\left((-\Delta)^{s}\lb \RT_{i_2}\mathcal{A}_{2, s}^{1,i_1} f \rb  \right) }]\|_{L^{p}(\mathbb{R}^n)} \\
&= \|J_{-t}\left((-\Delta)^{s}\lb \RT_{i_2}\mathcal{A}_{2, s}^{1,i_1} f\rb \right)\|_{L^p(\mathbb{R}^n)} 
 = \|(-\Delta)^{s}\lb \RT_{i_2}\mathcal{A}_{2, s}^{1,i_1} f\rb \|_{H^{t,p}(\mathbb{R}^n)}\\
&\leq C \| \RT_{i_2}\mathcal{A}_{2, s}^{1,i_1} f \|_{H^{t+2s,p}(\mathbb{R}^n)}
\leq  C \|\mathcal{A}_{2, s}^{1,i_1} f\|_{H^{t+2s,p}(\mathbb{R}^n)}, \ \ \text{for all}\ t>-2s, 1<p<\infty.
\end{align*} 
Combining all these estimates, we obtain the required result.
\end{proof}
\begin{theorem}
    [Stability for averaging operators for Sobolev class $2$-tensor fields]\label{stability_ave_2 tensor_sobolev}
Let $n\geq 2$ be an integer and $s \in \lb 0,1\rb \setminus  \left\{ \frac{1}{2} \right\} $. Assume that $f\in H^{t,p}(\mathbb{R}^n; S^2(\mathbb{R}^n))$. Then there exists a constant $C>0$ such that
\[
\begin{split}
 &\|f\|_{H^{t,p}(\mathbb{R}^n; S^2(\mathbb{R}^n))} \\
&  \leq C(n,s)\left[  \|\mathcal{A}_{2, s}^{2} f\|_{H^{t+2s,p}(\mathbb{R}^n; S^2(\mathbb{R}^n))} +  \|\mathcal{A}_{2, s}^{1} f\|_{H^{t+2s,p}(\mathbb{R}^n; S^2(\mathbb{R}^n))} + \|\mathcal{A}_{2, s}^{0} f\|_{H^{t+2s,p}(\mathbb{R}^n; S^2(\mathbb{R}^n))}\right],
  \end{split}
\]
holds for all $1<p<\infty$, $t>-2s$.
\end{theorem}
\begin{proof}
Since $\mathcal{S}(\mathbb{R}^n; S^2(\mathbb{R}^n))$ is dense in $H^{t,p}(\mathbb{R}^n; S^2(\mathbb{R}^n))$, it follows that there exists a sequence $\{f_{\mathrm{k}}\}\in \mathcal{S}(\mathbb{R}^n; S^2(\mathbb{R}^n))$ such that $f_{\mathrm{k}}$ converges to $f$ in $H^{t,p}(\mathbb{R}^n, S^2(\mathbb{R}^n))$ topology. In other words, for $\epsilon>0$, there exists a natural number $N$ such that
$\|f_{\mathrm{k}} - f\|_{H^{t,p}(\mathbb{R}^n ; S^2(\mathbb{R}^n))}<\epsilon$ for all $\mathrm{k}\geq N$.
Applying Minkowski inequality and Theorem \ref{stability_ave_2 tensor}, we have
\begin{align*}
 &\|f\|_{H^{t,p}(\mathbb{R}^n; S^2(\mathbb{R}^n))}
\leq    
  \|f-f_{\mathrm{k}}\|_{H^{t,p}(\mathbb{R}^n; S^2(\mathbb{R}^n))} + 
  \|f_{\mathrm{k}}\|_{H^{t,p}(\mathbb{R}^n; S^2(\mathbb{R}^n))}\\
& \qquad\leq  \epsilon + C \left[\|\mathcal{A}_{2, s}^{2} f_{\mathrm{k}}\|_{H^{t+2s,p}(\mathbb{R}^n; S^2(\mathbb{R}^n))} +\|\mathcal{A}_{2, s}^{1} f_{\mathrm{k}}\|_{H^{t+2s,p}(\mathbb{R}^n; S^2(\mathbb{R}^n))} + \|\mathcal{A}_{2, s}^{0} f_{\mathrm{k}}\|_{H^{t+2s,p}(\mathbb{R}^n; S^2(\mathbb{R}^n))}\right]\\
&\qquad \leq  \epsilon + C [\|\mathcal{A}_{2, s}^{2} f_{\mathrm{k}}-\mathcal{A}_{2, s}^{2} f\|_{H^{t+2s,p}(\mathbb{R}^n; S^2(\mathbb{R}^n))}+ \|\mathcal{A}_{2, s}^{2} f\|_{H^{t+2s,p}(\mathbb{R}^n; S^2(\mathbb{R}^n))} \\
&\qquad\qquad + \|\mathcal{A}_{2, s}^{1} f_{\mathrm{k}}-\mathcal{A}_{2, s}^{1} f\|_{H^{t+2s,p}(\mathbb{R}^n; S^2(\mathbb{R}^n))}+ \|\mathcal{A}_{2, s}^{1} f\|_{H^{t+2s,p}(\mathbb{R}^n; S^2(\mathbb{R}^n))} \\
&\qquad\qquad\qquad+ \|\mathcal{A}_{2, s}^{0} f_{\mathrm{k}}-\mathcal{A}_{2, s}^{0} f\|_{H^{t+2s,p}(\mathbb{R}^n; S^2(\mathbb{R}^n))}+\|\mathcal{A}_{2, s}^{0} f\|_{H^{t+2s,p}(\mathbb{R}^n; S^2(\mathbb{R}^n))}  ].
\end{align*}
Note that the operator norm $\|\mathcal{A}_{2, s}^{k}\|_{H^{t,p}\rightarrow H^{t+2s,p}}\leq C$, see Theorem \ref{es_bdd_A01ss_2ten_sobo:2}, where $C>0$ is constant and $k=0,1,2$,
we have
\[
\begin{split}
     \|\mathcal{A}_{2, s}^{k} f_{\mathrm{k}}-\mathcal{A}_{2, s}^{k} f\|_{H^{t+2s,p}(\mathbb{R}^n; S^2(\mathbb{R}^n))}
 &\leq \|\mathcal{A}_{2, s}^{k}\|_{H^{t,p}\rightarrow H^{t+2s,p}} \|f-f_{\mathrm{k}}\|_{H^{t,p}(\mathbb{R}^n; S^2(\mathbb{R}^n))}  \\&\leq C\epsilon, \ \forall \ \ \mathrm{k}\geq N, \ k=0,1,2.
 \end{split}
\]
Combining all the above estimates, we obtain the required result.
\end{proof}
 \begin{theorem}[Stability for fractional momentum ray transform]\label{stability_momentum_2tensorfield}
  Let $n\geq 2$ be an integer and $s \in \lb 0,1\rb \setminus  \left\{ \frac{1}{2} \right\} $. Assume $f\in \mathcal{S}(\mathbb{R}^n; S^2(\mathbb{R}^n))$ such that 
  \[
\int_{\mathbb{R}^n} \abs{\mathscr{F}^{-1}\left[(1+|\xi|^2)^{\frac{t+2s}{2}} \mathscr{F}_{\xi}(\chi_{s,2}f)(\xi,\eta)\right]}^pd\xi 
<\infty.
\]
Then there exists a constant $C>0$ such that
 \[
 \|f\|_{H^{t,p}(\mathbb{R}^n; S^2(\mathbb{R}^n))}
 \leq C \int_{\mathbb{S}^{n-1}} \left[ \int_{\mathbb{R}^n} \abs{\mathscr{F}^{-1}\left[(1+|\xi|^2)^{\frac{t+2s}{2}} \mathscr{F}_{\xi}(\chi_{s,2}f)(\xi,\eta)\right]}^pd\xi\right]^{\frac{1}{p}}dS_{\eta},
 \]    
  where $\mathscr{F}_{\xi}$ denotes the Fourier transform with respect to $\xi$ variable.
 \end{theorem}
\begin{proof}
The proof follows from Theorem \ref{stab_ave_momentum} and Theorem \ref{stability_ave_2 tensor}.
\end{proof}

\section{Weighted divergent beam ray transform: Unique continuation}
In this section, we discuss unique continuation principle for $k$-weighted divergent beam ray transform. We show that the transform admits unique continuation type result for functions (Theorem \ref{UCP:kWDBRT1}) but not for a general $m$-tensor fields (Theorem \ref{UCP:kWDBRT2}).
This distinction between fractional and $k$-weighted divergent beam ray transforms appears because of the fact that for the fractional case we could use the unique continuation principle for fractional Laplacian Lemma \ref{UCP_FL}.
 
\begin{theorem}\label{UCP:kWDBRT1}
    Let $f\in \Sc(\Rb^n)$. Let $U$ be a non-empty open subset of $\Rb^n$ and suppose we are given that 
    \[
    Df(x,\xi)=\int\limits_0^{\infty} f(x+t\xi) \D t =0 \mbox{ for every } x\in U \mbox{ and for all } \xi \in \Rb^{n}\setminus \{0\}. 
    \]
    Then $f\equiv 0$.
\end{theorem}
\begin{remark} We remark that we do not need the assumption that $f\equiv 0$ in $U$. This will follow as a consequence of the hypothesis.
\end{remark}
\bpr
We fix an $x_0\in U$. Then for $x$ in a small enough neighborhood of $x_0$, we have that $Df(x,\xi)=0$. Differentiating $Df(x,\xi)$ with respect to $x_j$ and multiplying by $\xi^{j}$, we get, 
\[
\begin{aligned}
0=\langle \xi,\n\rangle Df(x_0,\xi)=\int\limits_0^{\infty} \xi^{j}\frac{\PD f}{\PD x_j}(x_0+t\xi) \D t =\int\limits_0^{\infty} \frac{\D}{\D t} f(x_0+t\xi) \D t =-f(x_0).
\end{aligned} 
\]
Hence, we have that $f(x_0)=0$ and since $x_0$ was arbitrary, we have that $f(x)=0$ for all $x\in U$.
Next we consider 
\[
\begin{aligned}
\mathcal{A}_{0,\frac{1}{2}}^{0}f(x) &= \int\limits_{\Sb^{n-1}}Df(x,\xi) \D \xi=\int\limits_{\Sb^{n-1}}\int\limits_0^{\infty} f(x+t\xi) \D t \D \xi\\
&=\int\limits_{\Rb^{n}}f(z)\frac{1}{|x-z|^{n-1}} \D z=f * \frac{1}{|\cdot|^{n-1}}(x).
\end{aligned}
\]
We note that $\mathcal{A}_{0,\frac{1}{2}}^{0}f$ we have obtained here is exactly the normal operator of the ray transform of $f$ (up to a factor of $2$). Since $\mathcal{A}_{0,\frac{1}{2}}^{0}f(x)$ vanishes for $x\in U$ and $f$ vanishes on $U$ as well, using the unique continuation result for the ray transform of functions \cite{Ilmavirta:Monkkonen:2020}, we have that $f\equiv 0$.
\epr
\begin{remark}\label{sec4:rem2}
    A unique continuation type result, analogous to the theorem above for divergent beam ray transform of vector fields (or any other symmetric $m$-tensor field) cannot hold. This can be seen as follows for the case of vector fields. Consider a non-empty open set $U$. Construct a compactly supported smooth function $v$ with support in the complement of $U$. Let the vector field $f=dv$. Then for any ray starting from $U$, we have the divergent beam ray transform of $f$ is $0$, but $f$ is non-zero.
\end{remark}

\begin{theorem}\label{UCP:kWDBRT2}
    Let $f\in \Sc(\Rb^n; \Cb^n)$ be a vector field and suppose $Df(x,\xi)=0$ for all $x\in U$, where $U$ is a non-empty open subset of $\Rb^n$. Then $f=\D v$ for some smooth function $v$ satisfying $|v|\to 0$ as $|x|\to \infty$, and $v\equiv 0$ in $U$.
\end{theorem}
\bpr
By using Theorem \ref{tm:recon:dbrt}, we have that $f\equiv 0$ in $U$.  Hence, $\mbox{curl}(f)\equiv 0$ in $U$. 
Also, the following averaging operator gives: 
\[
\begin{aligned}
\mathcal{A}_{1,\frac{1}{2}}^{1,i}f(x) = \int\limits_{\Sb^{n-1}} \xi_{i}\int\limits_0^{\infty} f_{j}(x+t\xi)\xi^{j} \D t \D \xi=\int\limits_{\Rb^n} f_{j}(z)\frac{(x_i-z_i)(x_j-z_j)}{|x-z|^{n+1}} \D z.
\end{aligned} 
\]
This is exactly the normal operator corresponding to the ray transform (modulo a constant). Since this vanishes in $U$, using the unique continuation result already proved in \cite{Ilmavirta:Monkkonen:2021}, we have that $\mbox{curl}(f)\equiv 0$. By the Helmholtz decomposition, we have that $f\equiv \D v$ with $v$ vanishing at $\infty$. 

\noindent It remains to show that $v\equiv 0$ in $U$. This can be established by the fact that $Df(x,\xi)=0$ for $x\in U$ would give $v(x)=0$ for each $x\in U$, by the fundamental theorem of calculus. 
\epr
\begin{remark}
We note that $\mathcal{A}_{1,\frac{1}{2}}^{0}f$, the zeroth average (unlike the fractional case) does not give any additional information, since using the first average, we have already established that $f=\D v$, and hence the zeroth average is identically $0$.  
\end{remark}

\section*{Acknowledgements}
SRJ acknowledges the Flagship Program on Advanced Mathematics for Sensing, Imaging, and Modelling (decision number 359183) by the Research Council of Finland for supporting his research. Additionally, SRJ would like to thank IISER Bhopal for their support and hospitality during his visit, where part of this work was conducted. VPK acknowledges the support of the Department of Atomic Energy, Government of India, under Project No. 12-R\&D-TFR-5.01-0520.

\bibliographystyle{amsplain}

\end{document}